\documentclass[12pt,a4,final]{article}
\usepackage[utf8]{inputenc}

\pdfoutput=1

\newcommand\Tstrut{\rule{0pt}{2.9ex}}       
\newcommand\Bstrut{\rule[-1.3ex]{0pt}{0pt}}

\usepackage[left=3cm,right=3cm,top=3cm,bottom=4cm]{geometry}
\usepackage[normalem]{ulem}
\usepackage{float}
\usepackage{mathrsfs}
\usepackage{amsmath,amssymb,amsthm}
\usepackage{bm}
\usepackage{algorithm}
\usepackage{algpseudocode}
\usepackage{hyperref}
\usepackage[misc]{ifsym}
\usepackage{color}
\usepackage{graphicx}
\usepackage{subcaption}

\usepackage{amsfonts}
\usepackage{mathtools}

\usepackage{multirow}

\newcommand{\bbR}{\mathbb R}

\providecommand{\Q}{\ensuremath{\mathbb{Q}}}   
\providecommand{\Pb}{\ensuremath{\mathbb{P}}}

\DeclareMathOperator{\sgn}{sgn}
\DeclarePairedDelimiter\parens{\lparen}{\rparen}
\providecommand{\sgnp}[1]{\sgn\parens{#1}}

\providecommand{\rset}{\mathbb{R}}

\newtheorem{theorem}{Theorem}[section]
\newtheorem{lem}{Lemma}[section]

\newtheorem{prop}{Proposition}[section]
\newtheorem{cor}{Corollary}[section]

\newcounter{hypA}
\newenvironment{hypA}{\refstepcounter{hypA}\begin{itemize}
  \item[({\bf A\arabic{hypA}})]}{\end{itemize}}

\begin{document}

\title{A Wasserstein Coupled Particle Filter for Multilevel 
  Estimation\thanks{
        This work is supported by the KAUST Office of  
    Sponsored Research (OSR) under Award No. URF/1/2584-01-01 in 
    the KAUST Competitive Research Grants Program-Round 4 (CRG2015) 
    and the Alexander von Humboldt Foundation. 
    \newline
    ~
    \newline
    \Letter \hspace{1mm} Marco~Ballesio 
    $\mbox{\hspace{1.8mm}}$ \texttt{marco.ballesio@kaust.edu.sa},
    Tel.: +966-54-7974461
    \newline
    ~
        \newline
    $^a$ 
    Computer, Electrical and Mathematical Sciences and Engineering, 
    \newline
    $\mbox{\hspace{1.8mm}}$ 4700 King Abdullah University of Science and Technology (KAUST), 
    \newline
    $\mbox{\hspace{1.8mm}}$ Thuwal 23955-6900, Kingdom of Saudi Arabia.    
    \newline
    $^b$  
    KAUST SRI Center for Uncertainty Quantification in Computational
    Science and Engineering
        \newline
    $^c$ 
    Alexander von Humboldt Professor in Mathematics for Uncertainty
    Quantification, 
    RWTH Aachen University, Germany.
  }} 

\author{Marco~Ballesio$^{a,b}$~\href{https://orcid.org/0000-0001-5690-6926}{\small{ORCID}}
  \and
  Ajay~Jasra$^{a}$~\href{https://orcid.org/0000-0003-4808-9131}{\small{ORCID}}
  \and Erik~von~Schwerin$^{a,b}$~\href{https://orcid.org/0000-0002-2964-7225}{\small{ORCID}}
  \and Ra\'{u}l~Tempone$^{a,b,c}$~\href{https://orcid.org/0000-0003-1967-4446}{\small{ORCID}}
}

\markboth{Ballesio et al.}{A Wasserstein Coupled Particle Filter for Multilevel Estimation}
\maketitle
\begin{abstract}
  In this paper, we consider the filtering problem for partially observed
  diffusions, which are regularly observed at discrete times. We are
  concerned with the case when one must resort to time-discretization of
  the diffusion process if the transition density is not available in an
  appropriate form. In such cases, one must resort to advanced
  numerical algorithms such as particle filters to consistently estimate
  the filter. It is also well known that the particle filter can be
  enhanced by considering hierarchies of discretizations and the
  multilevel Monte Carlo (MLMC) method, in the sense of reducing
  the computational effort to achieve a given mean square error (MSE). A
  variety of multilevel particle filters (MLPF) have been suggested in
  the literature, e.g.,~in Jasra et al.,~\emph{SIAM J,
      Numer. Anal.}, {\bf 55}, 3068--3096. 
  Here we introduce a new
  alternative that involves a resampling step based on the optimal
  Wasserstein coupling. We prove a central limit theorem (CLT) for the
  new method. On considering the asymptotic variance, we establish that
  in some scenarios, there is a reduction, relative to
    the approach in the aforementioned paper by Jasra et al., in
  computational effort to achieve a given MSE. These findings are
  confirmed in numerical examples. We also consider filtering diffusions
  with unstable dynamics; we empirically show that in such cases
  a change of measure technique seems to be required to maintain our findings. 
\end{abstract}

\section{Introduction}

Hidden Markov models (HMMs) form a wide class of dynamical models
that are appropriate for modeling a variety of scenarios in
applications such as finance, economics, and engineering; see for
instance~\cite{cappe}. In this article, we consider the case where data
are observed at regular time intervals and the hidden Markov chain (or
signal) evolves according to a diffusion process. Our main objective
is to consider the filtering problem, which is the estimation of expectations of functions of the
signal given all the data observed up to the current time point,
recursively in time.

In most cases of practical interest, the filtering problem requires
numerical, Monte Carlo-based, approximation techniques, due to the
intractable integrals associated with the filtering distribution. In
particular, when the dimension of the hidden chain is moderate
(approximately 10 or less), an often-used and consistent approximation
uses the particle filter (PF); see, for instance,~\cite{fk}. The
PF generates a collection of $N$ samples (particles) in
parallel that evolve sequentially in time via `sampling' and then
`resampling' steps. The sampling step moves the samples according to
certain dynamics (for instance, the hidden chain) and then corrects for
the fact that the law of the samples is not the filter, by using
importance sampling. To ensure that importance sampling performs well
w.r.t.~the time parameter, the resampling step is used, which samples
with replacement from the current set of particles with probabilities
proportional to their weight. (The weights are then reset to one). The
samples are used to sequentially approximate expectations w.r.t.~the
filter. This method can provide estimates of expectations w.r.t.~the
filter with an algorithm that is of linear cost in time and with error
that (often) does not depend on the time parameter.

In the case of our HMM of interest, we assume explicitly that the
diffusion process has a transition density that is intractable, i.e., 
it is not known analytically, nor available up-to a non-negative
unbiased estimator (e.g.~\cite{fearnhead}). In such scenarios, one
often focuses on the case where filtering is performed when the
hidden dynamics have been time-discretized, for instance using the
Euler method. Then, to approximate expectations associated with the
filter, the PF can be run, when considering the `most-precise'
time discretization. It is well known, however, that this can be
improved, in the sense that there is a reduction in computational
effort to achieve a given mean square error (MSE) (for estimating expecatations w.r.t.~the
filter), by using the multilevel Monte Carlo (MLMC) method~\cite{giles,heinrich}. 

The MLMC method considers a collapsing sum of expectations associated
with a hierarchy of filters, at a given time, associated with increasingly
finer time discretizations of the diffusion process. Given access to
\emph{exact} samples of a `good' \emph{coupling} of pairs of filters at
consecutive time-discretizations, the method can achieve an
improvement as noted above. Essentially, this approach requires
couplings such that the variance of the difference of the
position of the filters are close in the sense of the
time discretization and then uses fewer samples as time
discretizations become increasingly finer (and hence more
expensive). The main issue is that exact sampling of couplings of
filters is challenging to achieve (hence the use of PFs), and this has
lead to a substantial number of contributions of MLPFs;
\cite{gregory,mlpf,mlpf_nc,sen}. 

Some of the first works on MLPFs include~\cite{gregory}
and~\cite{mlpf}. These two methods are similar in the sampling stage
of PFs but differ on the resampling stage of PFs. The approach
of~\cite{mlpf} generates pairs of particles that sequentially
approximate filters at a `fine' and `coarse' discretization. The
resampling step attempts to maximize the probability that the indices
of pairs of samples remain the same (the maximal coupling) while
ensuring the approximation is correct in the large sample
limit. \cite{mlpf} establishes the consistency of this method and
(mathematically) that there is a reduction of computational effort
relative to using the particle filter at the most accurate time
discretization, to obtain a target MSE of an estimate of filtering
expectations. We are not aware of such results in the case
of~\cite{gregory}, although it is established numerically in that
article also. The main issue with the work of~\cite{mlpf} is that
the rate of coupling, relative to the case where there is no data
(forward problem), is reduced by a factor of two. The objective of
this article is to consider a method that can provably, at least in
some scenarios, retain the rate of the forward problem. In addition,
we require that the cost of the method to be, in prinicple, linear in
the time parameter. We note that there are some solutions in this
direction. In~\cite{hous}, a procedure based on optimal transport is
derived, which \emph{experimentally} can achieve the aforementioned
objective; there is, however, no proof of this property and we
suspect due to the complexity of the numerical approximations
involved, that this is difficult to achieve. In~\cite{sen}, a mild
deviation of~\cite{mlpf} is considered, which in some contexts
(outside those considered here) appears to improve on~\cite{mlpf} in
numerical experiments, but again, we suspect that it is
  very challenging to verify this mathematically. There is also a method
in~\cite{deltapf} (see also~\cite{deltapf_jy}) that retains the
forward rate using a PF method. However, this latter method is not
useful for filtering, as the estimate is based on the path of
particles and subject to the notorious `path degeneracy' problem of
PFs (see~\cite{kantas}). 

In this article, we consider what appears to be a new resampling
mechanism in the context of MLPFs, focused on the case where the
hidden state is one dimensional. The procedure suggested in this 
article uses the optimal coupling of the resampled indices, 
in terms of squared Wasserstein distance with $L_2$ as the metric
(we call this the `Wasserstein coupling'). To motivate this, we develop an
original Feynman-Kac type interpretation of this method, which
establishes, in the large sample limit, that the resampling step
corresponds to sampling the optimal Wasserstein coupling of the
filters. We illustrate the benefit of this approach by proving a
central limit theorem for the estimation of the differences of
the predictors (and hence filters) associated with a fine and coarse
discretization. We show that the associated asymptotic variance, which is a 
proxy for the variance of a finite sample estimate (up-to a factor of
$\mathcal{O}(N^{-1})$), retains the forward rate under assumptions for a Euler
discretization, at least when the diffusion
coefficient is non-constant. The reason for this improvement, relative to~\cite{mlpf},
is that one is approximating an optimal type coupling, whereas (as noted in~\cite{jasra_yu}),
the limiting coupling of the method in~\cite{mlpf} does not have any optimality properties.
We verify this numerically in several
examples. We emphasize that the CLT is a non-trivial mathematical
result that requires several innovations, relative to the existing
CLTs for PFs (e.g.~\cite{fk}).  Some related work can be found in~\cite{jasra_yu}, which
proves CLTs for several MLPFs and also considers the approach in this article (although the CLT is not proved in~\cite{jasra_yu}).
Relative to that work, we show that the exact forward rate is maintained (as an upper-bound on the asymptotic variance) for certain
problems, whereas that is not exactly the case in~\cite{jasra_yu} (although we note that the assumptions are weaker in~\cite{jasra_yu} and the constants are uniform in
time, which is not the case in this study).

We also numerically consider the case of filtering diffusions with unstable dynamics, such as a double-well potential drift with constant diffusion
coefficient; the case of the forward problem has been considered in~\cite{MLMC_no_contractivity}, for example. In this scenario, simple PFs and MLPFs do not always
perform as well as they do in the 'standard case'. By adopting a change of measure on the time-discretized diffusion dynamics, as is used in~\cite{MLMC_no_contractivity}, our findings
can be extended to this problem as well.

This article is structured as follows. In Section~\ref{sec:model_algo}, we 
outline a general model as well as our new algorithm. 
In Section~\ref{sec:theory}, our CLT is stated and a consideration of the
asymptotic variance for an ML application is given. 
In Section~\ref{sec:numerics}, numerical experiments 
  confirming the theoretical analysis are presented.
  Proofs of the theoretical results and algorithm listings are included in 
  the appendices.

\section{Model and Algorithm}\label{sec:model_algo}

\subsection{Notations}

Let $(\mathsf{X},\mathcal{X})$ be a measurable space. 
For $\varphi:\mathsf{X}\rightarrow\mathbb{R}$, we write
$\mathcal{B}_b(\mathsf{X})$, $\mathcal{C}_b(\mathsf{X})$, and
as the collection of bounded measurable and 
continuous, bounded measurable functions, respectively. 
$\mathcal{C}^2(\mathsf{X})$ are the twice continuously differentiable real-valued functions on $\mathsf{X}$.
For $\varphi\in\mathcal{B}_b(\mathsf{X})$, we write the supremum norm
$\|\varphi\|=\sup_{x\in\mathsf{X}}|\varphi(x)|$. 
If $\mathsf{X}=\mathbb{R}$  and $\varphi:\mathsf{X}\rightarrow\mathbb{R}$,
$\varphi\in \textrm{Lip}(\mathsf{X})$ if 
there exists a finite constant $C$
for every $(x,y)\in\mathsf{X}\times\mathsf{X}$ $|\varphi(x)-\varphi(y)|\leq C|x-y|$.
$\mathscr{P}(\mathsf{X})$  denotes the collection of probability
measures on $(\mathsf{X},\mathcal{X})$. 
For a measure $\mu$ on $(\mathsf{X},\mathcal{X})$ and a
$\varphi\in\mathcal{B}_b(\mathsf{X})$, the notation
$\mu(\varphi)=\int_{\mathsf{X}}\varphi(x)\mu(dx)$ is used. 
For $(\mathsf{X}\times\mathsf{Y},\mathcal{X}\vee\mathcal{Y})$ a
measurable space and $\mu$ a non-negative measure on this space, 
we use the tensor-product of function notations for
$(\varphi,\psi)\in\mathcal{B}_b(\mathsf{X})\times\mathcal{B}_b(\mathsf{X})$, 
$\mu(\varphi\otimes\psi)=\int_{\mathsf{X}\times\mathsf{Y}}\varphi(x)\psi(y)\mu(d(x,y))$.
Let $K:\mathsf{X}\times\mathcal{X}\rightarrow[0,\infty)$ be a non-negative operator
and $\mu$ be a ($\sigma-$finite) measure; then we use the notations 
$
\mu K(dy) = \int_{\mathsf{X}}\mu(dx) K(x,dy)
$
and for $\varphi\in\mathcal{B}_b(\mathsf{X})$, 
$
K(\varphi)(x) = \int_{\mathsf{X}} \varphi(y) K(x,dy).
$
For $\mu,\nu\in\mathscr{P}(\mathsf{X})$, the total variation distance 
is written $\|\mu-\nu\|_{\textrm{tv}}=\sup_{A\in\mathcal{X}}|\mu(A)-\nu(A)|$.
For $A\in\mathcal{X}$, the indicator is written $\mathbb{I}_A(x)$.
$\mathcal{N}_q(\kappa,\Sigma)$
denotes an $q$-dimensional Gaussian distribution of mean $\kappa$ and
covariance $\Sigma$. We omit the subscript $q$ if $q=1$.

\subsection{Partially Observed Diffusion}\label{sec:diff}
We illustrate the general model to be considered in the context of partially observed
diffusions. This is simply to establish the model of principal interest in our numerical
work, but that the results to be derived extend to a more general framework. 
\cite{fk} provides a background for the material of the following subsections.

Consider the following diffusion process:
\begin{eqnarray}
dX_t & = & a(X_t)dt + b(X_t)dW_t
\label{eq:sde}
\end{eqnarray}
with $X_t\in\mathbb{R}^d$, $t\geq 0$, $x_0$ given and fixed and $\{W_t\}_{t\geq 0}$ a
standard Brownian motion of appropriate dimension. The following assumptions
will be made on the diffusion process: 
\begin{equation}
  \tag{D}\label{assumptions}
  \left.
  \parbox{\dimexpr\linewidth-4em}{
  {Denoting the $i^{th}-$ (resp.~$(i,j)^{th}$) element of  $a$ (resp.~$b$) as $a^i$ (resp.~$b^{i,j}$)}, the coefficients  $a^i, b^{i,j} \in \mathcal{C}^2(\mathbb{R}^d)$, for $i,j= 1,\ldots, d$. 
Also, $a$ and $b$ satisfy 
\begin{itemize}
\item[(i)] {\bf uniform ellipticity}: $b(x)b(x)^T$ is uniformly positive definite;
\item[(ii)] {\bf globally Lipschitz}:
there is a $C>0$ such that \newline
$|a(x)-a(y)|+|b(x)-b(y)| \leq C |x-y|$, 
for all $x,y \in \mathbb{R}^d$; 
\end{itemize}
}
\right\}
\end{equation}
(D) is assumed explicitly. 

We assume that data are regularly spaced, in discrete time, observations 
$y_1,\dots,y_{n}$, $y_k \in \bbR^m$. 
We assume that, conditional on $X_{k}$, 
$Y_k$ is independent of all other random variables with density $G(x_{k},y_k)$.
We will lag the time parameter by one
and omit the data from the notation and write $G_{k}(x_{k})$   instead of $G(x_{k+1},y_{k+1})$.
The joint probability density of the observations and the unobserved diffusion at the observation times 
is then
$$
\prod_{p=0}^n G_{p}(x_{p}) Q^\infty(x_{(p-1)},x_{p}), 
$$
where 
$Q^\infty(x_{(p-1)},x)$ 
is the transition density of the diffusion process as a function of $x$, i.e., the density of the solution
$X_1$ of~\eqref{eq:sde} at time $1$ given initial condition $X_0=x_{(p-1)}$.

In practice, in many applications, one must time discretize the
diffusion to use the model. We suppose
a Euler discretization with step size $h_l=2^{-l}$,
$l\geq 0$. Thus, in practice, we work with transition
  densities depending on the finite step size,
\begin{displaymath}
  \prod_{p=0}^n G_{p}(x_{p}) Q^{l}(x_{(p-1)},x_{p}).  
\end{displaymath}

\subsection{General Model}
Let $G_n\in\mathcal{B}_b(\mathsf{X})$, $G_n:\mathsf{X}\rightarrow\mathbb{R}_+$.
Let $\eta_0^{f},\eta_0^c\in\mathscr{P}(\mathsf{X})$ and
$\{M_n^f\}_{n\geq 1}$,   $\{M_n^c\}_{n\geq 1}$ be two sequences of
Markov kernels, i.e.~$M_n^f:\mathsf{X}\rightarrow\mathscr{P}(E)$, 
$M_n^c:\mathsf{X}\rightarrow\mathscr{P}(E)$. Define, for
$s\in\{f,c\}$, $\varphi\in\mathcal{B}_b(\mathsf{X})$, 
\begin{displaymath}
  \gamma_{n}^s(\varphi) = \int_{\mathsf{X}^{n+1}}\varphi(x_n)
  \Big(\prod_{p=0}^{n-1} G_p^s(x_p)\Big) \eta_0^s(dx_0)\prod_{p=1}^n
  M_p^s(x_{p-1},dx_p) 
\end{displaymath}
and
$$
\eta_n^s(\varphi) = \frac{\gamma_{n}^s(\varphi)}{\gamma_{n}^s(1)}.
$$
In the context of partially observed diffusions, $M_n^f$ corresponds to $Q^{l}$ ($l\geq 1$) and
$M_n^c$ corresponds to $Q^{l-1}$. 
The initial distribution, $\eta_0$, is simply the Euler kernel started
at some given $x_0$. 

The objective is to consider Monte Carlo type algorithms,
which, for $\varphi\in\mathcal{B}_b(\mathsf{X})$ 
 and recursively in $n$, will approximate quantities such as 
\begin{equation}\label{eq:pred}
  \eta_n^f(\varphi) - \eta_n^c(\varphi)
\end{equation}
or 
\begin{equation}\label{eq:filt}
  \frac{\eta_n^f(G_n\varphi)}{\eta_n^f(G_n)} -
  \frac{\eta_n^c(G_n\varphi)}{\eta_n^c(G_n)}. 
\end{equation}
For partially observed diffusions, this will correspond to the computation of
differences of expectations of predictors and filters. There are a
variety of reasons for why the differences are explicitly of
interest, which we explain below.

We would like to approximate couplings of $(\eta_n^f,\eta_n^c)$, say
$\check{\eta}_n\in\mathscr{P}(\mathsf{X}\times\mathsf{X})$, i.e.,~that
for any $A\in\mathcal{X}$ and every $n\geq 0$ 
\begin{displaymath}
  \check{\eta}_n(A\times\mathsf{X}) = \eta_n^f(A) \qquad
  \check{\eta}_n(\mathsf{X}\times A) = \eta_n^c(A) 
\end{displaymath}
and consider approximating 
\begin{displaymath}
  \check{\eta}_n(\varphi\otimes 1) - \check{\eta}_n(1\otimes \varphi)
\end{displaymath}
or 
\begin{displaymath}
  \frac{\check{\eta}_n((G_n\varphi)\otimes 1)}
  {\check{\eta}_n(G_n\otimes 1)} 
  - \frac{\check{\eta}_n(1\otimes(G_n\varphi))}
  {\check{\eta}_n(1\otimes G_n)}.
\end{displaymath}
Throughout the article, we assume that there exists (there is always at least one, the independent coupling)
$\check{\eta}_0\in\mathscr{P}(\mathsf{X}\times\mathsf{X})$ such that 
for any $A\in\mathcal{X}$
\begin{displaymath}
  \check{\eta}_0(A\times \mathsf{X}) = 
  \eta_0^f(A) \qquad \check{\eta}_0(\mathsf{X}\times A) = \eta_0^c(A)
\end{displaymath}
and moreover for any $n\geq 1$, there exists Markov kernels
$\{\check{M}_n\}$,
$\check{M}_n:\mathsf{X}\times\mathsf{X}\rightarrow\mathscr{P}(\mathsf{X}\times\mathsf{X})$ 
such that for any $A\in\mathcal{X}$, $(x,x')\in\mathsf{X}\times\mathsf{X}$:
\begin{displaymath}
  \check{M}_n(A\times \mathsf{X})(x,x') = 
  M_n^f(A)(x) \qquad \check{M}_n(\mathsf{X}\times A)(x,x') = M_n^c(A)(x').
\end{displaymath}
In the case of partially observed diffusions, a natural and non-trivial coupling of $M_n^f$ and $M_n^c$ (and
hence of $\eta_0$)  exists, denoted $\check{Q}^l$; see e.g.~\cite{mlpf}. 

\subsection{Illustration for Partially Observed Diffusions}\label{sec:pod_ex}
For partially observed diffusions, we consider why  quantities such as $\eta_n^f(\varphi) - \eta_n^c(\varphi)$ are of interest.
 In particular, we explore their role in multilevel schemes
  that can significantly reduce the cost of particle filters where
  the HMM is numerically approximated by a time-stepping scheme. 
It is well-known in the literature~\cite{giles,mlpf} that if exact sampling from $\eta_n^L$ was possible for any $L
\geq 0$,
then the use of the multilevel identity, for $\varphi\in\mathcal{B}_b(\mathsf{X})$,
\begin{equation}\label{eq:ml_id}
\eta_n^L(\varphi) = \eta_n^0(\varphi) + \sum_{l=1}^L[\eta_n^l-\eta_n^{l-1}](\varphi)
\end{equation}
can (in some cases) greatly reduce the computational cost to achieve a given MSE, compared to Monte Carlo estimation using samples from $\eta_n^L$.
The main key to this cost reduction is sampling from 'good' couplings of $(\eta_n^l,\eta_{n}^{l-1})$, as we explain below in a similar way to \cite[Section 2.3.1.]{jasra_yu}. 

We will assume that $\mathsf{X}=\mathbb{R}$ and that $\varphi\in\textrm{Lip}(\mathsf{X})\cap\mathcal{B}_b(\mathsf{X})$.
Suppose that one can sample from $\eta_n^L$ exactly and the cost to do this is $\mathcal{O}(h_L^{-1})$; we will, throughout this exposition, ignore the cost in terms of the time
parameter $n$.
If one obtains exact samples $x_n^{1,L},\dots,x_n^{N,L}$ from $\eta_n^L$, then an unbiased and consistent estimator of $\eta_n^L(\varphi)$ is $\tfrac{1}{N}\sum_{i=1}^N\varphi(x_n^{i,L})$ and the mean square error (relative to the expectation of $\varphi$ w.r.t.~the predictor with no discretization error) is at most, under assumptions,
$$
C\Big(\frac{1}{N} + h_L^2\Big)
$$
for some finite constant $C$ 
that does not depend upon $N$ nor $L$. Taking an arbitrary $\epsilon >0$, one can make the MSE $\mathcal{O}(\epsilon^2)$
by choosing $L=\mathcal{O}(|\log(\epsilon)|)$ and $N=\mathcal{O}(\epsilon^{-2})$; 
then the MSE is $\mathcal{O}(\epsilon^2)$ and the cost to achieve
this is $\mathcal{O}(\epsilon^{-3})$. 

Now suppose that one can obtain exact samples from some coupling $\check{\eta}_n$ of $(\eta_n^l,\eta_{n}^{l-1})$ (of cost $\mathcal{O}(h_l^{-1})$)
and one seeks to approximate the R.H.S.~of \eqref{eq:ml_id}. The first term of the R.H.S.~of \eqref{eq:ml_id} can be approximated by using the same
procedure as for $\eta_n^L$, with $L=0$ (using say $N_0$ samples). For the summands in \eqref{eq:ml_id}, one can consider an approach that is independent of each other and of that for $\eta_n^0(\varphi)$. One generates $N_l$ samples independently from $\check{\eta}_n$, that is, $u_n^{1},\dots,u_n^{N_l}$, where $u_n=(x_n^l,x_n^{l-1})\in\mathsf{X}\times\mathsf{X}$. An unbiased and consistent estimate of $[\eta_n^l-\eta_n^{l-1}](\varphi)$ is then $\tfrac{1}{N_l}\sum_{i=1}^{N_l} [\varphi(x_n^{i,l})-\varphi(x_n^{i,l-1})]$.
Now, the MSE of the approach detailed (as for the Monte Carlo estimator above) is at most
$$
C\Big(\frac{1}{N_0} + \sum_{l=1}^L\frac{\check{\eta}_n([\varphi\otimes1 - 1\otimes \varphi]^2)}{N_l} + h_L^2\Big)
$$
where $C$ does not depend upon $N_0,\dots,N_L$ or $L$.
Now, to make the terms $\check{\eta}_n([\varphi\otimes1 - 1\otimes \varphi]^2)$ `small', one has
$$
\check{\eta}_n([\varphi\otimes1 - 1\otimes \varphi]^2) \leq C\int_{\mathsf{X}\times\mathsf{X}}(x_n^l-x_n^{l-1})^2\check{\eta}_n\Big((d(x_n^l,x_n^{l-1})\Big)
$$
where $C$ is the Lipschitz constant of $\varphi$.
The coupling that minimizes the R.H.S.~of the above displayed equation is exactly the Wasserstein coupling. If, for instance, this meant
$$
\int_{\mathsf{X}\times\mathsf{X}}(x_n^l-x_n^{l-1})^2\check{\eta}_n\Big((d(x_n^l,x_n^{l-1})\Big) \leq C h_l
$$
where $C$ does not depend upon $l$ then the MSE of our approach is upper-bounded by
$$
C\Big(\sum_{l=0}^L\frac{h_l}{N_l} + h_L^2\Big).
$$
\cite{giles}, for example, shows that setting $L=\mathcal{O}(|\log(\epsilon)|)$ and $N_l=\mathcal{O}(\epsilon^{-2}h_l L)$ yields a MSE of $\mathcal{O}(\epsilon^2)$ for a cost of $\mathcal{O}(\epsilon^{-2}\log(\epsilon)^2)$ - reducing the cost over i.i.d.~sampling from $\eta_n^L$.
The issue is that general exact sampling from $\eta_n^L$, or couplings of $(\eta_n^l,\eta_{n}^{l-1})$, is not currently possible; 
thus we will focus on PFs and MLPFs below.

\subsection{Algorithm}\label{sec:algo}
We restrict our attention to the case that $\mathsf{X}=\mathbb{R}$. 
We explicitly assume that the cumulative distribution
function (CDF) associated with the
probability, and its generalized inverse, for $s\in\{f,c\}$, $n\geq 0$  
$$
\overline{\eta}_n^s(\varphi) = \frac{\eta_n^s(G_n\varphi)}{\eta_n^s(G_n)}
$$
exist and are continuous functions. We denote the
CDF (resp.~generalized inverse) of
$\overline{\eta}_n^s$ as 
$F_{\overline{\eta}_n^s}$ (resp.~$F_{\overline{\eta}_n^s}^{-1}$ ). 
In general, we write probability measures on $\mathsf{X}$ for which the CDF and generalized inverse are well-defined as $\mathscr{P}_F(\mathsf{X})$ with the 
associated CDF $F_{\mu}$.
 Let $n\geq 1$, $\varphi\in\mathcal{B}_b(\mathsf{X}\times\mathsf{X})$ and $\mu,\nu\in\mathscr{P}_F(\mathsf{X})$ and define the probability measure:
$$
\check{\Phi}_n^W(\mu,\nu)(\varphi)  = \int_{\mathsf{X}\times\mathsf{X}}  \Big(\int_{0}^1 \delta_{\{F_{\mu}^{-1}(w),F_{\nu}^{-1}(w)\}}(du) dw \Big) \check{M}_n(\varphi)(u).
$$
We remark that the probability measure, assuming that it is well-defined,
$$
\int_{0}^1 \delta_{\{F_{\overline{\eta}_{p-1}^f}^{-1}(w),F_{\overline{\eta}_{p-1}^c}^{-1}(w)\}}(du) dw
$$
is the optimal $L_2-$Wasserstein coupling of $(\overline{\eta}_{p-1}^f,\overline{\eta}_{p-1}^c)$.
Consider the joint probability measure on $(\mathsf{X}\times\mathsf{X})^{n+1}$
$$
\mathbb{P}(d(u_0,\dots,u_n)) = \check{\eta}_0^W(du_0) \prod_{p=1}^n \check{\Phi}_p^W(\overline{\eta}_{p-1}^f,\overline{\eta}_{p-1}^c)(du_p)
$$
where  $u_p=(x_p^f,x_p^c)\in\mathsf{X}\times\mathsf{X}$. Note that here, $\check{\eta}_0^W$ is just a given coupling of $\eta_0^f$ and $\eta_0^c$. 
For instance, in the diffusion case of Section~\ref{sec:diff}, it is the coupling of pairs of Euler discretizations. For $n\geq1$, $\varphi\in\mathcal{B}_b(\mathsf{X}\times\mathsf{X})$:
$$
\check{\eta}_n^W(\varphi) = \check{\Phi}_n^W(\overline{\eta}_n^f,\overline{\eta}_n^c)(\varphi).
$$
We can easily check that for any $\varphi\in\mathcal{B}_b(\mathsf{X})$
$$
\check{\eta}_n^W(\varphi\otimes 1) = \eta_n^f(\varphi) \quad\textrm{and}\quad \check{\eta}_n^W(1\otimes\varphi) = \eta_n^c(\varphi).
$$

The algorithm (an MLPF) used here is:
$$
\mathbb{P}(d(u_0^{1:N},\dots,u_n^{1:N})) = \Big(\prod_{i=1}^N \check{\eta}_0(du_0^i)\Big)\Big(\prod_{p=1}^n\prod_{i=1}^N\check{\Phi}_p^W(\overline{\eta}_{p-1}^{N,f},\overline{\eta}_{p-1}^{N,c})(du_p^i)\Big)
$$
where for $p\geq 1$, $s\in\{f,c\}$
$$
\overline{\eta}_{p-1}^{N,s}(dx) = \sum_{i=1}^N \frac{G_{p-1}(x_{p-1}^{i,s})}{\sum_{j=1}^N G_{p-1}(x_{p-1}^{j,s})}\delta_{x_{p-1}^{i,s}}(dx).
$$
For $p\geq 1$, $s\in\{f,c\}$
$$
\eta_{p-1}^{N,s}(dx) = \frac{1}{N}\sum_{i=1}^N \delta_{x_{p-1}^{i,s}}(dx).
$$
Set
$$
\check{\eta}_{p}^{N,W}(du) = \frac{1}{N}\sum_{i=1}^N \delta_{u_{p}^{i}}(du).
$$

\subsubsection{Partially Observed Diffusions}
To illustrate the algorithm in a slightly more practical format, we describe the procedure in terms of the model in Section~\ref{sec:diff}.
In Algorithm~\ref{alg:phi_samp}, we first explain how we can sample from the operator $\check{\Phi}_p^W(\overline{\eta}_{p-1}^{N,f},\overline{\eta}_{p-1}^{N,c})(\cdot)$.
Several remarks are required. Firstly, the sorting step in 1.~need only be done once at each time step, not for each sample; in the worst case, this is of cost $\mathcal{O}(N\log(N))$. Secondly, in 2.~one
can maintain the order of the samples, by using the method in, for instance, \cite[pp.~96]{ripley}. The MLPF is then summarized in Algorithm \ref{alg:mlpf_samp}. 

To estimate
$[\eta_p^l-\eta_p^{l-1}](\varphi)$ (differences of the predictor), one has the estimator
$$
\frac{1}{N}\sum_{i=1}^N[\varphi(x_p^{i,l})-\varphi(x_p^{i,l-1})]
$$
which can be written concisely as $[\eta_{p}^{N,l}-\eta_{p}^{N,l-1}](\varphi)$ or equivalently as $\check{\eta}_{p}^{N,W}(\varphi\otimes 1 - 1\otimes \varphi)$.
To estimate
$$
\frac{\eta_p^l(G_p\varphi)}
  {\eta_p^l(G_p)} 
  - \frac{\eta_p^{l-1}(G_p\varphi)}
  {\eta_p^{l-1}(G_p)}
$$
that is, the differences in the filters, one can use
$$
\frac{\sum_{i=1}^N G_p(x_p^{i,l})\varphi(x_p^{i,l})}
  {\sum_{i=1}^N G_p(x_p^{i,l})} 
  - \frac{\sum_{i=1}^N G_p(x_p^{i,l-1})\varphi(x_p^{i,l-1})}
  {\sum_{i=1}^N G_p(x_p^{i,l-1})}
$$
which can be written $\eta_p^{N,l}(G_p\varphi)/\eta_p^{N,l}(G_p)-\eta_p^{N,l-1}(G_p\varphi)/\eta_p^{N,l-1}(G_p)$ or equivalently as
$\check{\eta}_p^{N,W}((G_p\varphi)\otimes 1)/\check{\eta}_p^{N,W}(G_p\otimes 1) - 
\check{\eta}_p^{N,W}(1\otimes (G_p\varphi))/\check{\eta}_p^{N,W}(1\otimes G_p)$.
\begin{algorithm}
\begin{enumerate}
\item{Sort $(x_{p-1}^{1,s},\dots,x_{p-1}^{N,s})$, denote as $(x_{p-1}^{(1),s},\dots,x_{p-1}^{(N),s})$ for $s\in\{l,l-1\}$ ($x_{p-1}^{(1),s}\leq\cdots\leq x_{p-1}^{(N),s}$).}
\item{Generate $U\sim\mathcal{U}_{[0,1]}$. Set 
\begin{eqnarray*}
x_{p-1}^s & = & \min\Bigg\{x\in\{x_{p-1}^{(1),s},\dots,x_{p-1}^{(N),s}\}:u<\sum_{k=1}^j \frac{G_{p-1}(x_{p-1}^{(k),s})}{\sum_{i=1}^N G_{p-1}(x_{p-1}^{(i),s})}\cap \\ & & u\geq \sum_{k=1}^{j-1} \frac{G_{p-1}(x_{p-1}^{(k),s})}{\sum_{i=1}^N G_{p-1}(x_{p-1}^{(i),s})}\Bigg\}
\end{eqnarray*}
for $s\in\{l,l-1\}$.
}
\item{Sample $u_p$ from $\check{Q}^l\Big((x_{p-1}^l,x_{p-1}^{l-1}),\cdot\Big)$.}
\end{enumerate}
\caption{Sampling from $\check{\Phi}_p^W(\overline{\eta}_{p-1}^{N,l},\overline{\eta}_{p-1}^{N,l-1})(\cdot)$ with $p\geq 1$ and for Partially Observed Diffusions.}
\label{alg:phi_samp}
\end{algorithm}
\begin{algorithm}
\begin{enumerate}
\item{For $i\in\{1,\dots,N\}$ sample $u_0^{i}$ from $\check{Q}^l\Big((x_{0},x_{0}),\cdot\Big)$. Set $p=1$.}
\item{For $i\in\{1,\dots,N\}$ sample $u_p^{i}$ from $\check{\Phi}_p^W(\overline{\eta}_{p-1}^{N,l},\overline{\eta}_{p-1}^{N,l-1})(\cdot)$. Set $p=p+1$ and return to the start of 2..}
\end{enumerate}
\caption{MLPF for Partially Observed Diffusions.}
\label{alg:mlpf_samp}
\end{algorithm}

\section{Theoretical Results}\label{sec:theory}

\subsection{Central Limit Theorem}
We give a CLT for the predictors, which can be used to prove a CLT for the filter, 
as detailed below. The proof is presented in Appendix~\ref{app:clt} where the operators in the asymptotic
variance are also explained. Below $ \Rightarrow$ is used to denote convergence in distribution as $N\rightarrow\infty$. 
\begin{theorem}\label{theo:clt_was}
Suppose that $\check{M}_n$, $M_n^s$, $s\in\{f,c\}$  are Feller for every $n\geq 1$ and $G_n\in\mathcal{C}_b(\mathsf{X})$ for every $n\geq 0$.
Then for any $n\geq 0$, $(\varphi,\psi)\in\mathcal{C}_b(\mathsf{X})\times\mathcal{C}_b(\mathsf{X})$
$$
\sqrt{N}[\check{\eta}_{n}^{N,W}-\check{\eta}_{n}^{W}](\varphi\otimes 1 - 1 \otimes \psi) \Rightarrow \mathcal{N}(0,\sigma_n^{2,W}(\varphi,\psi))
$$
where
\begin{equation}\label{eq:av_clt}
\sigma_n^{2,W}(\varphi,\psi) = \sum_{p=0}^n \check{\eta}_p^W([D_{p,n}^f(\varphi)\otimes 1- 1\otimes D_{p,n}^c(\psi)]^2).
\end{equation}
\end{theorem}
Below for $1\leq q<+\infty$, $\Sigma_n^{W}(\varphi_{1:q},\psi_{1:q})$ is an $q\times q$ real symmetric matrix. For $1\leq i,j\leq q$,
we denote the $(i,j)^{th}-$element of $\Sigma_n^{W}(\varphi_{1:q},\psi_{1:q})$ as $\Sigma_{n,(ij)}^{W}(\varphi_{1:q},\psi_{1:q})$.
\begin{cor}\label{cor:mv_clt}
Suppose that $\check{M}_n$, $M_n^s$, $s\in\{f,c\}$  are Feller for every $n\geq 1$ and $G_n\in\mathcal{C}_b(\mathsf{X})$ for every $n\geq 0$.
Then for any $n\geq 0$, $1\leq q <+\infty$, $(\varphi_1,\dots,\varphi_q,\psi_1,\dots,\psi_q)\in\mathcal{C}_b(\mathsf{X})^{2q}$
$$ 
\sqrt{N}([\check{\eta}_{n}^{N,W}-\check{\eta}_{n}^{W}](\varphi_1\otimes 1 - 1 \otimes \psi_1),\dots, 
[\check{\eta}_{n}^{N,W}-\check{\eta}_{n}^{W}](\varphi_q\otimes 1 - 1 \otimes \psi_q) )\Rightarrow \mathcal{N}_s(0,\Sigma_n^{W}(\varphi_{1:q},\psi_{1:q}))
$$
where for $1\leq i,j\leq q$
$$
\Sigma_{n,(ij)}^{W}(\varphi_{1:q},\psi_{1:q}) = \sum_{p=0}^n \check{\eta}_p^W([D_{p,n}^f(\varphi_i)\otimes 1- 1\otimes D_{p,n}^c(\psi_i)]
[D_{p,n}^f(\varphi_j)\otimes 1- 1\otimes D_{p,n}^c(\psi_j)]).
$$
\end{cor}

\begin{proof}
Follows easily by the Cramer-Wold device and the proof is omitted.
\end{proof}
We note that
$$
\sqrt{N}\Big\{\frac{\eta_n^{N,f}(G_n\varphi)}{\eta_n^{N,f}(G_n)} - \frac{\eta_n^{N,c}(G_n\varphi)}{\eta_n^{N,c}(G_n)}-
\Big(\frac{\eta_n^{f}(G_n\varphi)}{\eta_n^{f}(G_n)} - \frac{\eta_n^{c}(G_n\varphi)}{\eta_n^{c}(G_n)}\Big)
\Big\} = 
$$
$$
\sqrt{N}\Big\{\frac{1}{\eta_n^{N,f}(G_n)}[\check{\eta}_n^{N,W}-\check{\eta}_n^{W}]((G_n \varphi)\otimes 1 - 1\otimes (G_n \varphi))
- \frac{\check{\eta}_n^{W}((G_n \varphi)\otimes 1 - 1\otimes (G_n \varphi))}{\eta_n^{f}(G_n)\eta_n^{N,f}(G_n)}\times
$$
$$
[\check{\eta}_n^{N,W}-\check{\eta}_n^{W}](G_n\otimes 1 - 1\otimes1) - \Big[
\frac{\eta_n^{N,c}(G_n\varphi)}{\eta_n^{N,c}(G_n)\eta_n^{N,f}(G_n)}
[\check{\eta}_n^{N,W}-\check{\eta}_n^{W}](G_n\otimes 1 - 1\otimes G_n) +
$$
$$
\frac{\check{\eta}_n^{W}(G_n\otimes 1 - 1\otimes G_n)}{\eta_n^{N,c}(G_n)\eta_n^{N,f}(G_n)}
[\check{\eta}_n^{N,W}-\check{\eta}_n^{W}](1\otimes (G_n\varphi) - 1\otimes 1) -
\frac{\check{\eta}_n^{W}(G_n\otimes 1 - 1\otimes G_n)\eta_n^{c}(G_n\varphi)}{\eta_n^{N,c}(G_n)\eta_n^{N,f}(G_n)\eta_n^{c}(G_n)\eta_n^{f}(G_n)}
\times
$$
$$
\Big[
\eta_n^{f}(G_n)[\check{\eta}_n^{N,W}-\check{\eta}_n^{W}](1\otimes G_n - 1\otimes 1) +
\eta_n^{c}(G_n)[\check{\eta}_n^{N,W}-\check{\eta}_n^{W}](G_n\otimes 1 - 1\otimes 1)
\Big]
\Big]\Big\}.
$$
Hence, to prove a CLT for the filter (and indeed a multivariate CLT by the Cramer-Wold device) we can use Slutsky along with 
Corollary~\ref{cor:mv_clt} and the delta method. As the resulting asymptotic variance term is rather complicated, 
we do not present this result for the sake of brevity. We will focus on 
considering the asymptotic variance in Theorem~\ref{theo:clt_was} in the sequel.

\subsection{Asymptotic Variance and Multilevel Considerations}

\subsubsection{Asymptotic Variance}

We consider bounding~\eqref{eq:av_clt}. The result will allow us to consider the 
utility of the algorithm in the context of multilevel applications,
such as in Section~\ref{sec:diff}. 

We make the following assumptions.
\begin{hypA}\label{hyp:1}
For every $n\geq 0$, $G_n\in \textrm{Lip}(\mathsf{X})\cap\mathcal{B}_b(\mathsf{X})$.
\end{hypA}
\begin{hypA}\label{hyp:2}
For every $n\geq 1$, $\varphi\in \textrm{Lip}(\mathsf{X})\cap\mathcal{B}_b(\mathsf{X})$
there exists a $C<+\infty$ such that for $s\in\{f,c\}$, we have for every $(x,y)\in\mathsf{X}\times\mathsf{X}$
$$
|M_n^s(\varphi)(x)-M_n^s(\varphi)(y)|\leq C|x-y|.
$$
\end{hypA}
\begin{hypA}\label{hyp:3}
For every $n\geq 0$, there exists a $C>0$ such that $\inf_{x\in\mathsf{X}}G_n(x) \geq C$.
\end{hypA}

Now define
$$
|\|M_n^f-M_n^c\|| := \sup_{\mathcal{A}}\sup_{x\in\mathsf{X}}|M_n^f(\varphi)(x)-M_n^c(\varphi)(x)| \
$$
where $\mathcal{A}=\{\varphi\in\mathcal{B}_b(\mathsf{X})\cap\textrm{Lip}(\mathsf{X}): \|\varphi\|\leq 1|\}$.
Set
\begin{eqnarray*}
|\|M_n^{f,c}\|| & = & \max\{|\|M_1^f-M_1^c\||,\dots,|\|M_n^f-M_n^c\||\} \\
\|\eta_n^{f,c}\| & = &  \max\{\|\eta_0^f-\eta_0^c\|_{\textrm{tv}},\dots,\|\eta_n^f-\eta_n^c\|_{\textrm{tv}}\} \\
\|\check{\eta}_n^{W}\| & = & \max\{\int_{\mathsf{X}^2}\check{\eta}_0^W(d(x^f,x^c))(x^f-x^c)^2,\dots,
\int_{\mathsf{X}^2}\check{\eta}_n^W(d(x^f,x^c))(x^f-x^c)^2
\}.
\end{eqnarray*}
Note that we are assuming that $\|\check{\eta}_n^{W}\|$ is finite.
\begin{prop}\label{prop:av}
Assume (A\ref{hyp:1}-\ref{hyp:3}). Then for any $n\geq 0$, $\varphi\in\mathcal{B}_b(\mathsf{X})$ there exists a $C<\infty$
such that
$$
\sigma_n^{2,W}(\varphi,\varphi) \leq C(|\|M_n^{f,c}\||^2 + \|\eta_n^{f,c}\|^2 + \|\check{\eta}_n^{W}\|)
$$
where $\sigma_n^{2,W}(\varphi,\varphi)$ is as in \eqref{eq:av_clt}.
\end{prop}
\begin{proof}
Follows by Lemma \ref{lem:av4} in Appendix~\ref{app:av_prfs}.
\end{proof}

\subsubsection{Partially Observed Diffusions}
Now in the context of Section~\ref{sec:diff} (see also Section~\ref{sec:pod_ex}), 
we consider the case where 
$f$ uses the diffusion with discretization $h_l$ and $c$ uses the diffusion with 
discretization $h_{l-1}$. Our objective is to utilize Proposition~\ref{prop:av} 
to deduce a complexity type theorem for an MLPF.

We have (under (D)) by~\cite[equation (2.4)]{del}
$$
|\|M_n^{f,c}\|| \leq C h_l
$$
and by \cite[Lemma D.2]{mlpf}
$$
\|\eta_n^{f,c}\| \leq C h_l
$$
where in both cases the constant $C$ can depend on $n$.  For any $n\geq 1$, we have by \cite[Proposition D.1]{mlpf}
$$
\int_{\mathsf{X}\times\mathsf{X}}(x_n^f-x_n^c)^2\check{\eta}_n^W\Big(d(x^f,x^c)\Big) \leq C(h_l+\int_{0}^1 (F_{\overline{\eta}_{n-1}^f}^{-1}(w)-F_{\overline{\eta}_{n-1}^c}^{-1}(w))^2 dw).
$$
We remark that in the case of $n=0$ the upper-bound is $Ch_l$.
Now, for ease of exposition, we shall assume that
$\mathsf{X}$ is compact. This is unrealistic in the setting of Section~\ref{sec:diff}, 
but can be relaxed. Then one has, by standard results on Wasserstein distances, that 
$$
\int_{\mathsf{X}\times \mathsf{X}}(x_n^f-x_n^c)^2\check{\eta}_n^W\Big(d(x_n^f,x_n^c)\Big) \leq C(h_l + \|\overline{\eta}_{n-1}^f-\overline{\eta}_{n-1}^c\|_{\textrm{tv}}).
$$
Using a simple adaptation of~\cite[Lemma D.2]{mlpf} and recalling (A\ref{hyp:1}-\ref{hyp:3}), one has
$$
\|\overline{\eta}_{n-1}^f-\overline{\eta}_{n-1}^c\|_{\textrm{tv}} \leq Ch_l.
$$
Hence we can deduce that there exists a constant $C$ that can depend on $n$ such that
$$
\|\check{\eta}_n^{W}\| \leq Ch_l.
$$
Thus we have shown that
\begin{equation}\label{eq:av_bound_pod}
\sigma_n^{2,W}(\varphi,\varphi) \leq C h_l.
\end{equation}

The significance of~\eqref{eq:av_bound_pod} in a multilevel context is as follows. 
Following the strategy described in Section~\ref{sec:pod_ex}, if one now runs $L$ MLPFs (with $N_l$ samples) as in Section~\ref{sec:algo}
with $1\leq l \leq L$, where $f$ uses the diffusion with discretization $h_l$ and $c$ uses the diffusion with discretization $h_{l-1}$.
One also runs a particle filter (with $N_0$ particles) targeting the predictor with discretization $h_0$ (see e.g.~\cite{mlpf}). Then, as a result
of our analysis, one expects that the \emph{finite} sample variance to be upper-bounded by
$$
C\sum_{l=0}^L\frac{h_l}{N_l}.
$$
This rate (exponent of $h_l$), in the case that the diffusion coefficient is non-constant, 
(in~\eqref{eq:sde}) retains the so-called `forward' rate, i.e., the variance one 
would obtain when there
is no data and one is exactly sampling the couplings of the euler discretizations. 
This is an improvement over the method in~\cite{mlpf} and appears to be one
of the few cases in the literature that (mathematically) verifies that the forward 
rate can be maintained for filtering, in theory. We can check, using the particle allocations
in~\cite{giles} and~\cite{mlpf}, for example, that even with an $\mathcal{O}(N_l\log(N_l))$ 
cost of resampling for the Wasserstein coupling (worst case) we still improve over the method 
of~\cite{mlpf} in terms of reducing the cost for a given MSE. This will be further explored in our numerical simulations.

We note that all of the constants in Proposition~\ref{prop:av} are time dependent. 
One can deal with the time behavior with considerable complication
of the analysis in  Appendix~\ref{app:av_prfs} (see \cite{jasra_yu} for instance). 

\section{Numerical Experiments}
\label{sec:numerics}

In this section, we numerically compare the complexity of the MLPFs associated with 
Algorithms~\ref{alg:coupled} and~\ref{alg:CDFcoupling}, in Appendix~\ref{app:alg}. 

We consider three partially observed diffusions. The exact dynamics of the diffusion 
processes are approximated by the Euler-Maruyama numerical scheme between observation 
times. Euler-Maruyama is a first order method to approximate SDEs dynamics, and thus the 
weak convergence has rate one and the strong convergence has rate one half, if the 
diffusion has a non-constant diffusion term, assuming sufficient regularity. In 
the case of constant diffusion coefficient, the strong rate of convergence is increased to one. 
The strong convergence rate is half the convergence rate of the variance of the 
increments between refinement levels, and the weak convergence corresponds to the bias.

The three diffusion processes are as follows.
The first is the Ornstein-Uhlenbeck process with constant diffusion considered 
in~\cite{mlpf}; the constant diffusion coefficient leads to twice the typical rate 
of strong convergence for the Euler-Maruyama. We then study a slightly modified 
version of the stochastic dynamics with a nonlinear diffusion term in~\cite{mlpf}, 
which has the usual rate of strong convergence. Finally, a double-well with constant 
diffusion is analyzed. The last example is different from the first two since the 
double-well does not satisfy the contractivity condition~\eqref{eq:contractivity} 
in Appendix~\ref{app:changeofmeasure} because the potential is composed of two wells. 
We establish that, as the process can jump between wells, the coupling in 
Algorithms~\ref{alg:coupled} and~\ref{alg:CDFcoupling} can be affected by the 
bistability of the process causing the possible split in different wells of the 
coarse and fine particles in the couple. The consequence is a degradation of the 
variance convergence, which significantly reduces the performance of the MLPF. 
This issue is relevant for many practical applications of MLPFs where the underlying 
SDEs do not possess the contractivity condition. To deal with this problem and 
overcome the variance degradation, we use a version of a change of measure 
described in~\cite{MLMC_no_contractivity} and summarized in Appendix~\ref{app:changeofmeasure}.

We constrain our numerical test to one dimensional dynamics. While Algorithm~\ref{alg:coupled} 
works in multiple dimensions, the extension of Algorithm~\ref{alg:CDFcoupling} 
to more than one dimension is the topic of upcoming work. A benefit of studying 
the one-dimensional case is that we can compare the MLPF results with more accurate 
reference solutions, which are harder to compute in higher dimensions.

For each model, we estimate the parameters that consist of the convergence rates 
of the variance, the bias, and the computational complexity. Then we use the 
estimated parameters to compute optimal levels $(l^{*}, \dots, L^{*})$ and 
optimal number of particles $N_{l^{*}:L^{*}}$ to build MLPFs to satisfy a sequence 
of decreasing tolerances $\epsilon_{0:k}$. We wish to study the observed errors 
and computational times for the MLPFs associated with Algorithms~\ref{alg:coupled} 
and~\ref{alg:CDFcoupling}. 

\subsection{Model settings} 

The general diffusion process we consider is
\begin{eqnarray}   
\mathrm{d}X_{t}&=&a(X_{t})\mathrm{d}t+b(X_{t})\mathrm{d}W_{t}, \qquad 0\leq t\leq T,   \label{eq:model_setting}
\end{eqnarray}
and $X_{0}\sim\mu(\cdot)$ with $\mu(\cdot)$ the initial distribution. Note that 
the initial distribution considered here is
more general than that of our analysis in previous sections.
Here, $\lbrace W_{t}\rbrace_{t\in [0,T]}$ is a Brownian motion and $X_{t}\in\mathbb{R}$.
In all three numerical test cases, the test function is $\varphi(x)=x$ and final 
time $T=500$, with observations equally spaced at $\delta=0.5$. 
We observe $Y_{n}$ at time $n \delta$, so our data are $(y_{1}, \dots, y_{D})$, with $D=\delta^{-1}T=1000$. 
We assume $Y_{n}|X_{n\delta}\sim \mathcal{N}(X_{n\delta},\tau^{2})$ with $\tau$ 
specified for each case. The goal is the estimation of the expected value of the 
filtering distribution $\mathbb{E}[\varphi(X_{D\delta})|y_{1:D}]$. Details of each example are described below.

\subparagraph{Ornstein-Uhlenbeck (OU)} In this example, from~\cite{mlpf}, 
\begin{eqnarray}
\mathrm{d}X_{t}&=&-\theta X_{t}\mathrm{d}t+\sigma\mathrm{d}W_{t}, \qquad 0\leq t\leq T,  \nonumber \label{eq:OU}
\end{eqnarray}
with $X_{0}=0$, $\theta=1$, $\sigma=0.5$, and $\tau^{2}=0.2$. 

\subparagraph{SDE with a nonlinear diffusion term (NDT)} In this slightly 
modified version of an example from~\cite{mlpf}, 
\begin{eqnarray}
\mathrm{d}X_{t}&=&-\theta X_{t}\mathrm{d}t+\frac{\sigma}{\sqrt{1+X^{2}_{t}}}\mathrm{d}W_{t}, \qquad 0\leq t\leq T,  \nonumber \\
X_{0}&\sim&\mathcal{N}(0,\tau^{2}),  \nonumber
\end{eqnarray}
with $\theta=1$, $\sigma=1$, and $\tau^{2}=0.1$. It differs from the example in~\cite{mlpf} in the distribution of the observation noise and in the initial distribution.

\subparagraph{Double-Well Constant Diffusion} Here,
\begin{eqnarray}
    \label{eq:DW_dynamics}
   dX_t  &=& a_\pi(X_t)\,dt + \sigma\,dW_t, \qquad 0\leq t\leq T,  \\
  X_0 &\sim& \mathcal{G}_{\pi}(x),  \nonumber
\end{eqnarray}
with $\sigma=1$, $a_\pi=-\frac{d\pi}{dx}$ for a double-well potential $\pi$ defined in Appendix~\ref{app:doublewell}, 
and $\mathcal{G}_{\pi}$ the equilibrium distribution of the dynamics. 
The variance of the observation noise is $\tau^{2}= 0.2$.

\subsection{Algorithm settings}

To estimate the parameters needed to construct the MLPFs, we consider $\mathcal{S}=5$ 
i.i.d.~time series. For each time series $i$, we repeat $\mathcal{R}=100$ i.i.d.~MLPFs 
on level~$l$ with $N_l=2^{13}$, which is the largest ensemble we consider for the parameter estimation.

We define $w^{l,1}_{n,j,i,u}$ as the weight of the coarser particle in coupled 
particle $j$ at time $n$ of the $u^{th}$ repeat for time series $i$ with time 
step $h_l$. We denote weights on level $l=0$ as $w^{0}_{n,j,i,u}$.
The effective sample size at time observation $n$ is defined as
\begin{eqnarray}   
ESS^{N_l}_{n,j,i,u} = 
\begin{cases}
\Bigg(\sum_{j=1}^{N_l} \Big(w^{l,1}_{n,j,i,u}\Big)^{2} \Bigg)^{-1},  & \text{if } l>0,   \label{eq:ess}  \\
\Bigg(\sum_{j=1}^{N_l} \Big(w^{0}_{n,j,i,u}\Big)^{2} \Bigg)^{-1},   & \text{if } l=0, 
\end{cases}
\end{eqnarray}
with $N_l$ denoting the number of particles on level $l$. Two separate expressions are needed as when $l=0$ the batch of propagated particles is not coupled.

For the OU and NDT cases, the resampling step by Algorithm~\ref{alg:coupled} or Algorithm~\ref{alg:CDFcoupling} is made at observation $n$ if~\eqref{eq:ess} is below a given threshold of $ESS^{N_l}_{n,j,i,u}<N_l/4$, while for the DW case, given the dynamics instability, the resampling is performed at each observation time.

\subsection{Numerical results}

\subsubsection{Parameter estimation}

To build the optimal hierarchies for the MLPFs, as mentioned in the introduction to Section~\ref{sec:numerics}, we need to estimate the variance and the bias convergence rates and the computational complexity.

We compute the quantity
\begin{eqnarray}
\widehat{\varphi}^{N_l}_{l,n,i,u}= 
\begin{cases}
\sum_{j=1}^{N_l}\Big(w^{l,2}_{n,j,i,u}\varphi(X^{l,2}_{n,j,i,u})-w^{l,1}_{n,j,i,u}\varphi(X^{l,1}_{n,j,i,u})\Big), & \text{if } l>0, \label{eq:varphi}   \\
\sum_{j=1}^{N_l}w^{0}_{n,j,i,u}\varphi(X^{0}_{n,j,i,u}),& \text{if } l=0.
\end{cases}
\end{eqnarray}
at each observation time $n\in\lbrace 1,\dots,D\rbrace$, each repeat $u\in\lbrace 1,\dots,\mathcal{R}\rbrace$, and each time series $i\in\lbrace 1,\dots,\mathcal{S}\rbrace$.

We compute the standard unbiased sample variance $\mathcal{V}^{l}_{i,n}$ across the repeats $\widehat{\varphi}^{N_l}_{l,n,i,1:\mathcal{R}}$. We average $\mathcal{V}^{l}_{i,n}$ over the observation times and time series to estimate the variance
\begin{equation}
V_{l}=\sum_{i=1}^{\mathcal{S}}\bigg(N_l \sum_{n=1}^{D}\mathcal{V}^{l}_{i,n}/D\bigg)/\mathcal{S}.    \label{eq:variance}
\end{equation}
From the same simulations used in~\eqref{eq:variance}, the weak convergence rate needed to build the optimal hierarchy is obtained from the bias estimates 
\begin{equation}
 B_{l}=\sum_{i=1}^{\mathcal{S}}\bigg(\sum_{n=1}^{D}Q^{90}(\widehat{\varphi}^{N_{l+1}}_{l+1,n,i,1:\mathcal{R}})/D\bigg)/\mathcal{S} \label{eq:bias}
\end{equation}     
with $Q^{90}(\cdot)$ the $90\%$ percentile estimated on quantities $\widehat{\varphi}^{N_l}_{l,n,i,1:\mathcal{R}}$ defined in \eqref{eq:varphi}.  

The workstation has $47.2$ GiB of memory and Intel Xeon processor with twelve CPUs X5650 with 2.67 GHz; the operating system is Ubuntu 18.04.4 LTS. The numerical test are programmed in \textit{python 3.7.0} and the wall clock times are measured using library \textit{timeit}.

We perform three numerical tests to analyze and estimate the computational complexity. Firstly, we consider the computational time of the dynamics to propagate $N_l$ coupled particles with time step $h_{l}$ up to the final time $T$. Secondly, we compare it with the time for resampling $N_l$ coupled particles. The comparison of these two quantities is important to check which part of the process dominates the total complexity. Finally, the total cost per coupled particle on level $l$, $W_{l}$, is estimated and used to build the optimal hierarchies.

\paragraph{Ornstein-Uhlenbeck}

Now we illustrate the results for the OU case. Level $l$ corresponds to time step $h_{l}=\delta 2^{-l}$. 

Figure~\ref{fig:ou_variance} shows the variance estimated over a grid of various 
sizes of particle ensembles and time steps. If the sample variance~\eqref{eq:variance} 
is close to the true variance, it must converge with respect to the time steps but 
should not converge with respect to the particle ensemble sizes. Since the 
diffusion term of the OU process is constant, the variance convergence is expected 
to be of rate two. This is the rate that we achieve when we use Algorithm~\ref{alg:CDFcoupling} 
but it degrades to one with Algorithm~\ref{alg:coupled}, consistent with Corollary~4.4 in~\cite{mlpf}. 
Specifically, the rate estimated for $V_{l}$ by a least squares fit is $2.07$ for 
Algorithm~\ref{alg:CDFcoupling} and $1.27$ for Algorithm~\ref{alg:coupled}.

The bias, $B_{l}$, is in Figure~\ref{fig:ou_bias}. We can observe that it converges 
with rate close to $1$ as expected: rate $0.98$ for algorithm \ref{alg:coupled} and 
rate $1.06$ for algorithm \ref{alg:CDFcoupling}.

The left part of Figure~\ref{fig:ou_costs} shows how the cost of approximating 
the particle dynamics depends on the time step and the ensemble size. 
We observe that the cost is inversely proportional to $h_l$, as expected, but 
due to computational overhead it only becomes approximately proportional to $N$ 
for $N$ larger than $2^{10}$.
From the right part of Figure~\ref{fig:ou_costs}, we see that the particle 
dynamics dominates the resampling costs regardless of which resampling algorithm is used.

To estimate the computational complexity per coupled particle $W_{l}$, we sum the 
particle dynamics and resampling cost for the largest tested particle ensemble 
$N_{l}=2^{13}$ and then normalize with respect $N_l$. 
The resulting $W_l$, shown in the right part of Figure~\ref{fig:ou_costs}, is 
approximately inversely proportional to the time~step size, $h_l$, as expected. 
Note that since the estimate $W_l$ does not take into account the overhead 
associated with computing on a small number of particles, it may lead to suboptimal computational times.

\paragraph{SDE with a nonlinear diffusion term}

We use the same setup as in the OU case, but the results are slightly different 
from the previous case, since the non-constant diffusion term affects the variance 
convergence rate. Indeed, the variance rate of convergence is one for Algorithm~\ref{alg:CDFcoupling} 
and is $1/2$ for Algorithm~\ref{alg:coupled}, as theoretically predicted in 
Corollary~4.4 in~\cite{mlpf}. The fitted rates of $V_l$ in Figure~\ref{fig:ndt_variance} 
is $1.15$ for Algorithm~\ref{alg:CDFcoupling} and $0.64$ for Algorithm~\ref{alg:coupled}.

The bias, $B_{l}$, which we theoretically expect to converge with rate one, converges 
with rate $0.9$ in Figure~\ref{fig:ndt_bias} with Algorithm~\ref{alg:CDFcoupling}, 
while it degrades to $0.74$ with Algorithm~\ref{alg:coupled}. 
To avoid giving an unfair disadvantage to Algorithm \ref{alg:coupled} compared to 
Algorithm~\ref{alg:CDFcoupling}, we use the estimated rate from Algorithm~\ref{alg:CDFcoupling} 
to build optimal hierarchies for both resampling algorithms.

A numerical illustration of the dynamics and resampling costs are completely analogous 
to displayed results for the OU case in Figure~\ref{fig:ou_costs}. The computational 
complexity per coupled particle $W_{l}$ is estimated as described for the OU case, and 
shown in Figure~\ref{fig:ndt_work}. We observe the expected rate of one.

\paragraph{Double-Well Constant Diffusion}

For this numerical test, level $l$ corresponds to time step $h_{l}=\delta 2^{-(4+l)}$. 
Here, the coarsest discretization time step is chosen to be smaller than the time 
between observation, $\delta$, due to stability constraints of the numerical scheme. 
Unlike in the previous two examples, the drift coefficient function of the DW model 
is not globally Lipschitz, but it satisfies only the one-sided Lipschitz 
condition~\eqref{eq:onesidedlipschitz}. To guarantee the stability of the simulations 
in an infinite time interval, the time step has to satisfy Assumption~8 in~\cite{no_globally_lip}. 

For the DW case, we directly compare the obtained averaged variance convergence 
with and without the change of measure described in Appendix~\ref{app:changeofmeasure} 
for both algorithms in Figure~\ref{fig:dw_variance}. Although we do not observe 
a drastic change with or without the change of measure with Algorithm~\ref{alg:coupled}, 
the variance convergence is substantially improved with Algorithm~\ref{alg:CDFcoupling} 
and small particle ensemble sizes. Therefore, we implement the change of measure 
for both algorithms in the construction of the MLPFs. 

We expect to observe the same convergence for DW as for OU, since both have constant 
diffusion coefficients. Using the change of measure, the measured convergence rate 
for $V_{l}$ as in~\eqref{eq:variance} is $2.23$ for Algorithm~\ref{alg:CDFcoupling} 
and $1.15$ for Algorithm~\ref{alg:coupled}. 

In Figure~\ref{fig:dw_bias}, the bias $B_{l}$ converges nearly with rate one; 
the rate with Algorithm~\ref{alg:CDFcoupling} is $1.1$ and $1.05$ with 
Algorithm~\ref{alg:coupled}. Figure~\ref{fig:dw_work} shows that computational 
work per coupled particle, $W_{l}$, is inversely proportional to $h_l$, as expected.

\subsubsection{Creation of the MLPF estimators from the determined parameters}

Having determined the parameters $B_{l}$, $V_{l}$, and $W_{l}$ as described in the 
previous section, we fix a sequence of decreasing tolerances
\begin{equation}
\epsilon_{k}=\epsilon_{1}\bigg(\dfrac{1}{\sqrt{2}}\bigg)^{k},   \label{eq:tol}
\end{equation}
with $\epsilon_{1}=3/100$, for the convergence study. For the OU and NDT cases, 
$k\in\lbrace 0,\dots,8\rbrace$, while for the DW case, $k\in\lbrace 0,\dots,9\rbrace$. 
With the specified $\epsilon_{k}$, we use the estimated variances $V_{l}$ and biases 
$B_{l}$ to compute the optimal levels $(l^{*},\dots,L^{*})$ and coupled particles on 
each level $N^{*}_{l^{*}:L^{*}}$ for the MLPFs via Algorithm~1 described in 
Section~5.2.2 in~\cite{seismology}. The algorithm is briefly described below.

The goal of the algorithm is to determine an MLPF, $\varphi^{*,L^{*}}_{k,i}(X_{T})$, 
that, for time series $i$, satisfies
\begin{eqnarray}    \label{eq:failure_probability}
|\mathbb{E}[\varphi(X_{T})]-\varphi^{*,L^{*}}_{k,i}(X_{T})|\leq \epsilon_{k}, \text{\qquad with probability $1-\xi$, for $0<\xi\ll 1$,}
\end{eqnarray}
and that has the minimum computational cost among the set of all the feasible MLPFs.
To satisfy the failure probability~\eqref{eq:failure_probability}, we control the 
bias and the statistical error separately by introducing a parameter $\phi\in(0,1)$ and requiring 
\begin{eqnarray}
|\mathbb{E}[\varphi(X_{T})-\varphi^{*,L}_{k,i}(X_{T})]|&\leq& (1-\phi)\epsilon_{k}   \label{eq:bias_constr}  \\
P[|\mathbb{E}[\varphi^{*,L^{*}}_{k,i}(X_{T})]-\varphi^{*,L^{*}}_{k,i}(X_{T})|>\phi \epsilon_{k}]&\leq& \xi.   \label{eq:statistical_error}
\end{eqnarray}     
To satisfy~\eqref{eq:statistical_error}, the asymptotic normality of coupled particle 
filters, in Theorem~\ref{theo:clt_was}, motivates the requirement
\begin{eqnarray}    
\mathbb{V}\textrm{ar}[\varphi^{*,L^{*}}_{k,i}(X_{T})]\leq \Bigg(\dfrac{\phi \epsilon_{k}}{C_{\xi}}\Bigg)^{2},    \label{eq:error_bound}
\end{eqnarray}
where $C_{\xi}$ is a confidence parameter such that $\Phi(C_{\xi})=1-\frac{\xi}{2}$ 
with $\Phi$ the cumulative distribution function of a standard normal random variable. 
Here, we fix $C_{\xi}=2$, which corresponds to $\xi\approx0.05$.

Given the available time step sizes, $h_l=h_0 2^{-l}$, $l=0,1,\dots$, the MLPF 
is characterized by the included time step sizes, corresponding to levels 
$l=l_0,\dots,L$ with $0\leq l_0\leq L$, and the number of particles 
$\{N_l\}_{l=l_0}^L$. Here, $L$ must be large enough for the bias estimate, $B_L$, 
to satisfy the bias constraint~\eqref{eq:bias_constr} for some $\phi\in(0,1)$. 
Any permissible choice of $L$ implicitly defines $\phi=1-B_L/\epsilon_{k}$ and
thus the constraint on the permissible variance in~\eqref{eq:error_bound}. 
Given $l_0$ and $L$, the optimal number of particles are given by 
\begin{equation*}
  N_l = \left\lceil\Bigg(\dfrac{C_{\xi}}{\phi \epsilon_{k}}\Bigg)^{2}
    \sqrt{\dfrac{V_l}{W_l}}\sum_{l=l_0}^L\sqrt{V_l W_l}\right\rceil,
\end{equation*}
as in standard MLMC. Finally, the optimal MLPF 
$(l_0^{*},L^{*},N^{*}_{l_0^{*}:L^{*}})$ is obtained by minimizing the estimated
work, $W=\sum_{l=l_0}^L W_{l}N_l$, over all permissible choises of $l_0$ and $L$.
The estimated optimal hierarchies for the OU, NDT, and DW examples are given 
in Table~\ref{tab:ou_hierarchy}, \ref{tab:ndt_hierarchy}, 
and~\ref{tab:dw_hierarchy} respectively. 

To study the accuracy and the efficiency of the MLMF with resampling by 
Algorithm~\ref{alg:coupled} and Algorithm~\ref{alg:CDFcoupling}, respectively, 
we generate $100$ i.i.d.~time series. For the $i^{th}$ such time series and 
given tolerance $\epsilon_{k}$, the MLPF estimator is defined as
\begin{eqnarray}   \label{eq:tolerances}
\varphi^{*,L^{*}}_{k,i}(X_{T})=\sum_{l=l_0^{*}}^{L^{*}}\widehat{\varphi}^{N^{*}_l}_{l,n,i,1},
\end{eqnarray}
with $\widehat{\varphi}^{N^{*}_l}_{l,n,i,1}$ in \eqref{eq:varphi} and $N^{*}_{l}$ 
the optimal number of coupled particles on level $l$ for a given $\epsilon_{k}$. 
In the case where $l_0>0$, the estimators in~\eqref{eq:varphi} use the single 
level particle filter on level $l=l_0$ and the coupled particles filters on all 
subsequent levels $l>l_0$.

\subsubsection{Observed errors and computational times for the constructed MLPFs}

For each time series $i\in\lbrace 1,\dots,100\rbrace$, we compute the reference 
solution $\widehat{\varphi}_{ex,i}(X_T)$ of the expected value of the filtering 
distribution $\mathbb{E}[\varphi(X_{T})|y_{1:D}]$. For the OU case, the reference 
solution is the exact one computed by the Kalman Filter, while for the NDT and 
DW cases, we approximate the solution of the corresponding Fokker-Planck equations 
numerically with accuracies that guarantee the numerical errors to be negligible 
in the numerical experiments. 

The error compared to the reference solution is
\begin{eqnarray}
E_{k,i}=|\widehat{\varphi}_{ex,i}(X_T)-\varphi^{*,L^{*}}_{k,i}(X_{T})|.   \label{eq:E_k_i}
\end{eqnarray}
The goal is to obtain errors $\lbrace E_{k,i}\rbrace_{u=1}^{100}$ bounded by 
$\epsilon_{k}$ with high probability; with the choice of $C_{\xi}$ above, we 
expect this goal to be satisfied in around $95\%$ of cases.

We estimate the total cost of generating $\varphi^{*,L^{*}}_{k,u}(X_{T})$ and 
compare to the theoretical complexity.

\paragraph{Ornstein-Uhlenbeck (OU)} The results are illustrated in 
Figures~\ref{fig:ou_tolvserror} and~\ref{fig:ou_actualwork}. 
Figure~\ref{fig:ou_tolvserror} displays the errors $\lbrace E_{k,i}\rbrace_{i=1}^{100}$ 
for the sequence of tolerances~\eqref{eq:tol}. 
On average $3\%$ of the errors are larger than the corresponding tolerances when 
using Algorithm~\ref{alg:coupled} and $5\%$ with Algorithm~\ref{alg:CDFcoupling}. 
In Figure~\ref{fig:ou_actualwork}, we observe that the asymptotic computational 
time is proportional to $\epsilon^{-2}$ with Algorithm~\ref{alg:CDFcoupling}, 
while it is proportional to $\epsilon^{-2}\log(\epsilon)^{2}$ with 
Algorithm~\ref{alg:coupled}, as theoretically predicted. 

\paragraph{SDE with a nonlinear diffusion term (NDT)} 
Figure~\ref{fig:ndt_tolvserror} shows higher rates of failure to meet the 
the prescribed tolerance than in the OU case for both resampling algorithms.
For Algorithm~\ref{alg:coupled} and~\ref{alg:CDFcoupling} the failure rates are 
on average $12\%$ and $8\%$, respectively. Figure~\ref{fig:ndt_actualwork} shows 
that the actual complexity of the MLPFs is $\epsilon^{-2}\log(\epsilon)^{2}$ 
with Algorithm~\ref{alg:CDFcoupling} and $\epsilon^{-3}$ with 
Algorithm~\ref{alg:coupled}. The increased complexity compared to the OU case is
due to the lower rate of strong convergence.

\paragraph{Double-Well Constant Diffusion (DW)} Here, the errors for the given
tolerances are shown in Figure~\ref{fig:dw_tolvserror}. Using 
Algorithm~\ref{alg:coupled}, the rate of failure is on average $7\%$, while using 
Algorithm~\ref{alg:CDFcoupling}, it is on average $4\%$. It can be seen in 
Figure~\ref{fig:dw_actualwork} that the observed computational times agree well 
with the theoretically expected asymptotic complexities, which are 
$\epsilon^{-2}$ when Algorithm~\ref{alg:CDFcoupling} is used, and 
$\epsilon^{-2}\log(\epsilon)^{2}$ with Algorithm~\ref{alg:coupled}. 

\paragraph{Conclusions of the numerical experiments} 
We can observe that for the OU and DW cases, the percentage of the independent 
runs that fail to meet the error tolerance is close to the target of $5\%$ used for 
the MLPFs construction. The rate of failure is slightly higher for the NDT case 
but this rate can likely be improved by increasing the numbers of independent 
time series used for the parameter estimations.  

The actual work for the three numerical cases is in agreement with the 
theoretically predicted complexity. Thus, we can conclude that, for a fixed 
decreasing sequence of tolerances, the MLPFs using the proposed resampling by
Algorithm~\ref{alg:CDFcoupling} are asymptotically cheaper than those 
using resampling by Algorithm~\ref{alg:coupled}.

\appendix

\section{Double-Well potential} \label{app:doublewell}

This numerical example has a state switching behaviour, illustrated in 
Figure~\ref{fig:dw_dynamics}, which is relevant to many potential applications 
of MLPFs.
The drift coefficient of the double-well diffusion equation~\eqref{eq:DW_dynamics} 
is $a_\pi=-\frac{d\pi}{dx}$ with $\pi$ the double-well potential
\begin{align}
  \label{eq:dwell_pot}
  \pi(x) & =  
    \begin{cases}
      k_1x + \left(x^2-1\right)^2,& |x|\leq k_2,\\
      k_1x + \sum_{j=0}^4c_j\left(|x|-k_2\right)^j, & k_2<|x|\leq 2k_2,\\
      k_1x + \sum_{j=0}^2\widehat{c}_j\left(|x|-2k_2\right)^j, & |x|>2k_2,\\
    \end{cases}
\end{align}
and with parameters defined in Table~\ref{tab:dwell_consts}. That is, 
\begin{align}
  \label{eq:drift_dwell_pot}
  a_\pi(x) & =  
    \begin{cases}
      -k_1 - 4x\left(x^2-1\right),& |x|\leq k_2,\\
      -k_1 - \sgnp{x}\sum_{j=1}^4jc_j\left(|x|-k_2\right)^{j-1}, & k_2<|x|\leq 2k_2,\\
      -k_1 - \sgnp{x}\left(\widehat{c}_1+2\widehat{c}_2\left(|x|-2k_2\right)\right), & |x|>2k_2.\\
    \end{cases}    \nonumber
\end{align}
The potential has been adjusted to quadratic growth outside an interval containing 
the two potential minima so that the drift coefficient function only grows linearly. 
To this  end, a parameter $k_2\geq\sqrt{2}$ is chosen such that the tilted
double-well potential is kept in $[-k_2,k_2]$ and the potential is a second degree 
polynomial outside $[-2k_2,2k_2]$.
By this construction, the drift coefficient function is in $\mathcal{C}^2(\rset)$
and has linear growth outside the interval $[-2k_2,2k_2]$.

The initial distribution is taken to be the Gibbs' measure, $\mathcal{G}_{\pi}$, defined as
\begin{equation}   \label{eq:gibbs}
 \mathcal{G}_{\pi}(x)=\dfrac{e^{-\frac{2}{\sigma^{2}}\pi(x)}}{\int_{-\infty}^{+\infty}e^{-\frac{2}{\sigma^{2}}\pi(x)}\mathrm{d}x}.
\end{equation}
The measure $\mathcal{G}_\pi$ corresponds to the equilibrium distribution of the stochastic dynamics.

\begin{table}[!ht]
  \centering
  \begin{tabular}{|c|c|c|}
    \hline
    \multicolumn{2}{|c|}{Parameters}  & Value \\
    \hline 
    $k_1$ & Tilt of double well; $|k_1|<\sqrt{\frac{64}{27}}$ \Tstrut & $\frac{\sqrt{2}}{12}$ \\
    \hline
    $k_2$ & Half length of unchanged interval; $k_2\geq\sqrt{2}$ \Tstrut & $\sqrt{2}$ \\
    \hline
    \multicolumn{3}{|c|}{Stationary Points (Depending on $k_1$)}\\
    \hline
    \multirow{3}{*}{$r_j$} & 
    \multirow{3}{*}
    {$\frac{2}{\sqrt{3}}\cos{\left(\frac{1}{3}
      \arccos{\left(-\frac{3\sqrt{3}k_1}{8}\right)}
      -\frac{2\pi j}{3}\right)}$, $j=0,1,2$} & -1.0144 \\ & & 0.0295\\ & & 0.9849 \\
    \hline
    \multicolumn{3}{|c|}{Polynomial Coefficients (Depending on $k_2$)}\\
    \hline
    $c_0$ & $\left(k_2^2-1\right)^2$ \Tstrut & \multicolumn{1}{|c|}{$1$} \\
    $c_1$ & $4k_2\left(k_2^2-1\right)$  \Tstrut & \multicolumn{1}{|c|}{$4\sqrt{2}$} \\
    $c_2$ & $2\left(3k_2^2-1\right)$ \Tstrut & \multicolumn{1}{|c|}{$10$} \\
    $c_3$ & $4k_2$ \Tstrut & \multicolumn{1}{|c|}{$4\sqrt{2}$} \\
    $c_4$ & $-1$ \Tstrut & \multicolumn{1}{|c|}{$-1$} \\
    \hline
    $\widehat{c}_0$ & $14k_2^4-8k_2^2+1$ \Tstrut & \multicolumn{1}{|c|}{$41$} \\
    $\widehat{c}_1$ & $c_2c_3$ \Tstrut & \multicolumn{1}{|c|}{$40\sqrt{2}$} \\
    $\widehat{c}_2$ & $2\left(6k_2^2-1\right)$ \Tstrut & \multicolumn{1}{|c|}{$22$} \\
    \hline
    \multicolumn{3}{|c|}{Stability Bound Euler-Maruyama (Depending on $k_2$)}\\
    \hline
    $h_\mathrm{max}$ & $\frac{1}{\widehat{c}_2}$ \Tstrut \Bstrut & \multicolumn{1}{|c|}{$\frac{1}{22}$}\\
    \hline    
  \end{tabular}
  \caption{Constants and parameters in the tilted double-well drift~\eqref{eq:DW_dynamics} 
    and potential~\eqref{eq:dwell_pot}.} 
  \label{tab:dwell_consts}
\end{table}

\section{Change of measure}          \label{app:changeofmeasure}

The change of measure described here is introduced in~\cite{MLMC_no_contractivity} 
for the case in which the dynamics~\eqref{eq:model_setting} have constant diffusion 
coefficient $\sigma=1$. We generalize the calculations for a generic constant 
$\sigma$.

Suppose that the dynamics~\eqref{eq:model_setting} does not satisfy the contractivity condition
\begin{eqnarray}
\langle x-y, a(x)-a(y)\rangle \leq -\lambda \|x-y \|^{2}.   \label{eq:contractivity}
\end{eqnarray} 
A change of measure that recovers property~\eqref{eq:contractivity} can be 
constructed by introducing a spring term with strength $S\in\mathbb{R}$ in the 
drift, provided that $S>\frac{\lambda}{2}$, where $\lambda$ is determined by the 
one-sided Lipschitz condition
\begin{eqnarray}
\langle x-y,a(x)-a(y)\rangle \leq \lambda \| x-y \|^{2}.       \label{eq:onesidedlipschitz}
\end{eqnarray} 

The objective here is a construction in continuous time that will be well approximated with MLMC. 
In a standard multilevel setup, given a diffusion dynamics~\eqref{eq:model_setting}, 
coarse and fine paths are simulated on two different measures, $\Q^{1}$ and $\Q^{2}$, respectively, so
\begin{subequations}
\begin{eqnarray}
\mathrm{d} X^{1}_{t}&=&a(X^{1})\mathrm{d}t+b(X^{1}_{t})\mathrm{d}W^{\mathbb{Q}^{1}}_{t}  \label{eq:Q1}  \\
\mathrm{d} X^{2}_{t}&=&a(X^{2})\mathrm{d}t+b(X^{2}_{t})\mathrm{d}W^{\mathbb{Q}^{2}}_{t}  \label{eq:Q2}
\end{eqnarray}
\end{subequations}
The change of measure consists in considering diffusion dynamics~\eqref{eq:Q1} 
and~\eqref{eq:Q2} under a common measure $\Pb$
\begin{eqnarray}
\mathrm{d}U_{t}^{1}&=&S(U_{t}^{2}-U_{t}^{1})\mathrm{d}t+a(U_{t}^{1})\mathrm{d}t+\sigma\mathrm{d}W^{\Pb}_{t},   \nonumber \\
\mathrm{d}U_{t}^{2}&=&S(U_{t}^{1}-U_{t}^{2})\mathrm{d}t+a(U_{t}^{2})\mathrm{d}t+\sigma\mathrm{d}W^{\Pb}_{t}.	 \nonumber
\end{eqnarray}
With quantity of interest $\varphi(\cdot)$, it holds by the Girsanov theorem that,
\begin{eqnarray}
\mathbb{E}_{\mathbb{Q}^{2}}\big[\varphi(X^{2}_{t})\big]-\mathbb{E}_{\mathbb{Q}^{1}}\big[\varphi(X^{1}_{t})\big]=\mathbb{E}_{\mathbb{P}}\bigg[\varphi(U^{2}_{t})\frac{\mathrm{d}\Q^{2}}{\mathrm{d}\Pb}-\varphi(U^{1}_{t})\frac{\mathrm{d}\Q^{1}}{\mathrm{d}\Pb}\bigg]=0  \label{eq:girsanov}
\end{eqnarray}
where $\frac{\mathrm{d}\Q^{2}}{\mathrm{d}\Pb}$ and $\frac{\mathrm{d}\Q^{1}}{\mathrm{d}\Pb}$ 
are Radon-Nikodym derivatives.

Assume that $\lbrace Z_t\rbrace$ follows the dynamics modeled by an SDE with a standard 
Brownian motion $\lbrace W_{t}\rbrace_{t\geq 0}$, as in~\eqref{eq:model_setting}. 
If we discretize the SDE with time step $h$ by Euler-Maruyama, we obtain 
$\lbrace Z_{t_n}\rbrace_{n\in\mathbb{N}}$. We define the function
\begin{eqnarray}
R(Z_{t_{n+1}},Z_{t_n},S,h)=\exp{\Big(-\dfrac{\langle \Delta W_{n},S\rangle}{\sigma}-\dfrac{\|S\|^{2}h}{2\sigma^{2}}\Big)},  \label{eq:R}
\end{eqnarray}  
with $\Delta W_{n}=W_{t_{n+1}}-W_{t_{n}}$. 
Then the Radon-Nikodym derivatives at time $p$ are discretized as 
\begin{eqnarray} 
\left.\frac{\mathrm{d}\mathbb{Q}^{2}}{\mathrm{d}\mathbb{P}}\right\vert_{p} \sim R^{l,2}_{p}&=&\prod_{q=0}^{2^{l}-1}R\Big(U^{l,2}_{p+h_{l}(q+1)},U^{l,2}_{p+h_{l}q},S(U^{l,1}_{p+h_{l}q}-U^{l,2}_{p+h_{l}q}),h_{l}\Big)     \label{eq:discretechangeofmeasure2} \\
\left.\frac{\mathrm{d}\mathbb{Q}^{1}}{\mathrm{d}\mathbb{P}}\right\vert_{p} \sim R^{l,1}_{p}&=&\prod_{q=0}^{2^{l-1}-1}R\Big(U^{l,1}_{p+h_{l}(2q+2)},U^{l,1}_{p+h_{l}2q},S(U^{l,2}_{p+h_{l}2q}-U^{1,l}_{p+h_{l}2q}),h_{l-1}\Big)    \label{eq:discretechangeofmeasure1}
\end{eqnarray}
where $h_l$ and $h_{l-1}$ are time step sizes.
It follows that \eqref{eq:girsanov} is discretized at time observation $k$ as
\begin{eqnarray}
\mathbb{E}_{\mathbb{Q}^{2}}\big[\varphi(X^{l,2}_{k})\big]-\mathbb{E}_{\mathbb{Q}^{1}}\big[\varphi(X^{l,1}_{k})\big]&=&\mathbb{E}_{\Pb}\Bigg[\varphi(U^{l,2}_{k})\prod_{p=1}^{k}R^{l,2}_{p}-\varphi(U^{l,1}_{k})\prod_{q=1}^{k}R^{l,1}_{p}\Bigg].	
\end{eqnarray}

\subsection{Particle filter in presence of change of measure}

We consider the predictor, firstly where no approximation is given, then with a 
Euler approximation and finally a combination of Euler and the particle/ML filter.

In the absence of bias, the predictor of the particle filter at time $k$ is
\begin{eqnarray} 
\eta_{k}(\varphi)=\frac{\mathbb{E}\Big[ \varphi(X_{k})\prod_{p=0}^{k-1}G_{p}(X_{p})\Big]}{\mathbb{E}\Big[\prod_{p=1}^{k-1}G_{p}(X_{p})\Big]}     \label{eq:predictor}
\end{eqnarray}
where $G_{p}$ is the likelihood density of the data point at observation time $p$ and the expectation is w.r.t.~the law of the diffusion process.
Discretizing~\eqref{eq:predictor} gives
\begin{eqnarray}
\eta_{k}^{l}(\varphi)=\frac{\mathbb{E}^{l}\Big[\varphi(X^{l}_{k})\prod_{p=0}^{k-1}G_p(X^{l}_{p})\Big]}{\mathbb{E}^{l}\Big[\prod_{p=0}^{k-1}G_p(X^{l}_{p})\Big]},    \nonumber
\end{eqnarray}
where now the expectation is w.r.t.~the law of the Euler discretized diffusion with time step $h_{l}$.
If we consider the change of measure for two dynamics on consecutive time step 
sizes, $X^{l,1}_{k}$ and $X^{l,2}_{k}$, the predictor at time $k$ is:
\begin{eqnarray}   \label{eq:discretepredictor}
\eta^{l,i}_{k}(\varphi)=\frac{\mathbb{E}^{l,i}\big[\varphi(X^{l,i}_{k})R^{l,i}_{k}\prod_{p=0}^{k-1}R^{l,i}_{p}G_p(X^{l,i}_{p})\big]}{\mathbb{E}^{l,i}\big[R^{l,i}_{k}\prod_{p=0}^{k-1}R^{l,i}_{p}G_p(X^{l,i}_{p})\big]}, \qquad i=1,2.  \nonumber
\end{eqnarray}
with $R^{l,1}_{p}$, $R^{l,1}_{k}$ and $R^{l,2}_{p}$, $R^{l,2}_{k}$ 
in~\eqref{eq:discretechangeofmeasure1} and~\eqref{eq:discretechangeofmeasure2}, 
respectively.
Consequently, the estimate of the predictors given a coupled particle ensemble 
$\lbrace X^{l,2}_{j},X^{l,1}_{j} \rbrace_{j=1}^{N_{l}}$, without resampling is
\begin{eqnarray} \label{eq:estimatediscretepredictor}
\widehat{\eta}^{l,i}_{k}(\varphi)=\frac{\sum_{j=1}^{N_{l}}\varphi(X^{l,i}_{k,j})R^{l,i}_{k,j} \prod_{p=0}^{k-1}R^{l,i}_{p,j}G_{p}(X^{l,i}_{k,j})}{\sum_{j=1}^{N_{l}}R^{l,i}_{k,j}\prod_{p=0}^{k-1}R^{l,i}_{p,j}G_{p}(X^{l,i}_{k,j})}, \qquad i=1,2.     \nonumber
\end{eqnarray}
If we resample at each observation time, then
\begin{eqnarray}
\widehat{\eta}^{l,i}_{k}(\varphi)=\frac{\sum_{j=1}^{N_{l}}\varphi(X^{l,i}_{k,j})R^{l,i}_{k,j}}{\sum_{j=1}^{N_{l}}R_{k,j}^{l,i}}, \qquad i=1,2     \nonumber
\end{eqnarray}
and the difference is simply
\begin{eqnarray}
\Delta \eta_{k}^{l}(\varphi) &=&\widehat{\eta}^{l,2}_{k}(\varphi)-\widehat{\eta}^{l,1}_{k}(\varphi).    \nonumber
\end{eqnarray}
Similar calculations can be performed for the filter, but are omitted.

\section{Algorithm Listings}\label{app:alg}

\subsection{Resampling Algorithms}
\label{app:resampling}

Here, we describe the two alternative resampling algorithms used with MLPFs in 
the present paper. 

\paragraph*{Particle Index Coupled Resampling Algorithm}

This is Algorithm~1, page~3074 in~\cite{mlpf}; listed as 
Algorithm~\ref{alg:coupled} here for completeness. The idea behind 
Algorithm~\ref{alg:coupled} is to minimize the probability of decoupling of the 
coarse and fine trajectories in the resampling step. However, this minimization comes 
with the cost of complete decoupling of the trajectories that are decoupled, in 
contrast to Algorithm~\ref{alg:CDFcoupling} proposed in this paper, which aims 
to correlate the coarse and fine particles of any pair through the CDF.

In the following, we denote by $F^{-1}$ the generalized inverse of a CDF, 
that is 
\begin{align*}
  F^{-1}(y) & = \inf\{s\colon F(s)\geq y\}.  
\end{align*}
Algorithm~\ref{alg:coupled} is based on this inverse of a CDF over the integer 
range of particle integers.

\subparagraph*{Complexity}

All five computations in the first eight lines of Algorithm~\ref{alg:coupled} 
are $\mathcal{O}(N)$. 
To find the indices $I$ or $(I_{1},I_{2})$, for each of the $N$ steps in the loop 
on lines 9--21,
one performs a search in a sorted array. An efficient algorithm to do so is the 
Binary Search Algorithm, which has average cost $\mathcal{O}(\log N)$. It follows 
that the total cost is $\mathcal{O}(N\cdot \log N)$.

\paragraph*{CDF Coupled Resampling Algorithm} 
This algorithm is based on inverting the empirical cumulative density function 
(CDF) of the particle positions, in one dimension, in order to obtain correlated 
coarse and fine level samples even after resampling. This procedure is described 
in Algorithm~\ref{alg:CDFcoupling}.

\subparagraph*{Complexity}

In Algorithm~\ref{alg:CDFcoupling}, line \ref{row:sort}, we sort two arrays 
of size $N$, which can be done in average complexity $\mathcal{O}(N\log N)$, 
using, for example, the Quicksort Algorithm.
In line~\ref{row:empiricalCDF}, the complexity of creating the reweighted 
empirical CDFs is $\mathcal{O}(N)$. 
In line~\ref{row:inver}, as pointed out before, the searches through two sorted 
arrays are done by Binary Search, at cost $\mathcal{O}(\log N)$.
We conclude that the total complexity of Algorithm~\ref{alg:CDFcoupling} is 
$\mathcal{O}(N\log N)$.

\newlength{\Inputindent}
\settowidth{\Inputindent}{\textbf{Input:}}
\begin{algorithm}[H]
  \hspace*{\algorithmicindent} \textbf{Input:} $N$ particle pairs
  $\lbrace X^{1}_{n},X^{2}_{n}\rbrace_{n=1}^{N}$ and normalized weights 
  $\lbrace w^{1}_{n}, w^{2}_{n}\rbrace_{n=1}^{N}$. \\
  \hspace*{\algorithmicindent} \textbf{Output}: $N$ particle pairs
  $\lbrace \widetilde{X}^{1}_{n},\widetilde{X}^{2}_{n}\rbrace_{n=1}^{N}$ and normalized weights 
  $\lbrace \widetilde{w}^{1}_{n}, \widetilde{w}^{2}_{n}\rbrace_{n=1}^{N}$. \\
  \hspace*{\algorithmicindent} \textbf{Note:} All random variables sampled are mutually independent.
  \begin{algorithmic}[1]
    \For{n=1:N}  \Comment{Compute overlapping part of probabilities on $1,\dots,N$}
      \State $w_{n}^{\mathtt{min}}\gets\min{\{w^{1}_n,w^{2}_n\}}$
    \EndFor
    \State $\alpha\gets\sum_{n=1}^N w_n^\mathtt{min}$
    \For{n=1:N}  \Comment{Compute cumulative distributions on $1,\dots,N$}
      \State $F_0(n)\gets\sum_{j=1}^n w^\mathtt{min}_{j}/\alpha = F_0(n-1)+w^\mathtt{min}_{n}/\alpha$
        \Comment{Shared emprical CDF}
      \State $F_j(n)\gets\sum_{i=1}^n (w_i^j-w_i^\mathtt{min})/(1-\alpha)$, for $j=1,2$
        \Comment{Distinct emprical CDFs}
    \EndFor
    \For{n=1:N}
      \State Sample $V_n\sim\mathcal{U}(0,1)$.
      \If{$V_n<\alpha$} 
        \Comment{With probability $\alpha$ sample the same index for both.}
        \State Sample $U_{n}\sim \mathcal{U}(0,1)$. 
        \State $I\gets {F_0}^{-1}(U_n)$ \label{row:noelse}
        \State $\lbrace\widetilde{X}^{1}_{n},\widetilde{X}^{2}_{n}\rbrace 
        \gets\lbrace X^{1}_{I},X^{2}_{I}\rbrace$
      \Else \label{row:else} 
        \Comment{With probability $1-\alpha$ sample the indices independently.}
        \State Sample $P_{n}\sim \mathcal{U}(0,1)$ and $Q_{n}\sim \mathcal{U}(0,1)$ independently.
        \State $\lbrace I_1,I_2\rbrace \gets
          \lbrace F_{1}^{-1}(P_{n}),F_{2}^{-1}(Q_n)\rbrace$ \label{row:else1} 
        \State
        $\lbrace\widetilde{X}^{1}_{n},\widetilde{X}^{2}_{n}\rbrace
        \gets \lbrace X^{1}_{I_1}, X^{2}_{I_2}\rbrace$\label{row:else2}
      \EndIf
      \State $(\widetilde{w}_n^1,\widetilde{w}_n^2) \gets (1/N,1/N)$
    \EndFor
  \end{algorithmic}
  \caption{Coupled Resampling through Particle Indices (from paper~\cite{mlpf})}
  \label{alg:coupled}
\end{algorithm}
\begin{algorithm}[H]
 \hspace*{\algorithmicindent} \textbf{Input:} $N$ particle pairs $\lbrace X^{1}_{n},X^{2}_{n}\rbrace_{n=1}^{N}$ and normalized weights $\lbrace w^{1}_{n}, w^{2}_{n}\rbrace_{n=1}^{N}$. \\
 \hspace*{\algorithmicindent} \textbf{Output}: $N$ particle pairs $\lbrace \widetilde{X}^{1}_{n},\widetilde{X}^{2}_{n}\rbrace_{n=1}^{N}$ and normalized weights $\lbrace \widetilde{w}^{1}_{n}, \widetilde{w}^{2}_{n}\rbrace_{n=1}^{N}$. \\
  \hspace*{\algorithmicindent} \textbf{Note:} All random variables sampled are mutually independent.
  \begin{algorithmic}[1]
	\State Sort $\lbrace\overline{X}^{i}_{n},\overline{w}^{i}_{n}\rbrace_{n=1}^{N}=\lbrace X^{i}_{I^{i}_n},w^{i}_{I^{i}_n}\rbrace_{n=1}^{N}$, with $I^{i}_{n}$ the index array that sorts $\lbrace X^{i}_{n}\rbrace^{N}_{n=1}$, for $i=1,2$.  \label{row:sort}
    \State $F_{i}(n) \gets \sum_{j=1}^{n}\overline{w}^{i}_{j}$, for $i=1,2$.   \label{row:empiricalCDF}
    \For{n=1:N}
        \State Sample $U_{n}\sim \mathcal{U}(0,1)$
        \State $\lbrace q^{1}_{n},q^{2}_{n}\rbrace \gets \lbrace F^{-1}_{1}(U_{n}),F^{-1}_{2}(U_{n})\rbrace$ \label{row:inver}
        \State $\lbrace\widetilde{X}^{1}_{n}, \widetilde{X}^{2}_{n}\rbrace \gets \lbrace\overline{X}^{1}_{q^{1}_n}, \overline{X}^{2}_{q^{2}_n}\rbrace$
    \State $(\widetilde{w}_n^1,\widetilde{w}_n^2) \gets (1/N,1/N)$
    \EndFor
  \end{algorithmic}
  \caption{Coupled Resampling through inverse CDF}
  \label{alg:CDFcoupling}
\end{algorithm}

\section{Proofs for the CLT}\label{app:clt}

The underlying strategy is that used in~\cite{jasra_yu} and we share various notational conventions and approaches used in that article.
The main new results in this paper are Lemmata~\ref{lem:ind1}-\ref{lem:cond_exp_was}. The other results have parallels in~\cite{jasra_yu} but are included for completeness  (note that Appendix C of this article differs from~\cite{jasra_yu}).
We set for $\mu\in\mathscr{P}(\mathsf{X})$, $s\in\{f,c\}$, $\varphi\in\mathcal{B}_b(\mathsf{X})$, $n\geq 1$
$$
\Phi_n^s(\mu)(\varphi) = \frac{\mu(G_{n-1}M_n^s(\varphi))}{\mu(G_{n-1})}.
$$
Denote the sequence of non-negative kernels $\{Q_n^s\}_{n\geq 1}$, $s\in\{f,c\}$, $Q_n^s(x,dy) = G_{n-1}(x) M_n^s(x,dy)$ and for $\varphi\in\mathcal{B}_b(\mathsf{X})$, $x_p\in\mathsf{X}$
$$
Q_{p,n}^s(\varphi)(x_p) = \int_{\mathsf{X}^{n-p}} \varphi(x_n) \prod_{q=p}^{n-1} Q_{q+1}^s(x_q,dx_{q+1})
$$
$0\leq p<n$ and in the case $p=n$, $Q_{p,n}^s(\varphi)(x)=\varphi(x)$.
Now denote for $0\leq p<n$, $s\in\{f,c\}$, $\varphi\in\mathcal{B}_b(\mathsf{X})$, $x_p\in\mathsf{X}$
$$
D_{p,n}^s(\varphi)(x_p) = \frac{Q_{p,n}^s(\varphi - \eta_n^s(\varphi))}{\eta_p^s(Q_{p,n}^s(1))}
$$
in the case $p=n$, $D_{p,n}^s(\varphi)(x)=\varphi(x)-\eta_n^s(\varphi)$.

For $p\geq 0$, $\varphi\in\mathcal{B}_b(\mathsf{X})$, $s\in\{f,c\}$
$$
V_p^{N,s}(\varphi) = \sqrt{N}[\eta_p^{N,s}-\Phi_p^s(\eta_{p-1}^{N,s})](\varphi)
$$
with the convention that $V_0^{N,s}(\varphi) = \sqrt{N}[\eta_0^{N,s}-\eta_{0})](\varphi)$.
For $n>0$, $0\leq p\leq n-1$, $\varphi\in\mathcal{B}_b(\mathsf{X})$, $s\in\{f,c\}$,
$$
R_{p+1}^{N,s}(D_{p,n}^s(\varphi)) = \frac{\eta_p^{N,s}(D_{p,n}^s(\varphi))}{\eta_p^{N,s}(G_p)}[\eta_p^{s}(G_p)-\eta_p^{N,s}(G_p)].
$$

Now we note that, using the calculations in \cite{mlsmc,ddj2012}
\begin{eqnarray}
\sqrt{N}[\check{\eta}_{n}^{N,W}-\check{\eta}_{n}^{W}](\varphi\otimes 1 - 1 \otimes \psi) & = &  \sum_{p=0}^n \{V_p^{N,f}(D_{p,n}^f(\varphi)) -V_p^{N,c}(D_{p,n}^c(\psi))\}
+ \nonumber\\ & &
\sqrt{N}\sum_{p=0}^{n-1} \{R_{p+1}^{N,f}(D_{p,n}^f(\varphi)) - R_{p+1}^{N,c}(D_{p,n}^c(\psi))\}\label{eq:master_ind}.
\end{eqnarray}
\begin{proof}[Proof of Theorem \ref{theo:clt_was}]
The proof follows immediately from \eqref{eq:master_ind}, Lemma \ref{lem:r_cont_ind}, and Proposition \ref{prop:field_ind}.
\end{proof}

\subsection{Technical Results}
\begin{prop}\label{prop:lp_bound_was}
For any $n\geq 0$, $s\in\{f,c\}$, $p\geq 1$ there exists a $C<+\infty$
such that for any $\varphi\in\mathcal{B}_b(\mathsf{X})$, $N\geq 1$
$$
\mathbb{E}[|[\eta_n^{N,s}-\eta_n^{s}](\varphi)|^{p}]^{1/p} \leq \frac{C\|\varphi\|}{\sqrt{N}}.
$$
\end{prop}
\begin{proof}
This can be proved easily, e.g.,~by the induction strategy in~\cite[Proposition 2.9]{dm2000}.
\end{proof}
\begin{lem}\label{lem:r_cont_ind}
For any $n>0$, $0\leq p\leq n-1$, $\varphi\in\mathcal{B}_b(\mathsf{X})$, $s\in\{f,c\}$,
$$
\sqrt{N}R_{p+1}^{N,s}(D_{p,n}^s(\varphi)) \rightarrow_{\mathbb{P}} 0.
$$
\end{lem}
\begin{proof}
By Proposition \ref{prop:lp_bound_was} $\eta_p^{N,s}(G_p)$ converges in probability to a well-defined limit. Hence, we need only show that
$$
\sqrt{N}\eta_p^{N,s}(D_{p,n}^s(\varphi))[\eta_p^{s}(G_p)-\eta_p^{N,s}(G_p)]
$$
will converge in probability to zero. By Cauchy-Schwarz:
$$
\sqrt{N}\mathbb{E}[|\eta_p^{N,s}(D_{p,n}^s(\varphi))[\eta_p^{s}(G_p)-\eta_p^{N,s}(G_p)]|] \leq \sqrt{N}\mathbb{E}[|\eta_p^{N,s}(D_{p,n}^s(\varphi))|^2]^{1/2}\mathbb{E}[|[\eta_p^{s}(G_p)-\eta_p^{N,s}(G_p)]|^2]^{1/2}.
$$
Applying Proposition \ref{prop:lp_bound_was} it easily follows that there is a finite constant $C<+\infty$ that does not depend upon $N$ such that
$$
\sqrt{N}\mathbb{E}[|\eta_p^{N,s}(D_{p,n}^s(\varphi))[\eta_p^{s}(G_p)-\eta_p^{N,s}(G_p)]|] \leq \frac{C}{\sqrt{N}}.
$$
This bound allows one to easily conclude the following lemma.
\end{proof}
\begin{lem}\label{lem:ind1}
For any $p\geq 0$, $\varphi\in\mathcal{B}_b(\mathsf{X})$, $s\in\{f,c\}$
\begin{eqnarray*}
\mathbb{E}[V_p^{N,s}(\varphi)] & = &  0 \\
\lim_{N\rightarrow+\infty}\mathbb{E}[V_p^{N,s}(\varphi)^2] & = & \eta_p^s((\varphi-\eta_p^s(\varphi))^2).
\end{eqnarray*}
\end{lem}
\begin{proof}
$\mathbb{E}[V_p^{N,s}(\varphi)] =  0$ follows immediately from the expression, so we focus on the second property.
\begin{eqnarray*}
\mathbb{E}[V_p^{N,s}(\varphi)^2] & = &  \frac{1}{N}\sum_{i=1}^N \mathbb{E}[(\varphi(X_p^{i,s})-\Phi^s_p(\eta_{p-1}^{N,s})(\varphi))^2] \\
& = & \mathbb{E}[\varphi(X_p^{1,s})^2] - \mathbb{E}[\Phi^s_p(\eta_{p-1}^{N,s})(\varphi)^2]
\end{eqnarray*}
$\Phi^s_p(\eta_{p-1}^{N,s})(\varphi)$ is a bounded random quantity and moreover by Proposition \ref{prop:lp_bound_was} it converges in probability
to $\eta_p^s(\varphi)$. Hence, by~\cite[Theorem 25.12]{bill}, $\lim_{N\rightarrow+\infty}\mathbb{E}[\Phi^s_p(\eta_{p-1}^{N,s})(\varphi)^2] = \eta_p^s(\varphi)^2$.
Hence, we consider 
$$
\mathbb{E}[\varphi(X_p^{1,s})^2] = \mathbb{E}[\frac{1}{N}\sum_{i=1}^N \varphi(X_p^{i,s})^2 - \eta_p^s(\varphi^2)] + \eta_p^s(\varphi^2).
$$
By Jensen
$$
\mathbb{E}[\frac{1}{N}\sum_{i=1}^N \varphi(X_p^{i,s})^2 - \eta_p^s(\varphi^2)] \leq \mathbb{E}[|\frac{1}{N}\sum_{i=1}^N \varphi(X_p^{i,s})^2 - \eta_p^s(\varphi^2)|^2]^{1/2}
$$
and hence we conclude via Proposition \ref{prop:lp_bound_was} that 
$$
\lim_{N\rightarrow+\infty}\mathbb{E}[\varphi(X_p^{1,s})^2] = \eta_p^s(\varphi^2)
$$
and the result thus follows.
\end{proof}
\begin{lem}\label{lem:cdf_conv_was}
For any $n\geq 0$, $s\in\{f,c\}$, $x\in\mathsf{X}$
$$
F_{\overline{\eta}_n^{N,s}}(x) \rightarrow_{\mathbb{P}} F_{\overline{\eta}_n^{s}}(x).
$$
\end{lem}
\begin{proof}
Follows immediately from Proposition \ref{prop:lp_bound_was}.
\end{proof}
\begin{lem}\label{lem:cdf_inv_conv_was}
For any $n\geq 0$, $s\in\{f,c\}$, $x\in[0,1]$
$$
F_{\overline{\eta}_n^{N,s}}^{-1}(x) \rightarrow_{\mathbb{P}} F_{\overline{\eta}_n^{s}}^{-1}(x).
$$
\end{lem}
\begin{proof}
We will show that $F_{\overline{\eta}_n^{N,s}}^{-1}(x)\Rightarrow F_{\overline{\eta}_n^{s}}^{-1}(x)$, which will allow us to conclude.
Let $x\in[0,1]$, $t\in\mathsf{X}$ be fixed, then we have
$$
\mathbb{P}(F_{\overline{\eta}_n^{N,s}}^{-1}(x)\leq t) = \mathbb{P}(F_{\overline{\eta}_n^{N,s}}(t)\geq x) = 1- \mathbb{P}(F_{\overline{\eta}_n^{N,s}}(t)\leq x).
$$
If $t$ is such that $F_{\overline{\eta}_n^{s}}(t)>x$ then by Lemma \ref{lem:cdf_conv_was}
$$
\lim_{N\rightarrow\infty}\mathbb{P}(F_{\overline{\eta}_n^{N,s}}^{-1}(x)\leq t) = 1
$$
and, respectively, if $t$ is such that $F_{\overline{\eta}_n^{s}}(t)<x$, then by Lemma~\ref{lem:cdf_conv_was}
$$
\lim_{N\rightarrow\infty}\mathbb{P}(F_{\overline{\eta}_n^{N,s}}^{-1}(x)\leq t) = 0.
$$
Hence, $F_{\overline{\eta}_n^{N,s}}^{-1}(x)\Rightarrow F_{\overline{\eta}_n^{s}}^{-1}(x)$, which completes the proof.
\end{proof}
\begin{lem}\label{lem:cond_exp_was}
Suppose that $\check{M}_n$ is Feller for every $n\geq 1$. Then for any $n\geq 1$, $\varphi\in\mathcal{C}_b(\mathsf{X}\times\mathsf{X})$:
$$
\lim_{N\rightarrow\infty}\mathbb{E}[\check{\Phi}_n^W(\overline{\eta}_{n-1}^{N,f},\overline{\eta}_{n-1}^{N,c})(\varphi)] =  \check{\Phi}_n^W(\overline{\eta}_{n-1}^{f},\overline{\eta}_{n-1}^{c})(\varphi).
$$
\end{lem}
\begin{proof}
$$
\check{\Phi}_n^W(\overline{\eta}_{n-1}^{N,f},\overline{\eta}_{n-1}^{N,c})(\varphi) =  \int_{0}^1 \check{M}_n(\varphi)(F_{\overline{\eta}_{n-1}^{N,f}}^{-1}(w),F_{\overline{\eta}_{n-1}^{N,c}}^{-1}(w)) dw.
$$
By Lemma \ref{lem:cdf_inv_conv_was} and the Feller property of $\check{M}_n$
$$
\check{M}_n(\varphi)(F_{\overline{\eta}_{n-1}^{N,f}}^{-1}(w),F_{\overline{\eta}_{n-1}^{N,c}}^{-1}(w)) \rightarrow_{\mathbb{P}} \check{M}_n(\varphi)(F_{\overline{\eta}_{n-1}^{f}}^{-1}(w),F_{\overline{\eta}_{n-1}^{c}}^{-1}(w))
$$
and hence the proof is completed via \cite[Theorem 25.12]{bill}.
\end{proof}
\begin{lem}\label{lem:was2}
Suppose that $\check{M}_n$ is Feller for every $n\geq 1$. Then for any $p\geq 0$, $\varphi,\psi\in\mathcal{C}_b(\mathsf{X})$, 
$$
\lim_{N\rightarrow+\infty}\mathbb{E}[V_p^{N,f}(\varphi)V_p^{N,c}(\psi)] = \check{\eta}_p^W(\varphi\otimes \psi) - \eta_p^f(\varphi)\eta_p^c(\psi).
$$
\end{lem}
\begin{proof}
We have
\begin{eqnarray*}
\mathbb{E}[V_p^{N,f}(\varphi)V_p^{N,c}(\psi)] & = & N\Big(\mathbb{E}[\eta_p^{N,f}(\varphi)\eta_p^{N,c}(\psi)]-\mathbb{E}[\eta_p^{N,f}(\varphi)\Phi_p^c(\eta_{p-1}^{N,c})(\psi)]- \\ & & 
\mathbb{E}[\Phi_p^f(\eta_{p-1}^{N,f})(\varphi)\eta_p^{N,c}(\psi)] + \mathbb{E}[\Phi_p^f(\eta_{p-1}^{N,f})(\varphi)\Phi_p^c(\eta_{p-1}^{N,c})(\psi)]
\Big) \\
& = &  N\Big(\mathbb{E}[\eta_p^{N,f}(\varphi)\eta_p^{N,c}(\psi)] - \mathbb{E}[\Phi_p^f(\eta_{p-1}^{N,f})(\varphi)\Phi_p^c(\eta_{p-1}^{N,c})(\psi)]
\Big) .
\end{eqnarray*}
Now
\begin{eqnarray*}
N\mathbb{E}[\eta_p^{N,f}(\varphi)\eta_p^{N,c}(\psi)] & = & \frac{1}{N}\sum_{i=1}^N\sum_{j=1}^N\mathbb{E}[\varphi(X_p^{i,f})\psi(X_p^{i,c})] \\
& = & \mathbb{E}[\check{\Phi}_p^W(\overline{\eta}_{n-1}^{N,f},\overline{\eta}_{n-1}^{N,c})(\varphi\otimes\psi)] + (N-1)\mathbb{E}[\Phi_p^f(\eta_{p-1}^{N,f})(\varphi)\Phi_p^c(\eta_{p-1}^{N,c})(\psi)].
\end{eqnarray*}
Thus,
$$
\mathbb{E}[V_p^{N,f}(\varphi)V_p^{N,c}(\psi)] = \mathbb{E}[\check{\Phi}_p^W(\overline{\eta}_{n-1}^{N,f},\overline{\eta}_{n-1}^{N,c})(\varphi\otimes\psi)] - \mathbb{E}[\Phi_p^f(\eta_{p-1}^{N,f})(\varphi)\Phi_p^c(\eta_{p-1}^{N,c})(\psi)].
$$
The result now follows by Lemma \ref{lem:cond_exp_was} and  Proposition \ref{prop:lp_bound_was} with \cite[Theorem 25.12]{bill}.
\end{proof}

Define for $p\geq 0$, $\varphi,\psi\in\mathcal{B}_b(\mathsf{X})$.
$$
V_p^N(\varphi,\psi) = V_p^{N,f}(\varphi) - V_p^{N,c}(\psi).
$$
\begin{prop}\label{prop:field_ind}
Let $n\geq 0$, then for any $(\varphi_0,\dots,\varphi_n) \in\mathcal{C}_b(\mathsf{X})^{n+1}$, $(\psi_0,\dots,\psi_n) \in\mathcal{C}_b(\mathsf{X})^{n+1}$, 
 $(V_0^N(\varphi_0,\psi_0),\dots,$ $V_n^N(\varphi_n,\psi_n))$ converges in distribution to a $(n+1)-$dimensional Gaussian random variable with zero mean and
diagonal covariance matrix, with the $p\in\{0,\dots,n\}$ diagonal entry
$$
\check{\eta}_p^W(\{(\varphi_p\otimes 1- 1\otimes \psi_p) - \check{\eta}_p^W(\varphi_p\otimes 1- 1\otimes \psi_p)\}^2).
$$
\end{prop}
\begin{proof}
This follows by using almost the same exposition and proofs as \cite{fk}, pp.~293-294, Theorem 9.3.1 and Corollary 9.3.1 and the results (of this paper) Lemma \ref{lem:ind1}
and Lemma \ref{lem:was2}. The proof is thus omitted.
\end{proof}

\section{Proofs for the Asympotic Variance}\label{app:av_prfs}

\begin{lem}\label{lem:av1}
Assume (A\ref{hyp:1}-\ref{hyp:2}). Then for any $n\geq 1$, $0\leq p < n$, $s\in\{f,c\}$, $\varphi\in \textrm{\emph{Lip}}(\mathsf{X})\cap\mathcal{B}_b(\mathsf{X})$,
$Q_{p,n}^s(\varphi)\in \textrm{\emph{Lip}}(\mathsf{X})\cap\mathcal{B}_b(\mathsf{X})$.
\end{lem}
\begin{proof}
The boundedness is clear, so we concentrate on the Lipschitz property.
The proof is by induction, starting with the case $p=n-1$. Throughout, $C$ is a constant whose value may change from line-to-line and may depend on $n,p,\varphi,G_n$, but
critically does not depend on $s\in\{f,c\}$. We have for any $(x,y)\in\mathsf{X}^2$
$$
|Q_{n}^s(\varphi)(x) - Q_{n}^s(\varphi)(y)| = |[G_{n-1}(x) - G_{n-1}(y)]M_{n}^s(\varphi)(x) + G_{n-1}(y)[M_{n}^s(\varphi)(x) -M_{n}^s(\varphi)(y) ]|.
$$
Applying the triangular inequality with (A\ref{hyp:1}) for the left term on the R.H.S.~and (A\ref{hyp:2}) for the other term yields
$$
|Q_{n}^s(\varphi)(x) - Q_{n}^s(\varphi)(y)| \leq C|x-y|.
$$
We assume the result for a given $p+1$ and consider $p$. Then for any $(x,y)\in\mathsf{X}^2$
\begin{eqnarray*}
|Q_{p,n}^s(\varphi)(x) - Q_{p,n}^s(\varphi)(y)| & = & |[G_{p}(x) - G_{p}(y)]M_{p+1}^s(Q_{p+1,n}^s(\varphi))(x) + \\ & & G_{p}(y)[M_{p+1}^s(Q_{p+1,n}^s(\varphi))(x) -M_{p+1}^s(Q_{p+1,n}^s(\varphi))(y) ]|.
\end{eqnarray*}
Then 
\begin{eqnarray*}
|Q_{p,n}^s(\varphi)(x) - Q_{p,n}^s(\varphi)(y)| & \leq & |G_{p}(x) - G_{p}(y)|M_{p+1}^s(Q_{p+1,n}^s(\varphi))(x) + \\ & &  \|G_p\||M_{p+1}^s(Q_{p+1,n}^s(\varphi))(x) -M_{p+1}^s(Q_{p+1,n}^s(\varphi))(y)|.
\end{eqnarray*}
Clearly by (A\ref{hyp:1})
$$
|G_{p}(x) - G_{p}(y)|M_{p+1}^s(Q_{p+1,n}^s(\varphi))(x) \leq C|x-y|.
$$
By the induction hypothesis and (A\ref{hyp:2})
$$
\|G_p\||M_{p+1}^s(Q_{p+1,n}^s(\varphi))(x) -M_{n}^s(Q_{p+1,n}^s(\varphi))(y)|
\leq C|x-y|
$$ 
and so one can easily conclude the proof from here.
\end{proof}
\begin{lem}\label{lem:av2}
Assume (A\ref{hyp:1}-\ref{hyp:2}). Then for any $n\geq 1$, $0\leq p < n$, $\varphi\in \textrm{\emph{Lip}}(\mathsf{X})\cap\mathcal{B}_b(\mathsf{X})$
there exists a $C<+\infty$ such that for any $x\in\mathsf{X}$
$$
|Q_{p,n}^f(\varphi)(x)-Q_{p,n}^c(\varphi)(x)| \leq C |\|M_n^{f,c}\||.
$$
\end{lem}
\begin{proof}
Throughout, $C$ is a constant whose value may change from line-to-line and may depend on $n,p,\varphi,G_n$.
We have the decomposition
$$
Q_{p,n}^f(\varphi)(x)-Q_{p,n}^c(\varphi)(x) = \sum_{k=p+1}^n [Q_{p,k-1}^c(Q_{k-1,n}^f(\varphi))(x)-Q_{p,k}^c(Q_{k,n}^f(\varphi))(x)].
$$
We only consider the summand, which is
$$
[Q_{p,k-1}^c(Q_{k-1,n}^f(\varphi))(x)-Q_{p,k}^c(Q_{k,n}^f(\varphi))(x)] = Q_{p,k-1}^c([Q_k^{f}-Q_k^{c}](Q^f_{k,n}(\varphi)))(x).
$$
By Lemma \ref{lem:av1} $Q^f_{k,n}(\varphi)\in\textrm{Lip}(\mathsf{X})\cap\mathcal{B}_b(\mathsf{X})$, so
\begin{eqnarray*}
|[Q_k^{f}-Q_k^{c}](Q^f_{k,n}(\varphi))(x)| & = & |G_{k-1}(x)[M_k^f-M_k^s](Q^f_{k,n}(\varphi))(x)| \\
& = & \Big|G_{k-1}(x)\|Q^f_{k,n}(\varphi)\| [M_k^f-M_k^s]\Big(\frac{Q^f_{k,n}(\varphi)}{\|Q^f_{k,n}(\varphi)\|}\Big)(x)\Big|\\
& \leq & C|\|M_k^f-M_k^s\|| \\
& \leq & C|\|M_n^{f,c}\||.
\end{eqnarray*}
Thus, it easily follows that
$$
|Q_{p,n}^f(\varphi)(x)-Q_{p,n}^c(\varphi)(x)| \leq C |\|M_n^{f,c}\||.
$$
\end{proof}

\begin{lem}\label{lem:av3}
Assume (A\ref{hyp:1}-\ref{hyp:3}). Then for any $n\geq 1$, $0\leq p < n$, $\varphi\in \textrm{\emph{Lip}}(\mathsf{X})\cap\mathcal{B}_b(\mathsf{X})$
there exist a $C<+\infty$ such that for any $(x,y)\in\mathsf{X}\times\mathsf{X}$
$$
\Big|\frac{Q_{p,n}^f(\varphi)(x)}{\eta_p^f(Q_{p,n}^f(1))}-\frac{Q_{p,n}^c(\varphi)(y)}{\eta_p^c(Q_{p,n}^c(1))}\Big| \leq
C\Big(|x-y|+\|\eta_p^f-\eta_p^c\|_{\textrm{\emph{tv}}}+|\|M_n^{f,c}\||\Big)
$$
where $C$ does not depend on $\eta_p^f,\eta_p^c$.
\end{lem}
\begin{proof}
Throughout, $C$ is a constant whose value may change from line-to-line and may depend on $n,p,\varphi,G_n$.
We have
$$
\frac{Q_{p,n}^f(\varphi)(x)}{\eta_p^f(Q_{p,n}^f(1))}-\frac{Q_{p,n}^c(\varphi)(y)}{\eta_p^c(Q_{p,n}^c(1))}
= 
$$
\begin{equation}\label{eq:av1}
\frac{1}{\eta_p^f(Q_{p,n}(1))}[Q_{p,n}^f(\varphi)(x)-Q_{p,n}^c(\varphi)(y)] + 
\frac{Q_{p,n}^c(\varphi)(y)}{\eta_p^f(Q_{p,n}^f(1))\eta_p^c(Q_{p,n}^c(1))}
[\eta_p^c(Q_{p,n}^c(1))-\eta_p^f(Q_{p,n}^f(1))]
\end{equation}
we deal with the two terms on the R.H.S.~of \eqref{eq:av1} separately. 

For the first term, by (A\ref{hyp:3}) $\eta_p^f(Q_{p,n}(1))^{-1}\leq C$ and then
$$
|Q_{p,n}^f(\varphi)(x)-Q_{p,n}^c(\varphi)(y)| \leq |Q_{p,n}^f(\varphi)(x)-Q_{p,n}^c(\varphi)(x)| + |Q_{p,n}^c(\varphi)(x)-Q_{p,n}^c(\varphi)(y)|.
$$
Applying Lemma \ref{lem:av1} for the second term on the R.H.S.~and Lemma \ref{lem:av2} for the first term
$$
|Q_{p,n}^f(\varphi)(x)-Q_{p,n}^c(\varphi)(y)| \leq C(|\|M_n^{f,c}\||+|x-y|).
$$
Hence, 
\begin{equation}\label{eq:av2}
\Big|\frac{1}{\eta_p^f(Q_{p,n}(1))}[Q_{p,n}^f(\varphi)(x)-Q_{p,n}^c(\varphi)(y)]\Big| \leq C(|\|M_n^{f,c}\||+|x-y|).
\end{equation}

For the second term on the R.H.S.~of \eqref{eq:av1}, by (A\ref{hyp:1}) and (A\ref{hyp:3})
$$
\frac{Q_{p,n}^c(\varphi)(y)}{\eta_p^f(Q_{p,n}^f(1))\eta_p^c(Q_{p,n}^c(1))} \leq C.
$$
Then 
$$
|\eta_p^c(Q_{p,n}^c(1))-\eta_p^f(Q_{p,n}^f(1))| \leq 
|[\eta_p^c-\eta_p^f](Q_{p,n}^c(1))| + |\eta_p^f(Q_{p,n}^c(1)-Q_{p,n}^f(1))|.
$$
Clearly,
$$
|[\eta_p^c-\eta_p^f](Q_{p,n}^c(1))| =  \|Q_{p,n}^c(1)\|\Big|[\eta_p^c-\eta_p^f]\Big(\frac{Q_{p,n}^c(1)}{\|Q_{p,n}^c(1)\|}\Big)\Big| \leq C\|\eta_p^f-\eta_p^c\|_{\textrm{tv}}.
$$
By Lemma \ref{lem:av2}
$$
|\eta_p^f(Q_{p,n}^c(1)-Q_{p,n}^f(1))| \leq  C |\|M_n^{f,c}\||.
$$
Thus, we have shown that
\begin{equation}\label{eq:av3}
\Big|\frac{Q_{p,n}^c(\varphi)(y)}{\eta_p^f(Q_{p,n}^f(1))\eta_p^c(Q_{p,n}^c(1))}
[\eta_p^c(Q_{p,n}^c(1))-\eta_p^f(Q_{p,n}^f(1))]\Big| \leq C
(\|\eta_p^f-\eta_p^c\|_{\textrm{tv}} + |\|M_n^{f,c}\||).
\end{equation}
Combining \eqref{eq:av1} with \eqref{eq:av2} and  \eqref{eq:av3} completes the proof.
\end{proof}
\begin{lem}\label{lem:av4}
Assume (A\ref{hyp:1}-\ref{hyp:3}). Then for any $n\geq 1$, $0\leq p < n$, $\varphi\in \textrm{\emph{Lip}}(\mathsf{X})\cap\mathcal{B}_b(\mathsf{X})$,
there exists a $C<+\infty$ such that for any $(x,y)\in\mathsf{X}\times\mathsf{X}$
$$
|D_{p,n}^f(\varphi)(x)-D_{p,n}^c(\varphi)(y)| \leq
C\Big(|x-y|+\|\eta_p^f-\eta_p^c\|_{\textrm{\emph{tv}}}+\|\eta_n^f-\eta_n^c\|_{\textrm{\emph{tv}}}+|\|M_n^{f,c}\||\Big)
$$
where $C$ does not depend on $\eta_p^f,\eta_p^c,\eta_n^f,\eta_n^c$.
\end{lem}
\begin{proof}
We have
$$
D_{p,n}^f(\varphi)(x)-D_{p,n}^c(\varphi)(y) = 
\Big(\frac{Q_{p,n}^f(\varphi)(x)}{\eta_p^f(Q_{p,n}^f(1))}-\frac{Q_{p,n}^c(\varphi)(y)}{\eta_p^c(Q_{p,n}^c(1))}\Big)(1 - \eta_n^f(\varphi)) 
+ \frac{Q_{p,n}^c(\varphi)(y)}{\eta_p^c(Q_{p,n}^c(1))}[\eta_n^f-\eta_n^c](\varphi).
$$
The proof then easily follows by Lemma \ref{lem:av3} and standard calculations.
\end{proof}

\section{Figures}

\begin{figure}[H]
\centering
  \includegraphics[width=0.475\textwidth]{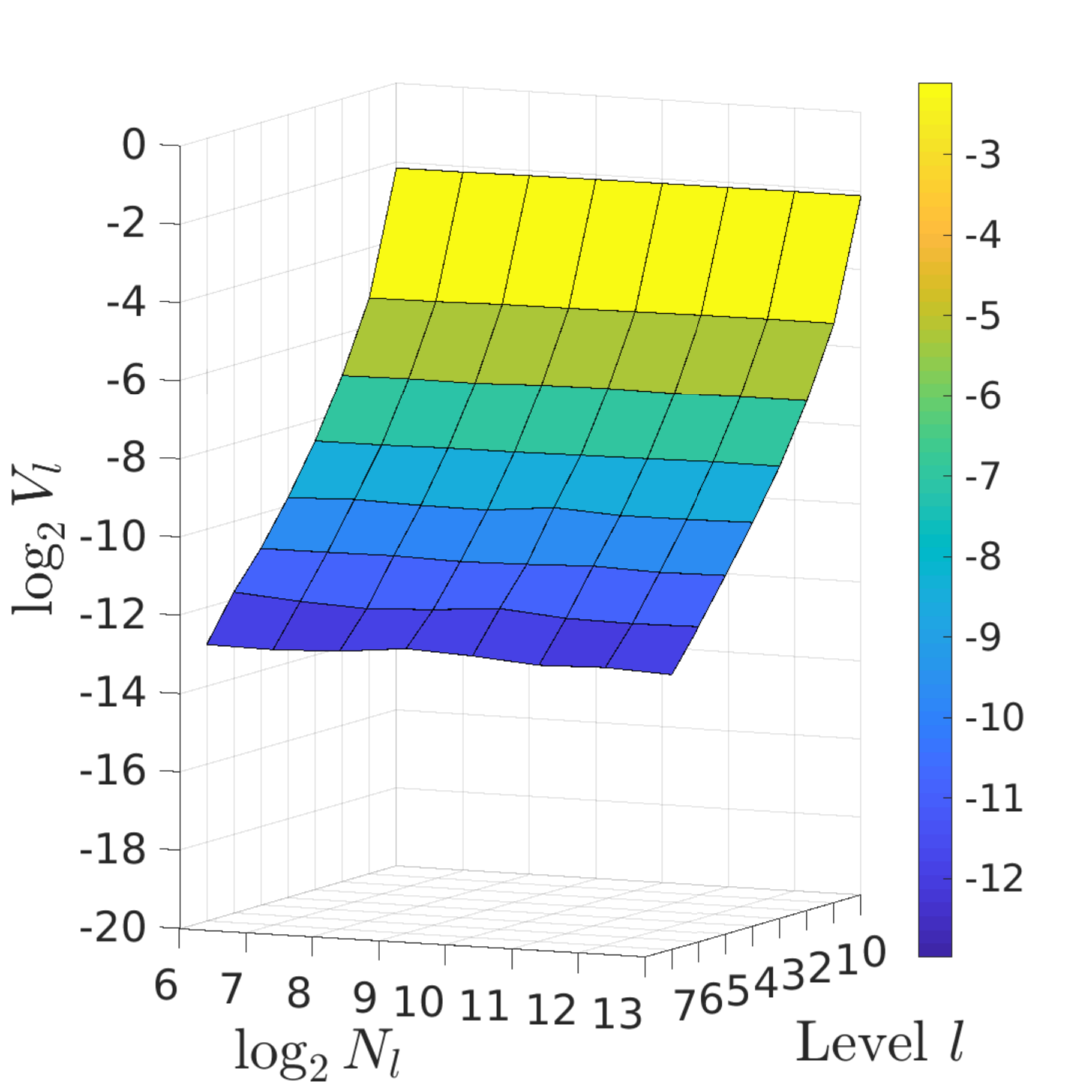}
  \hfill
  \includegraphics[width=0.475\textwidth]{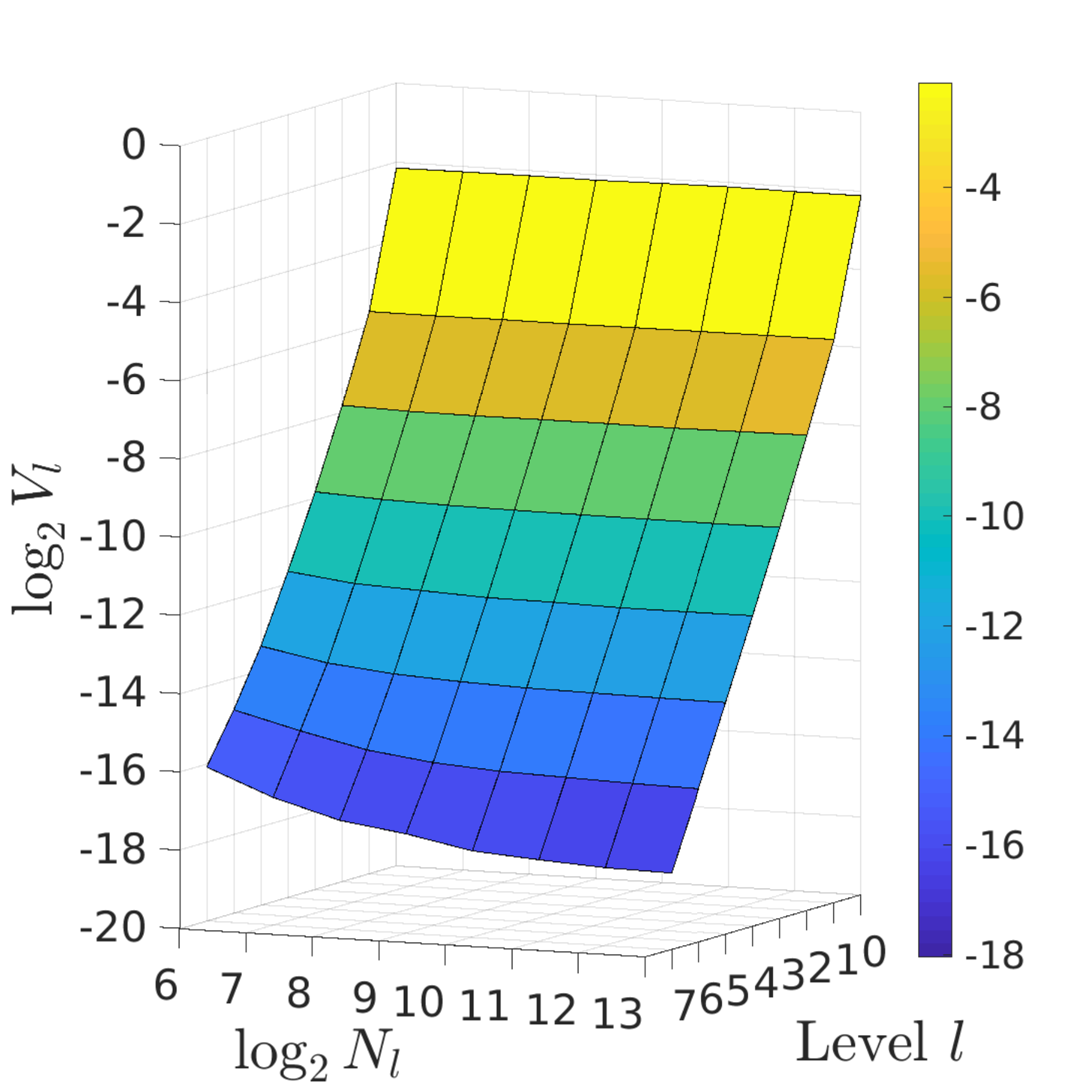}
  \caption[]{\textbf{OU. Variance convegence study.} 
    The variance estimate $V_l$, as defined in~\eqref{eq:variance}, as a function 
    of the discretization level, $l$, and the number of particles, $N_l$.
    On the left the results using Algorithm~\ref{alg:coupled} are shown, 
    and on the right those using Algorithm~\ref{alg:CDFcoupling}.} 
  \label{fig:ou_variance}
\end{figure}

\begin{figure}[H]
\centering
  \includegraphics[width=90mm]{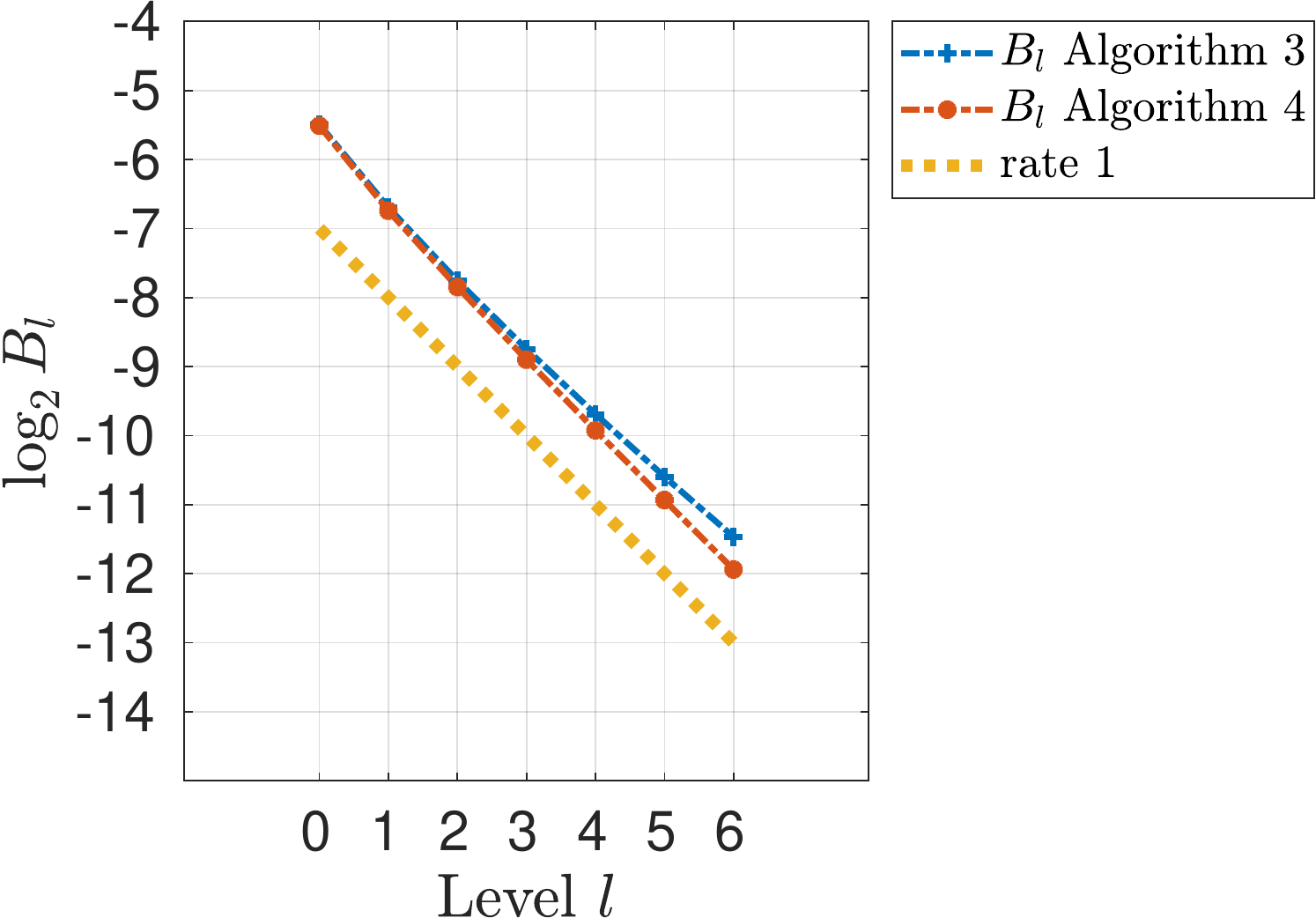} 
    \caption[]{\textbf{OU. Bias estimate.}
    The estimated bias of the filter distribution expectation, $B_l$, as defined 
    in~\eqref{eq:bias}, as a function of the discretization level, $l$.
}   
 \label{fig:ou_bias}
\end{figure}

\begin{figure}[H]
\centering
  \includegraphics[width=0.42\textwidth]{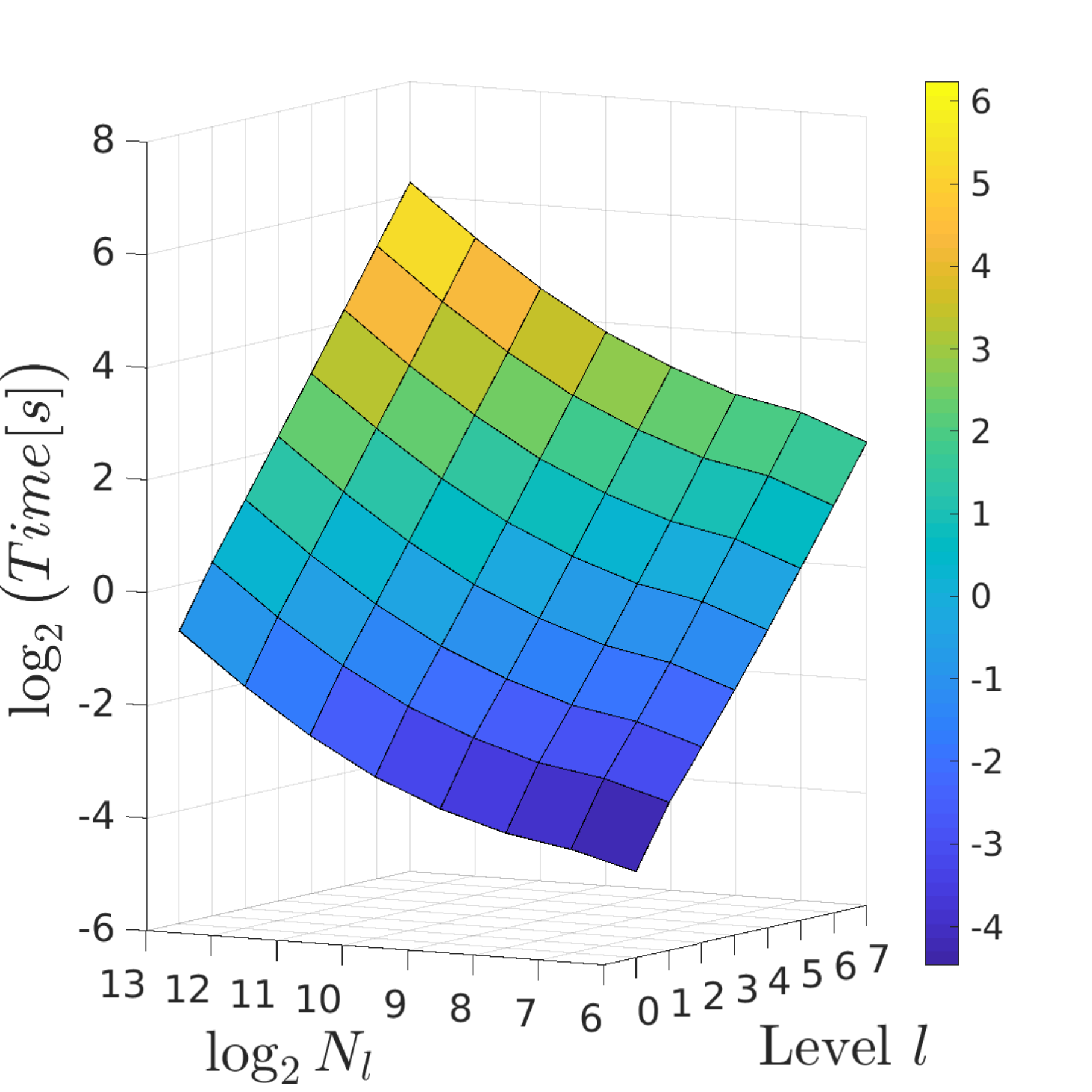} 
  \hfill
  \includegraphics[width=0.55\textwidth]{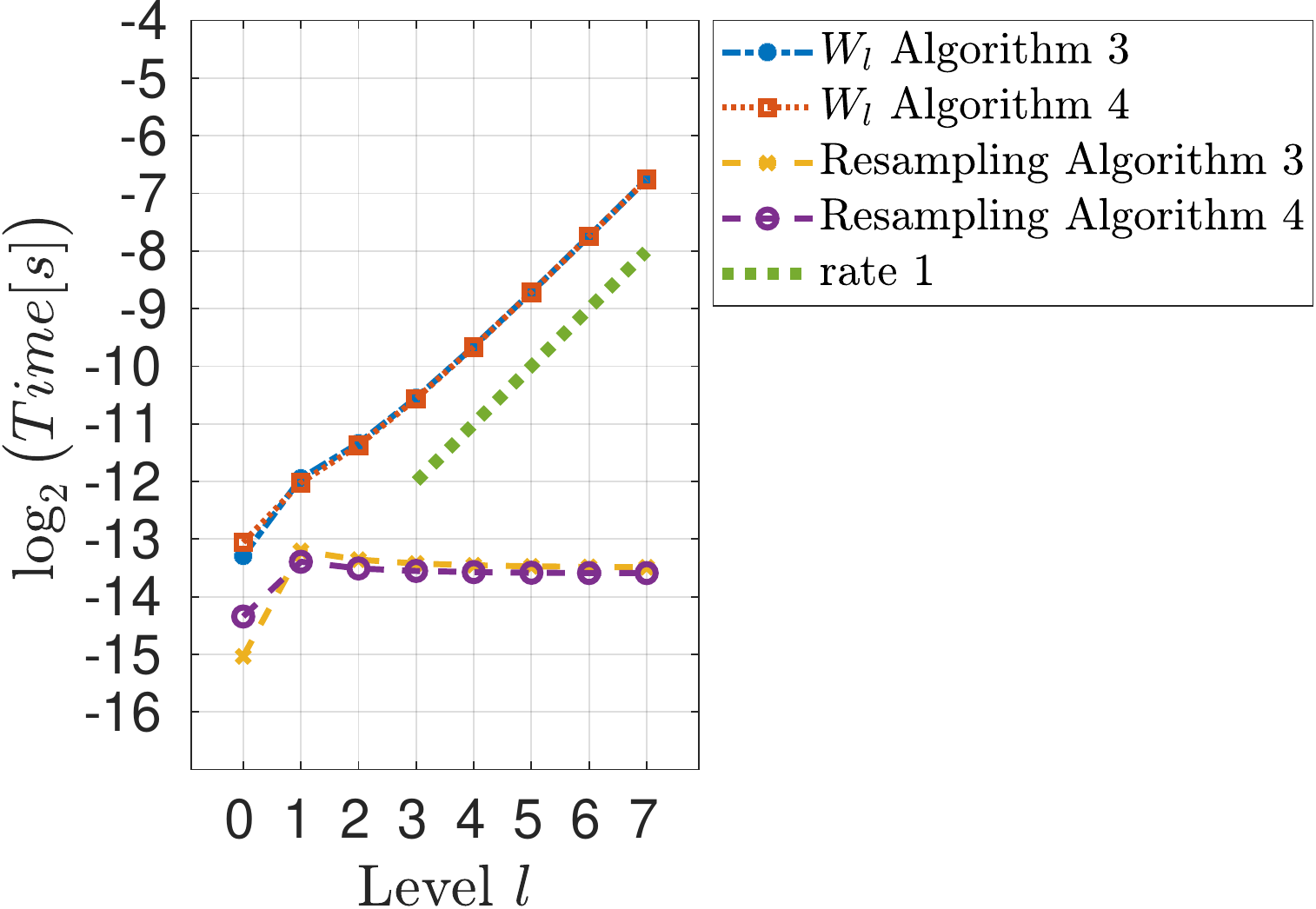}    
  \caption[]{\textbf{OU. Computational time.} 
    Measured wall clock time in seconds for the Euler-Maruyama time stepping as a 
    function of number of coupled particles $N_l$ and level $l$ (left). 
    Total measured wall clock time per particle, as a function of 
    the level $l$, with the measured time of the resampling alone for comparison. 
    Time per particle is based on the measured time for $N=2^{13}$ 
    particles (right).} 
  \label{fig:ou_costs}
\end{figure}

\begin{figure}[H]
\centering
  \includegraphics[width=75mm]{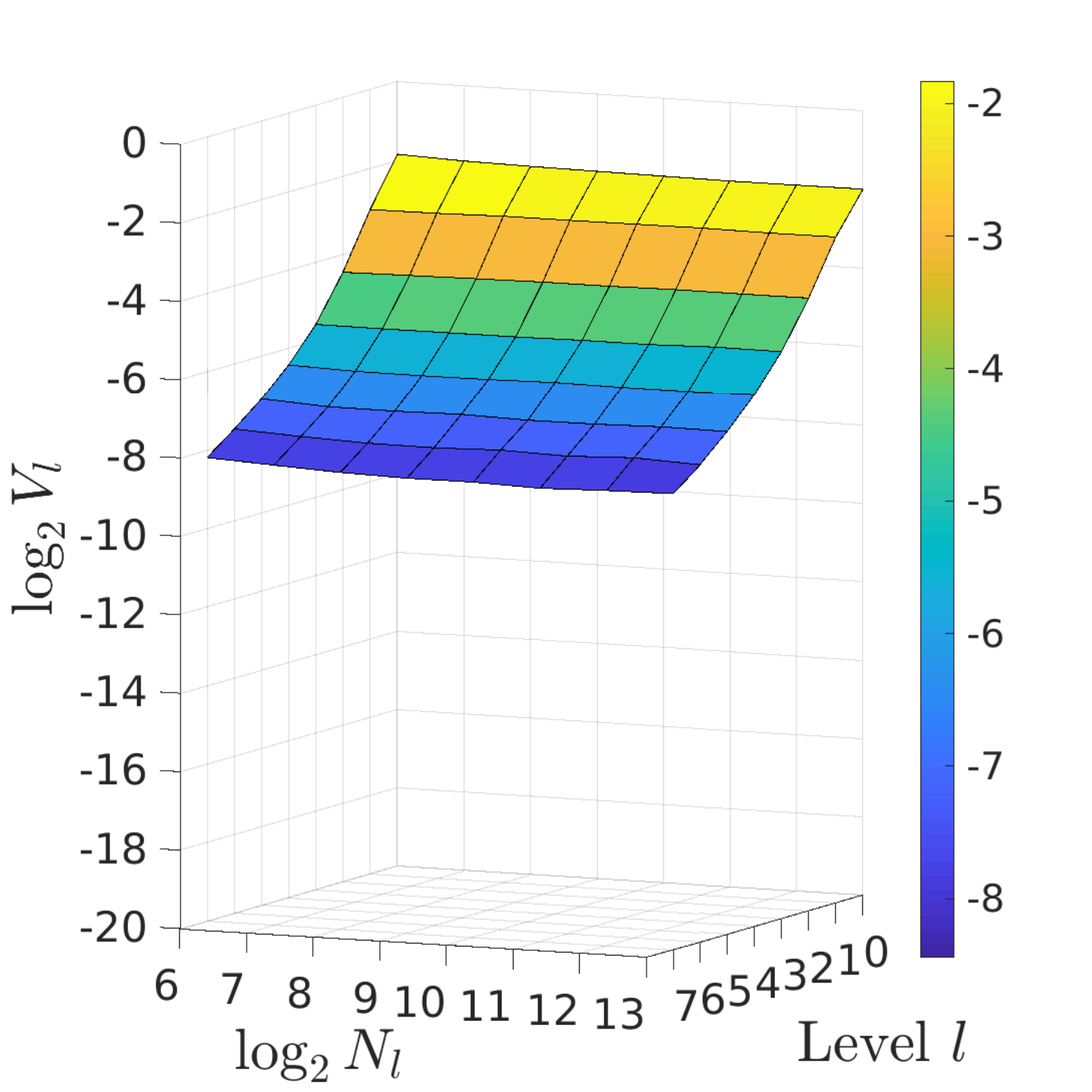}
  \hfill
  \includegraphics[width=75mm]{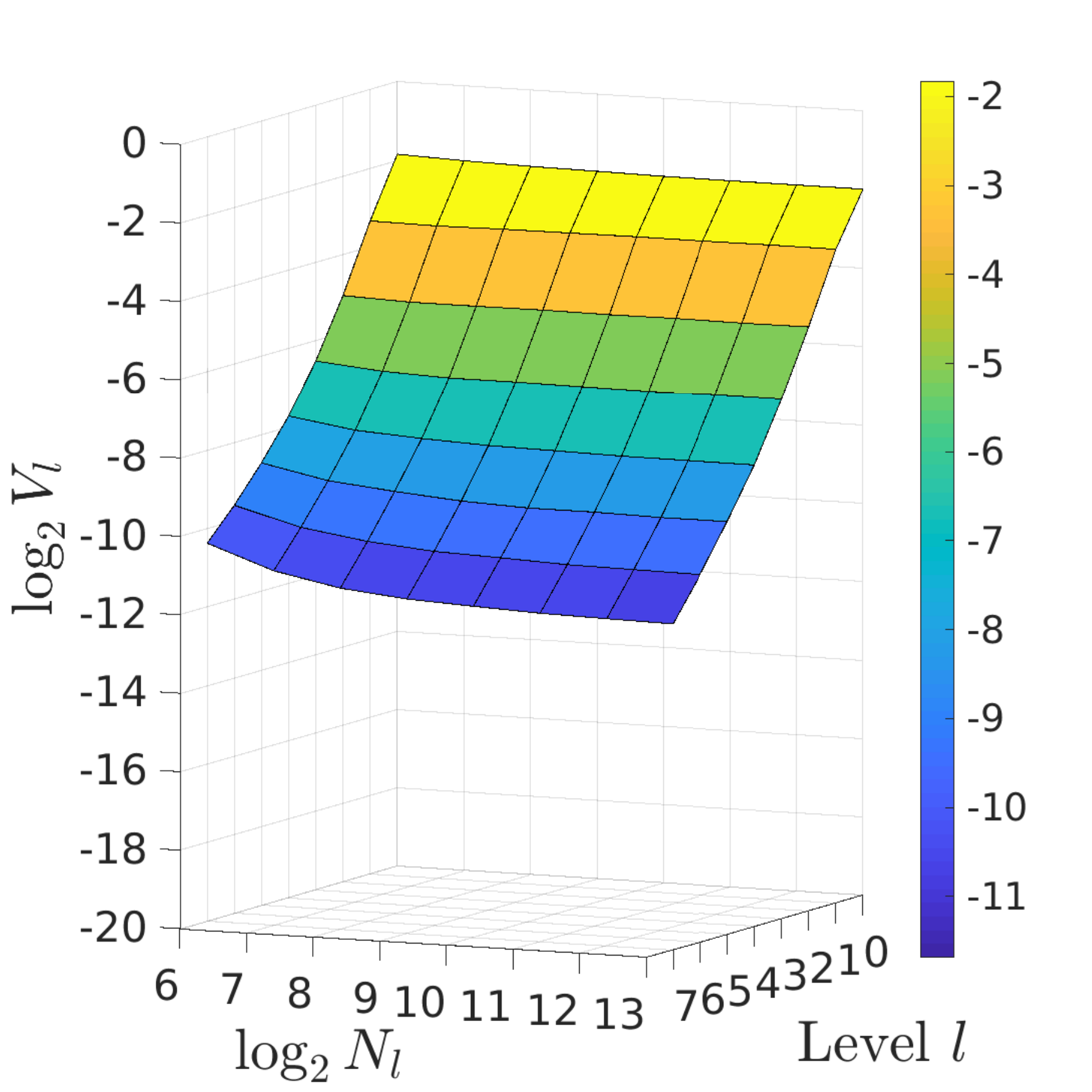}
  \caption[]{\textbf{NDT. Variance convegence study.} 
  The variance estimate $V_l$, as defined in~\eqref{eq:variance}, as a function 
  of the discretization level, $l$, and the number of particles, $N_l$.
  On the left the result using Algorithm~\ref{alg:coupled} are displayed, and on 
  the right those using Algorithm~\ref{alg:CDFcoupling}.} 
  \label{fig:ndt_variance}
\end{figure}

\begin{figure}[H]
\centering
  \includegraphics[width=90mm]{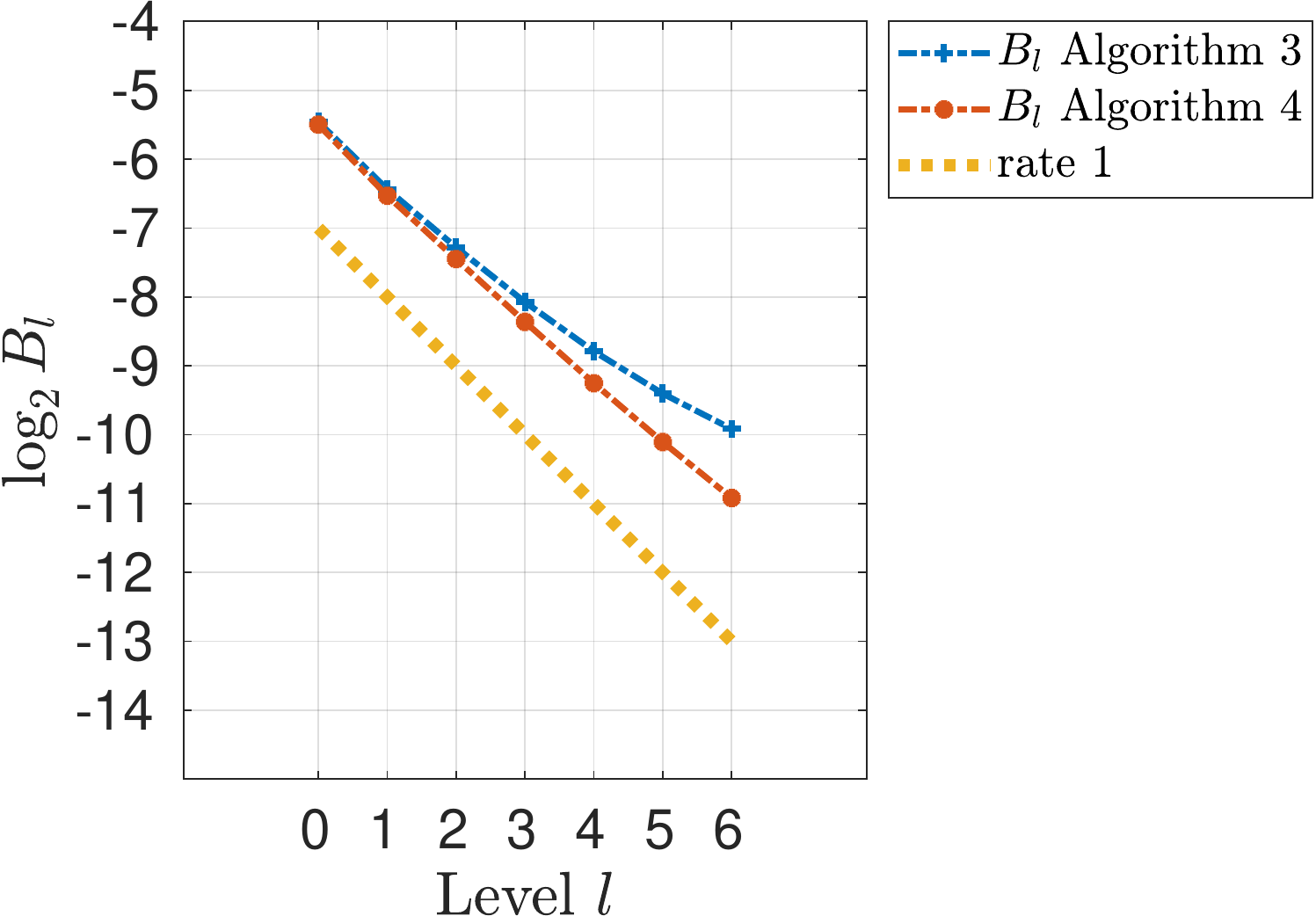}  
  \caption[]{\textbf{NDT. Bias estimate.}
    The estimated bias of the filter distribution expectation, $B_l$, as defined 
    in~\eqref{eq:bias}, as a function of the discretization level, $l$.}   
  \label{fig:ndt_bias}
\end{figure}

\begin{figure}[H]
\centering
  \includegraphics[width=0.55\textwidth]{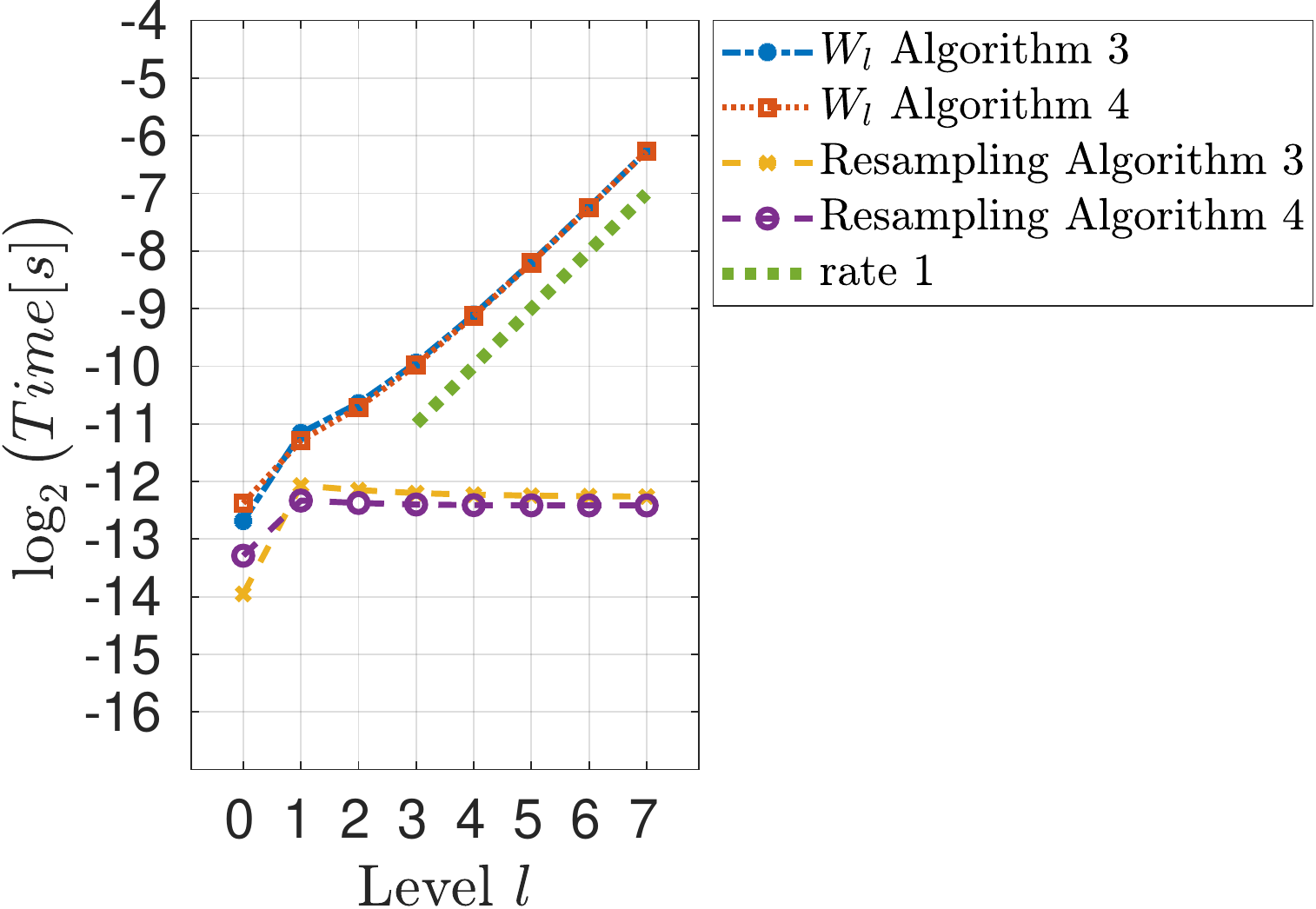}
\caption[]{\textbf{NDT. Computational time.} 
  Measured wall clock time per particle in seconds as a function of levels $l$, 
  based on the measured time for $N=2^{13}$ particles.}
\label{fig:ndt_work}
\end{figure}

\begin{figure}[H]
  \centering
  \includegraphics[width=15cm]{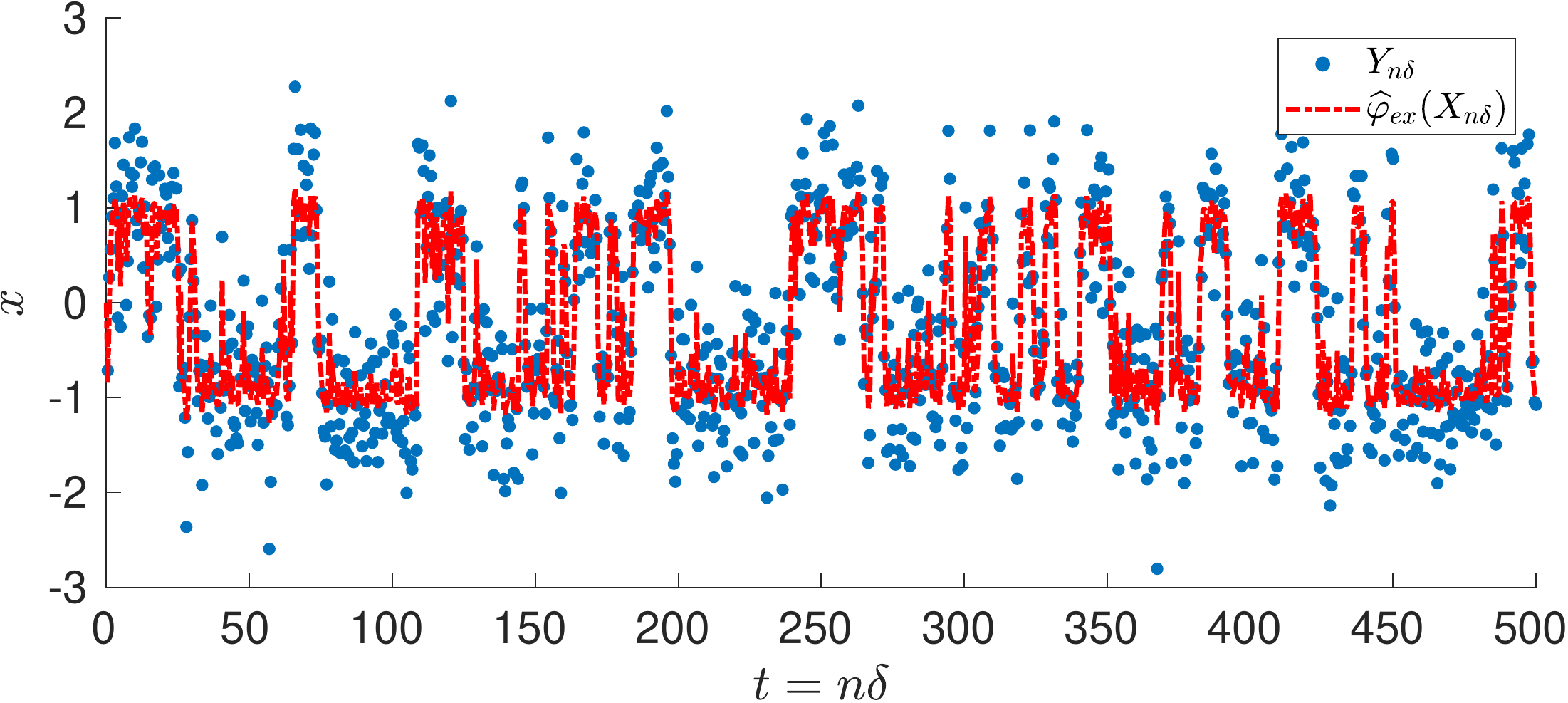}
  \caption[]{\textbf{DW. Sample trajectory of data and filter expectation.} Here,
    $Y_{n\delta}$ are synthetic observations, from a numerical approximation of
    the underlying dynamics~\eqref{eq:DW_dynamics} with additive measurement 
    noise. A highly accurate reference solution, $\widehat{\varphi}_{ex}(X_T)$, 
    of the corresponding filter expectation, computed by numerical approximation 
    of the Fokker-Planck equation is included.
  } 
\label{fig:dw_dynamics}
\end{figure}

\begin{figure}[H]
\begin{subfigure}{.5\textwidth}
\centering%  
  \includegraphics[width=\textwidth]{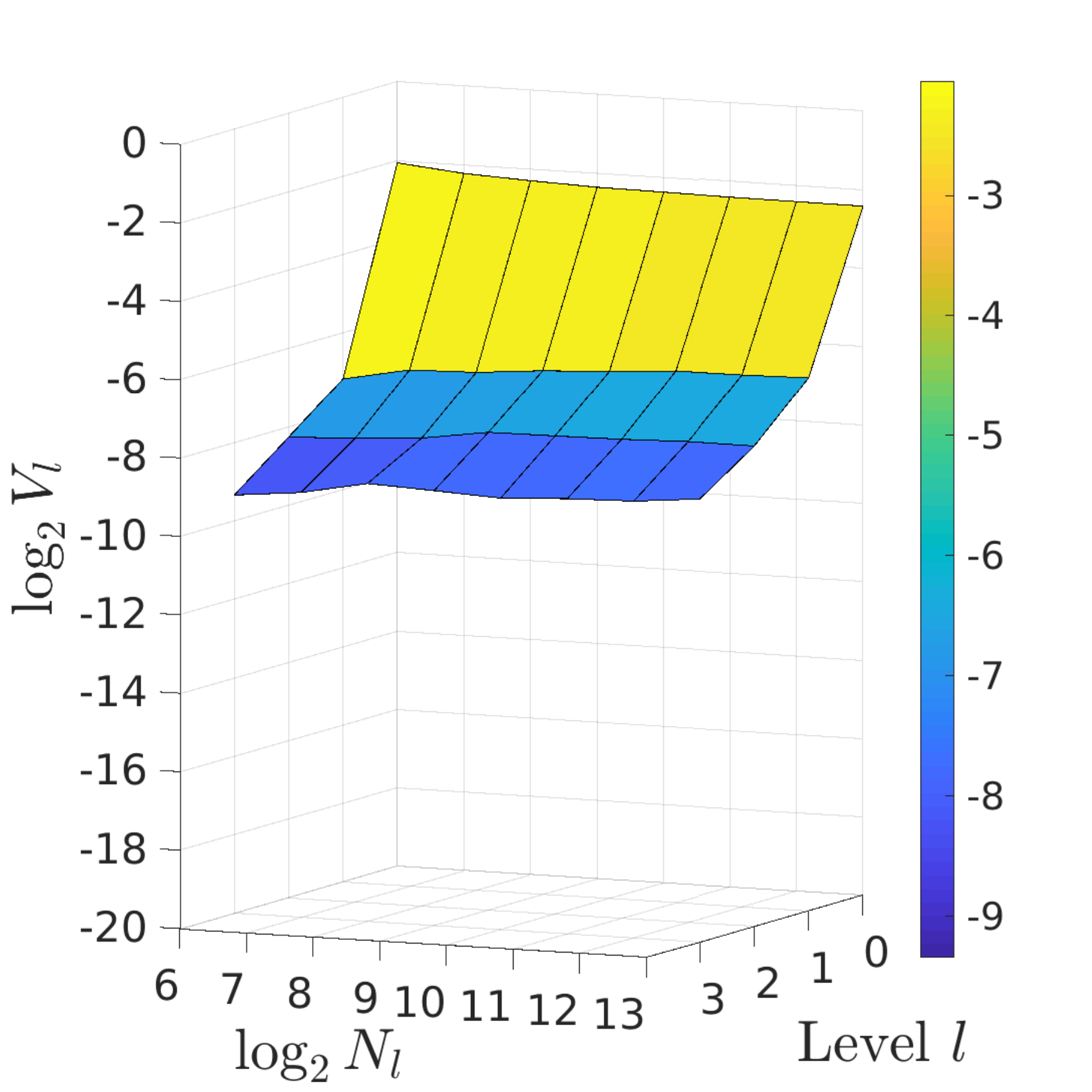}
\caption{\textbf{Algorithm~\ref{alg:coupled} with change of measure}}
\label{fig:s1}
\end{subfigure}
~
\begin{subfigure}{.5\textwidth}
\centering%  
  \includegraphics[width=\textwidth]{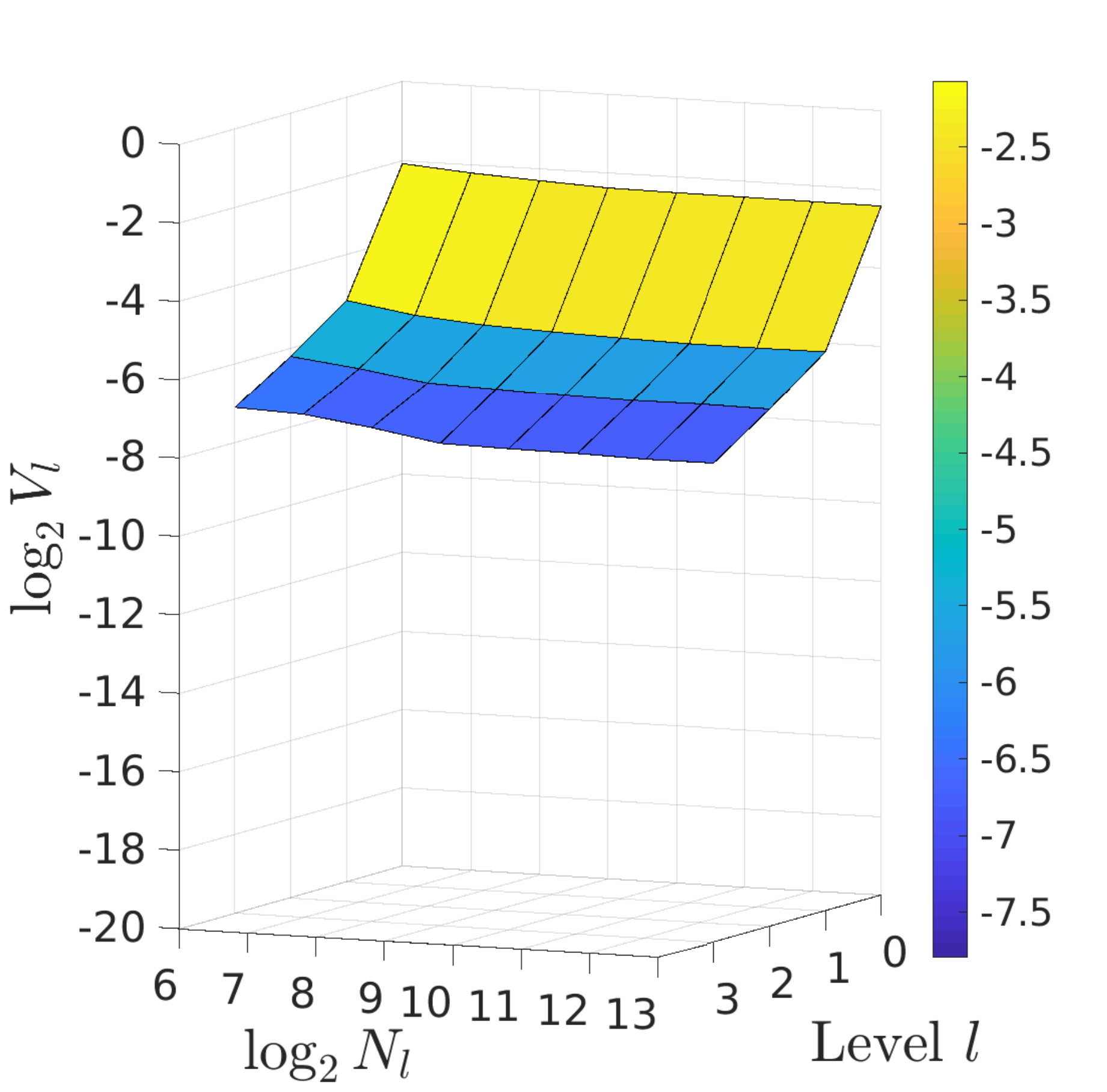}  \\
\caption{\textbf{Algorithm~\ref{alg:coupled}}}
\label{fig:s2}
\end{subfigure}
  
\begin{subfigure}{.5\textwidth}
\centering%    
  \includegraphics[width=\textwidth]{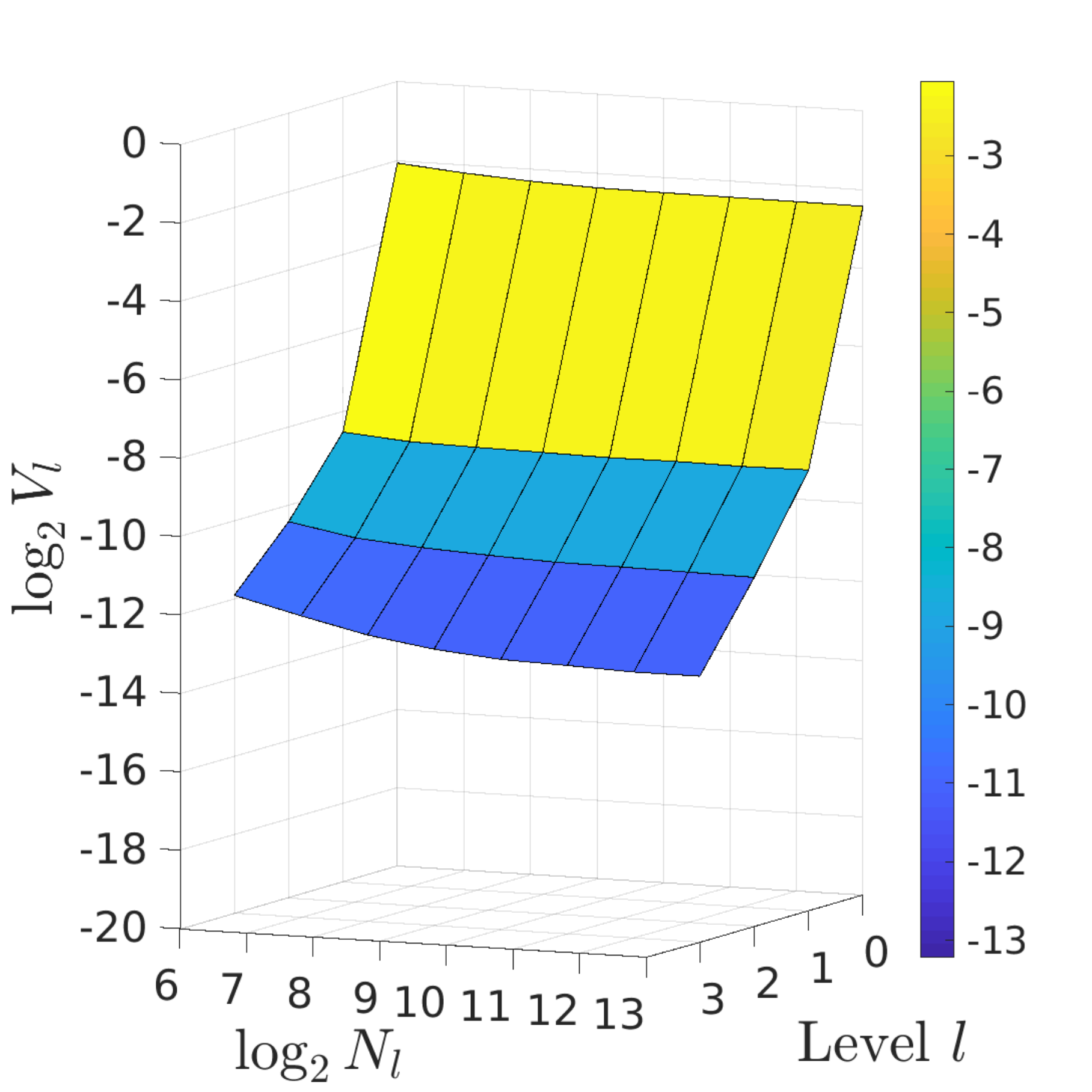}
\caption{\textbf{Algorithm~\ref{alg:CDFcoupling} with change of measure}}
\label{fig:s3}
\end{subfigure}
~  
\begin{subfigure}{.5\textwidth}
\centering%    
  \includegraphics[width=\textwidth]{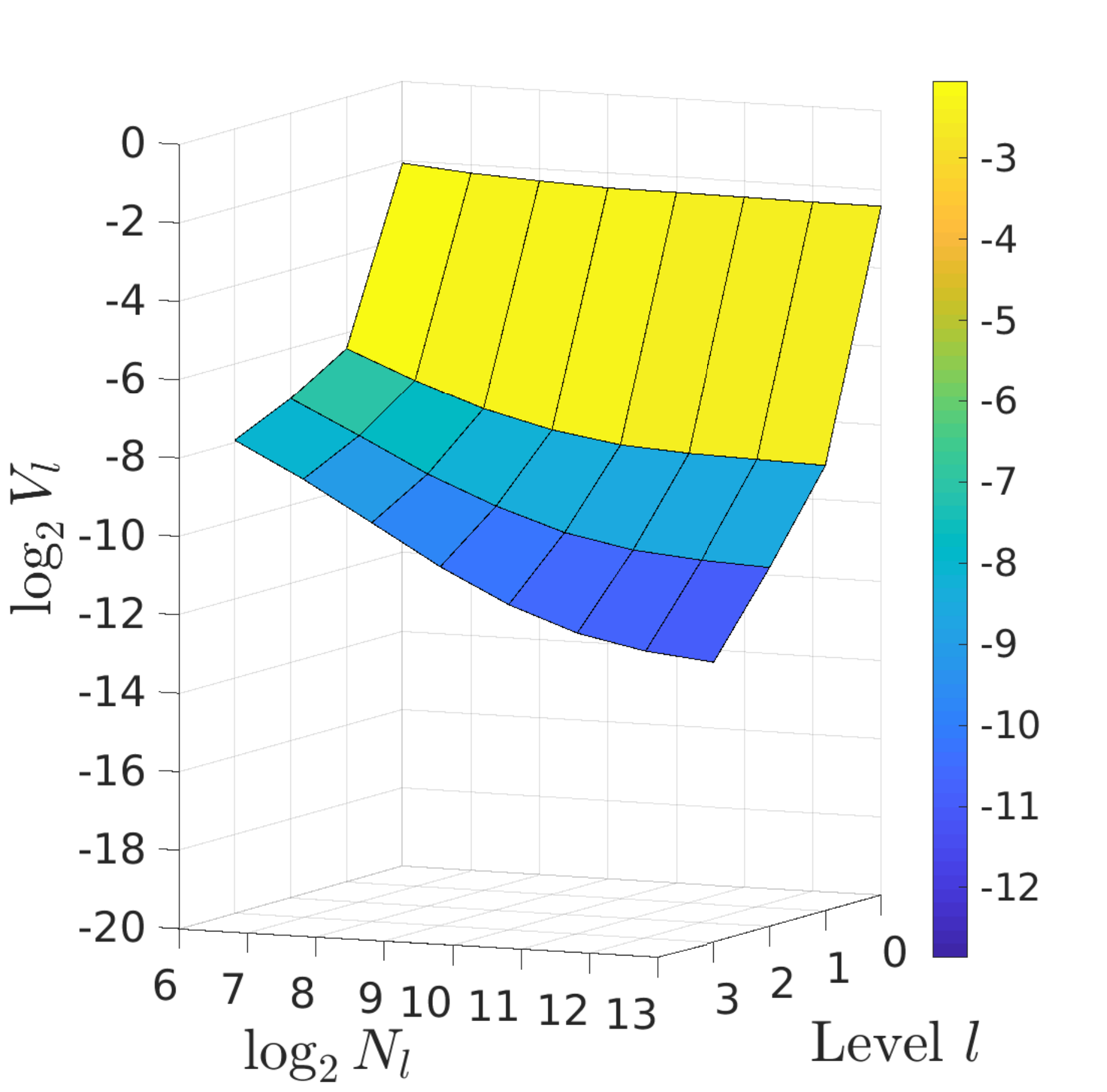}  
  \caption{\textbf{Algorithm~\ref{alg:CDFcoupling}}}
\label{fig:s4}
\end{subfigure}

  \caption[]{\textbf{DW. Variance convegence study.} 
  The variance estimate $V_l$, as defined in~\eqref{eq:variance}, as a function 
  of the discretization level, $l$, and the number of particles, $N_l$. Based  
  on the dynamics~\eqref{eq:DW_dynamics} with and without the change of 
  measure in Appendix~\ref{app:changeofmeasure}.} 
  \label{fig:dw_variance}
\end{figure}

\begin{figure}[H]
\centering
  \includegraphics[width=90mm]{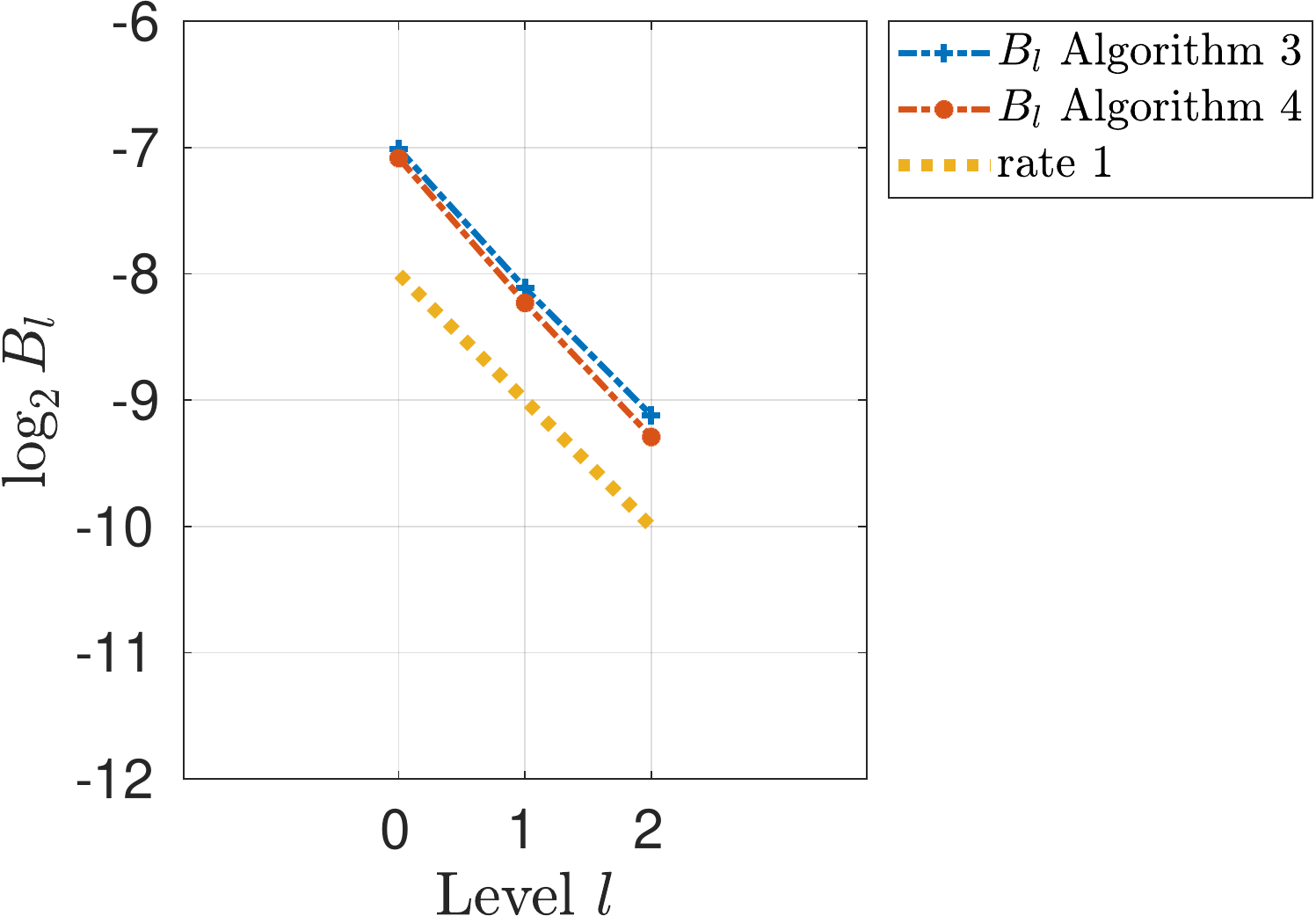}  \label{fig:}
    \caption[]{\textbf{DW. Bias estimate.} The estimated bias of the filter 
      distribution expectation, $B_l$, as defined in~\eqref{eq:bias}, as a 
      function of the discretization level, $l$. Based on the 
      dynamics~\eqref{eq:DW_dynamics} with the change of measure in 
      Appendix~\ref{app:changeofmeasure}.} 
     \label{fig:dw_bias}  
\end{figure}

\begin{figure}[H]
\centering
  \includegraphics[width=0.42\textwidth]{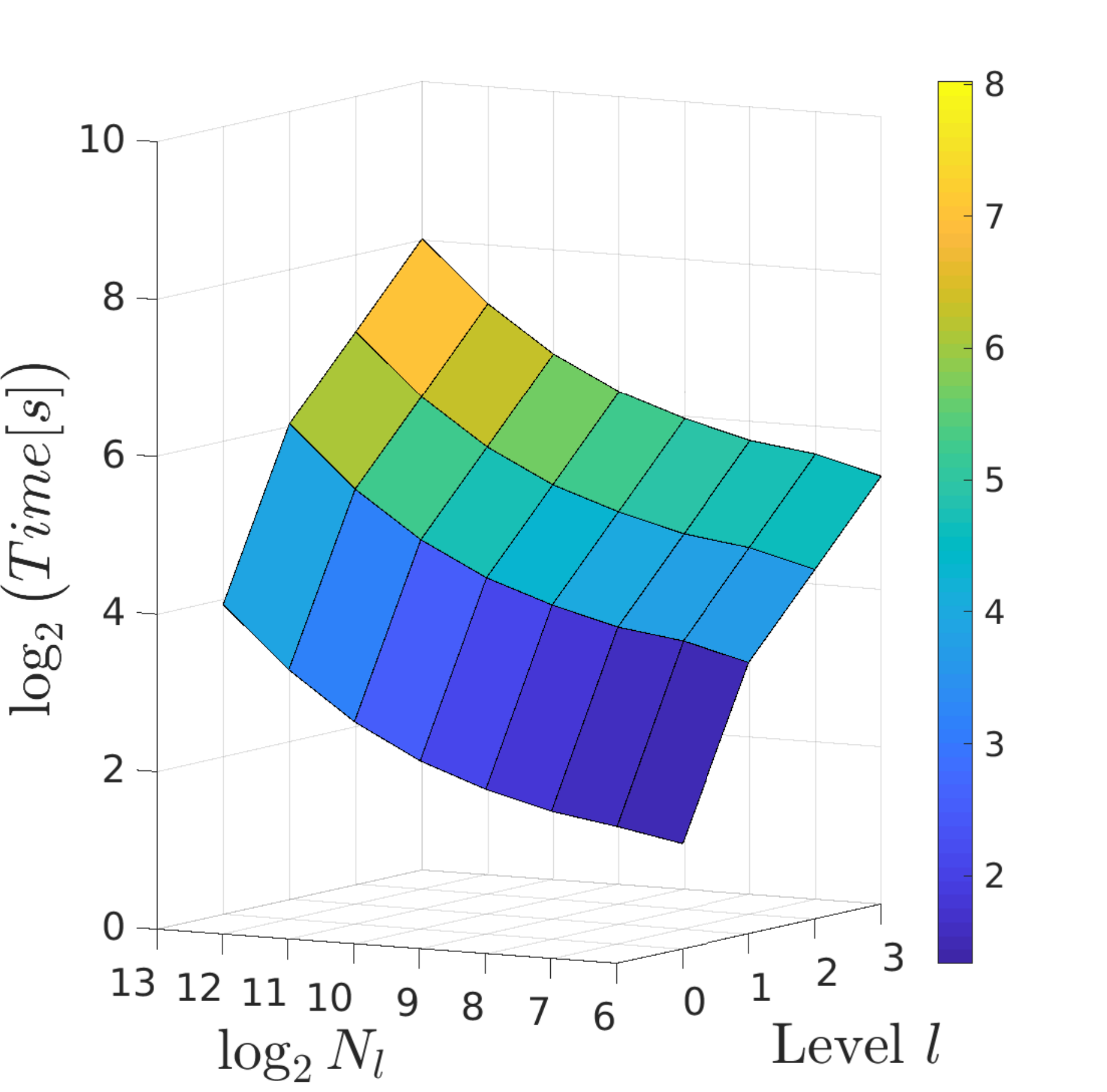} 
  \hfill
  \includegraphics[width=0.55\textwidth]{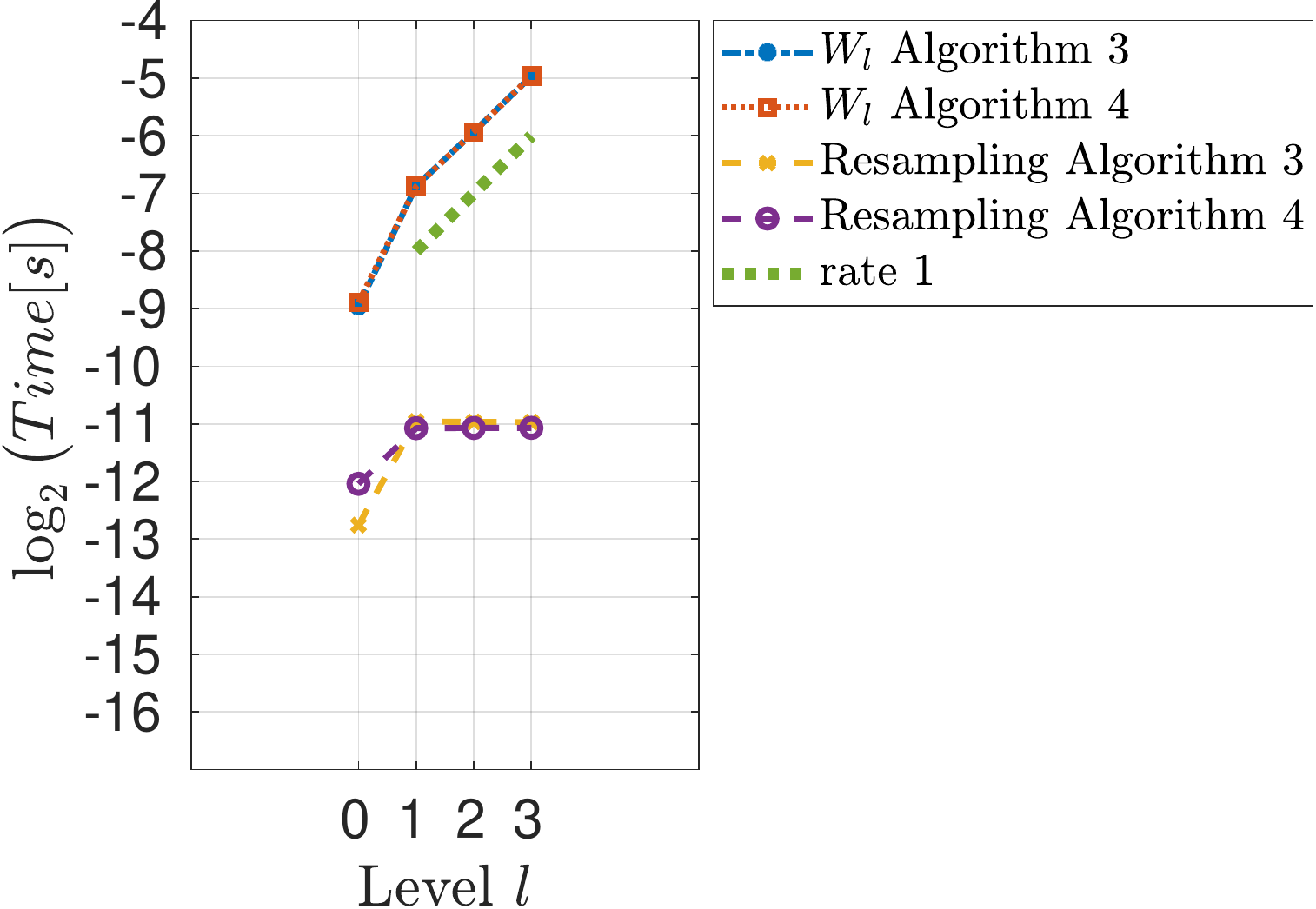}    
  \caption[]{\textbf{DW. Computational time.} 
    Measured wall clock time in seconds for the Euler-Maruyama time stepping as a 
    function of number of coupled particles $N_l$ and level $l$ (left). 
    Total measured wall clock time per particle, as a function of 
    the level $l$, with the measured time of the resampling alone for comparison. 
    Time per particle is based on the measured time for $N=2^{13}$ 
    particles (right). Based on the dynamics~\eqref{eq:DW_dynamics} with the 
    change of measure in Appendix~\ref{app:changeofmeasure}.} 
  \label{fig:dw_work}
\end{figure}

\begin{figure}[H]
\centering
  \includegraphics[width=0.45\textwidth]{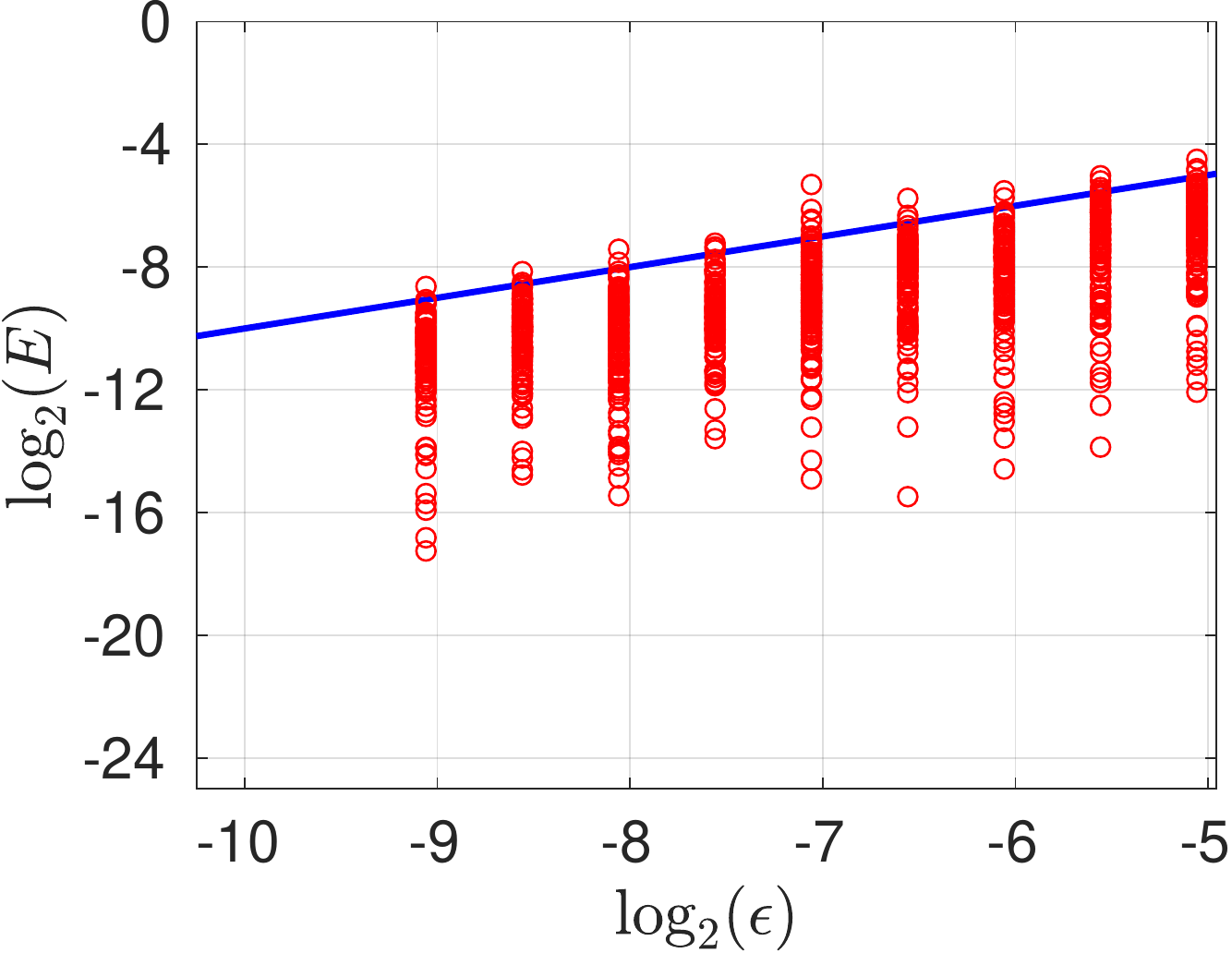} 
  \hfill
  \includegraphics[width=0.45\textwidth]{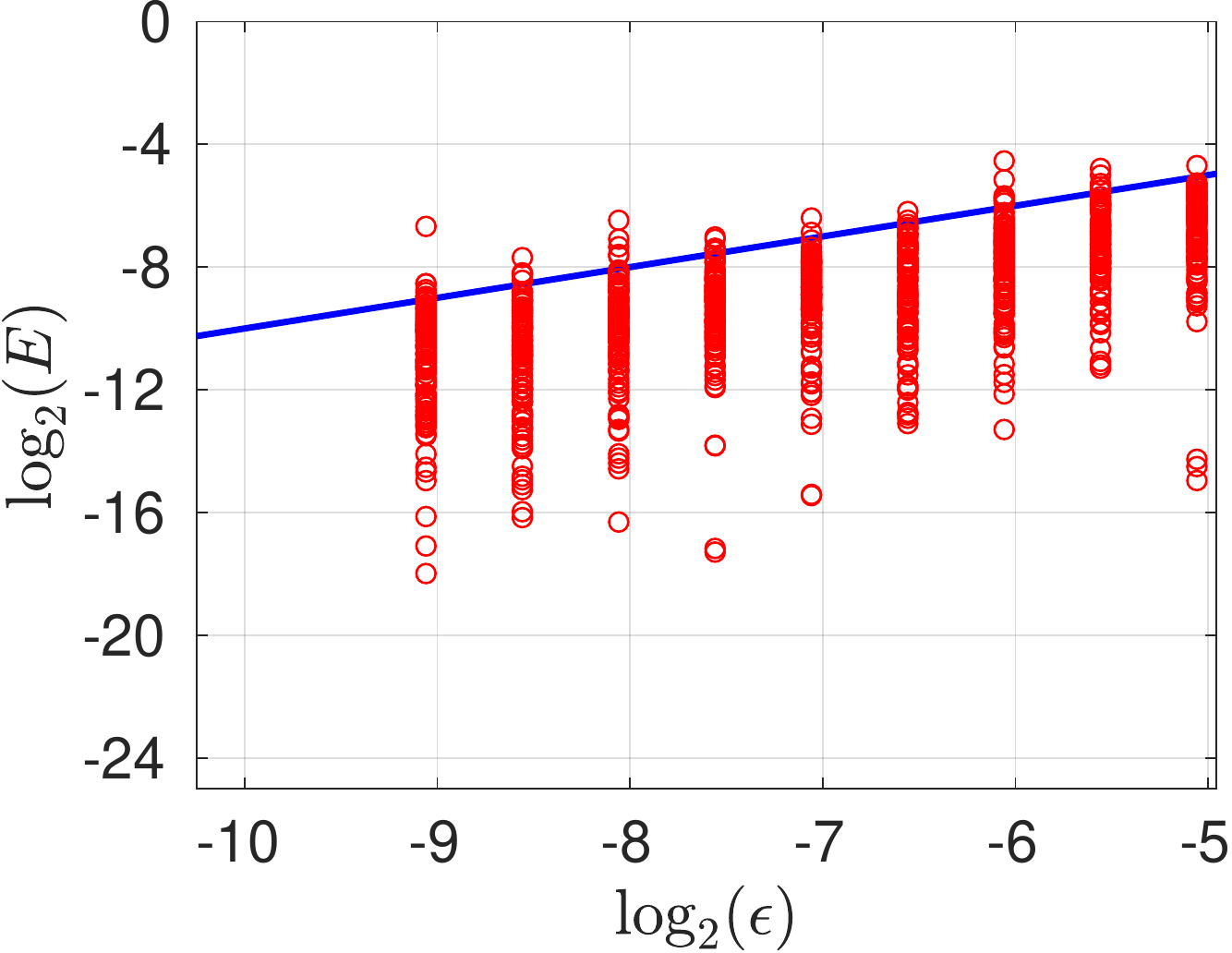} 
  \caption[]{\textbf{OU. Accuracy.} Errors of the expected values of the filter 
    distributions compared to reference solution for tolerances $\epsilon_{k}$, 
    $k=\lbrace 0,\dots,8\rbrace$, in~\eqref{eq:tol}. The error $E$ is estimated 
    by $E_{k,i}$ for $i\in \lbrace 1,\dots,100\rbrace$, defined in~\eqref{eq:E_k_i}, 
    and shown as red circles. 
    On the left, the results using Algorithm~\ref{alg:coupled} are displayed, and 
    on the right, those using Algorithm~\ref{alg:CDFcoupling}.}  
  \label{fig:ou_tolvserror}
\end{figure}

\begin{figure}[H]
\centering  
  \includegraphics[width=0.55\textwidth]{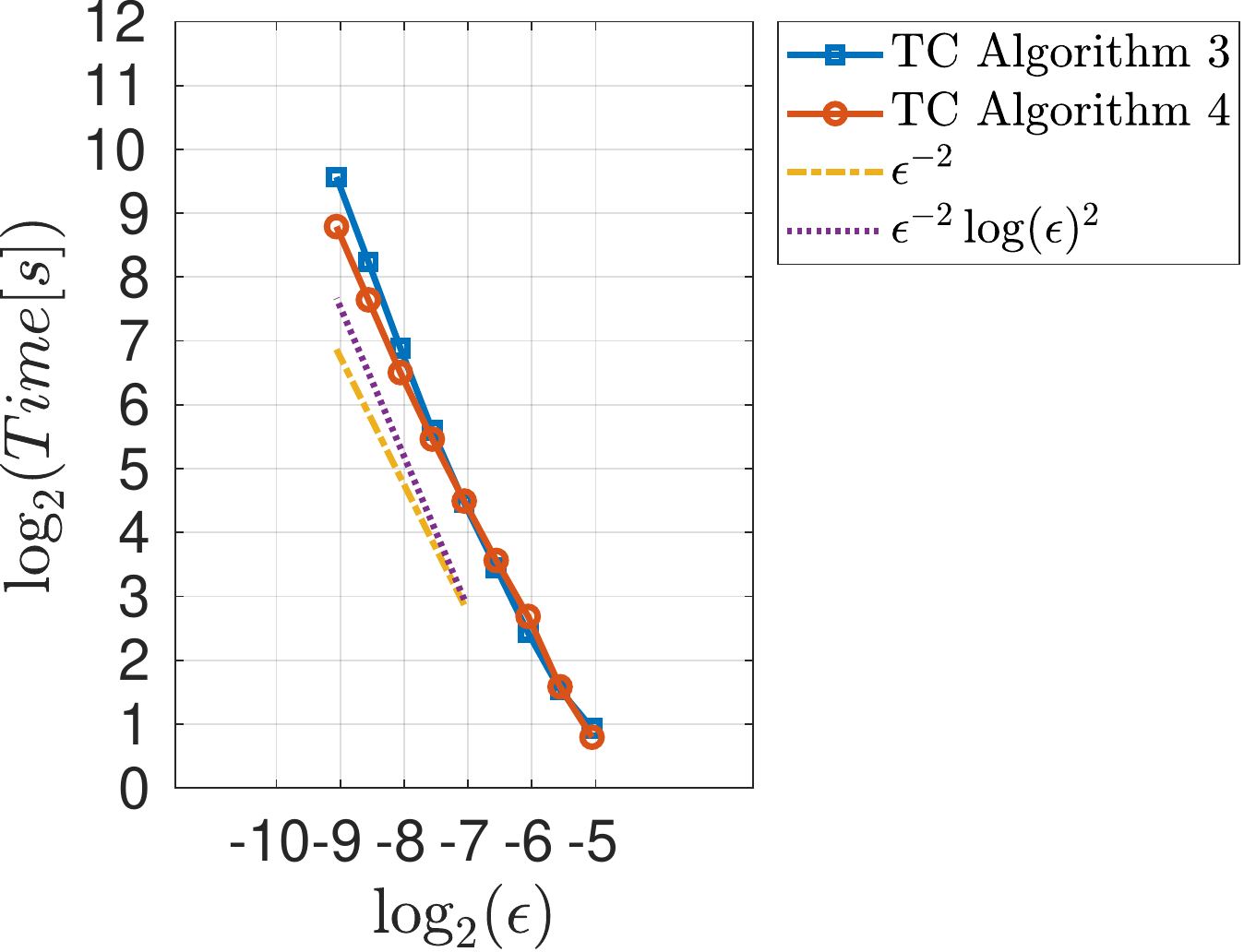}   
\caption[]{\textbf{OU. Complexity.} Total computational time of the optimal 
  filter, measured as wall clock time in seconds, as a function of the tolerance 
  $\epsilon$.} 
  \label{fig:ou_actualwork}
\end{figure}

\begin{figure}[H]
\centering
  \includegraphics[width=0.45\textwidth]{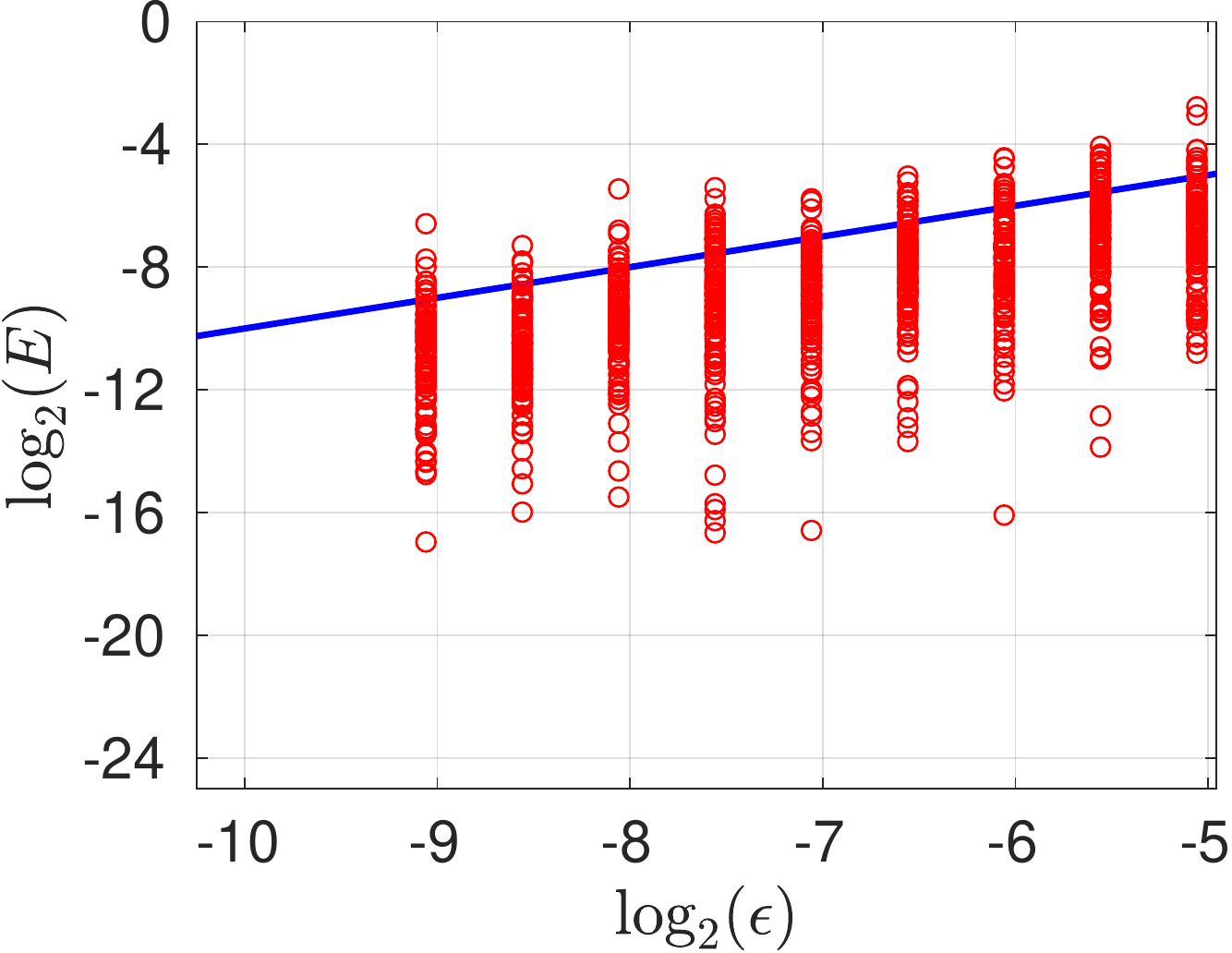} 
\hfill
  \includegraphics[width=0.45\textwidth]{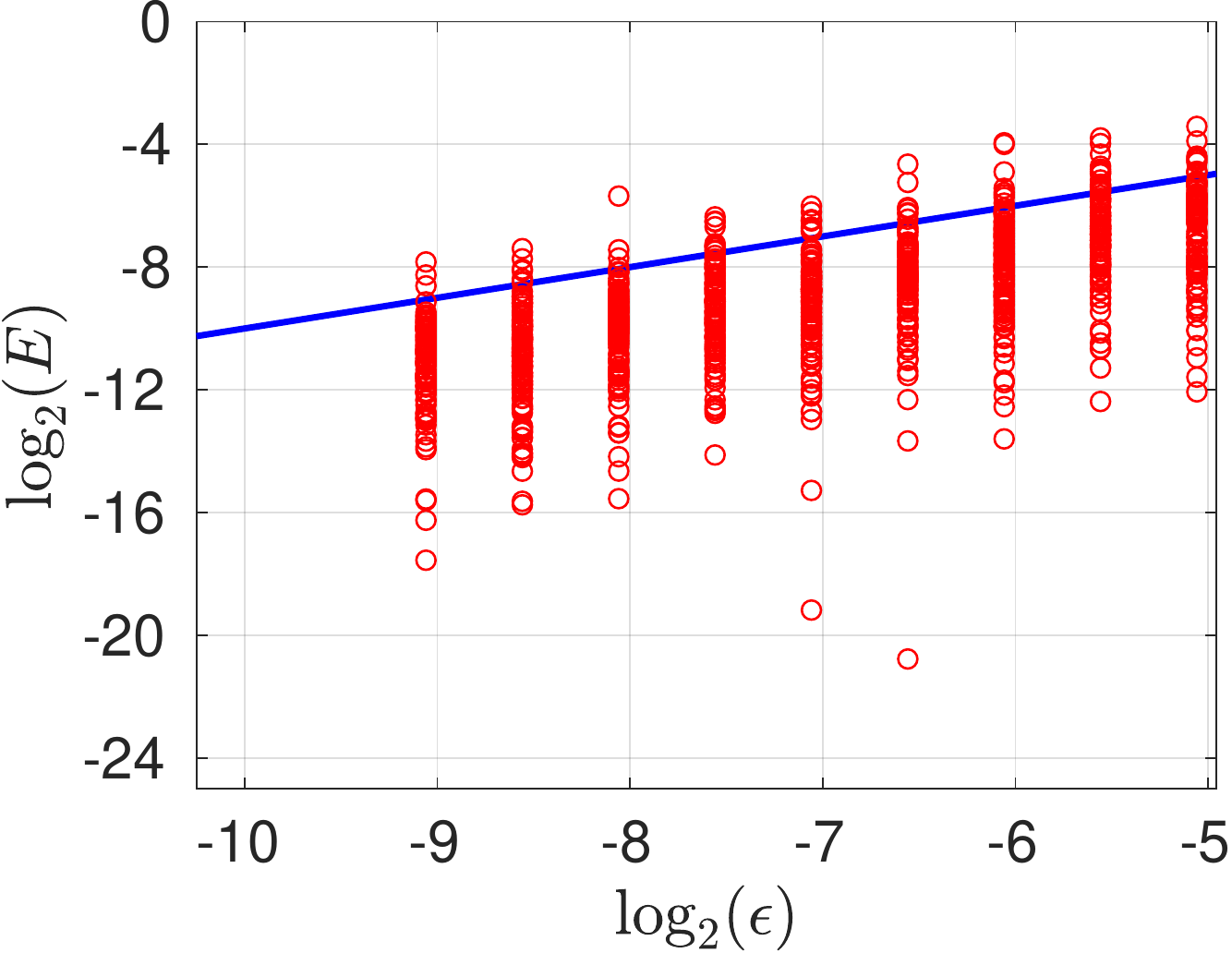} 
\caption[]{\textbf{NDT. Accuracy.} Errors of the expected values of the filter 
    distributions compared to reference solution for tolerances $\epsilon_{k}$, 
    $k=\lbrace 0,\dots,8\rbrace$, in~\eqref{eq:tol}. The error $E$ is estimated 
    by $E_{k,i}$ for $i\in \lbrace 1,\dots,100\rbrace$, defined in~\eqref{eq:E_k_i}, 
    and shown as red circles. 
    On the left the results using Algorithm~\ref{alg:coupled} are displayed, and 
    on the right those using Algorithm~\ref{alg:CDFcoupling}.}   
\label{fig:ndt_tolvserror}
\end{figure}

\begin{figure}[H]
\centering  
  \includegraphics[width=0.55\textwidth]{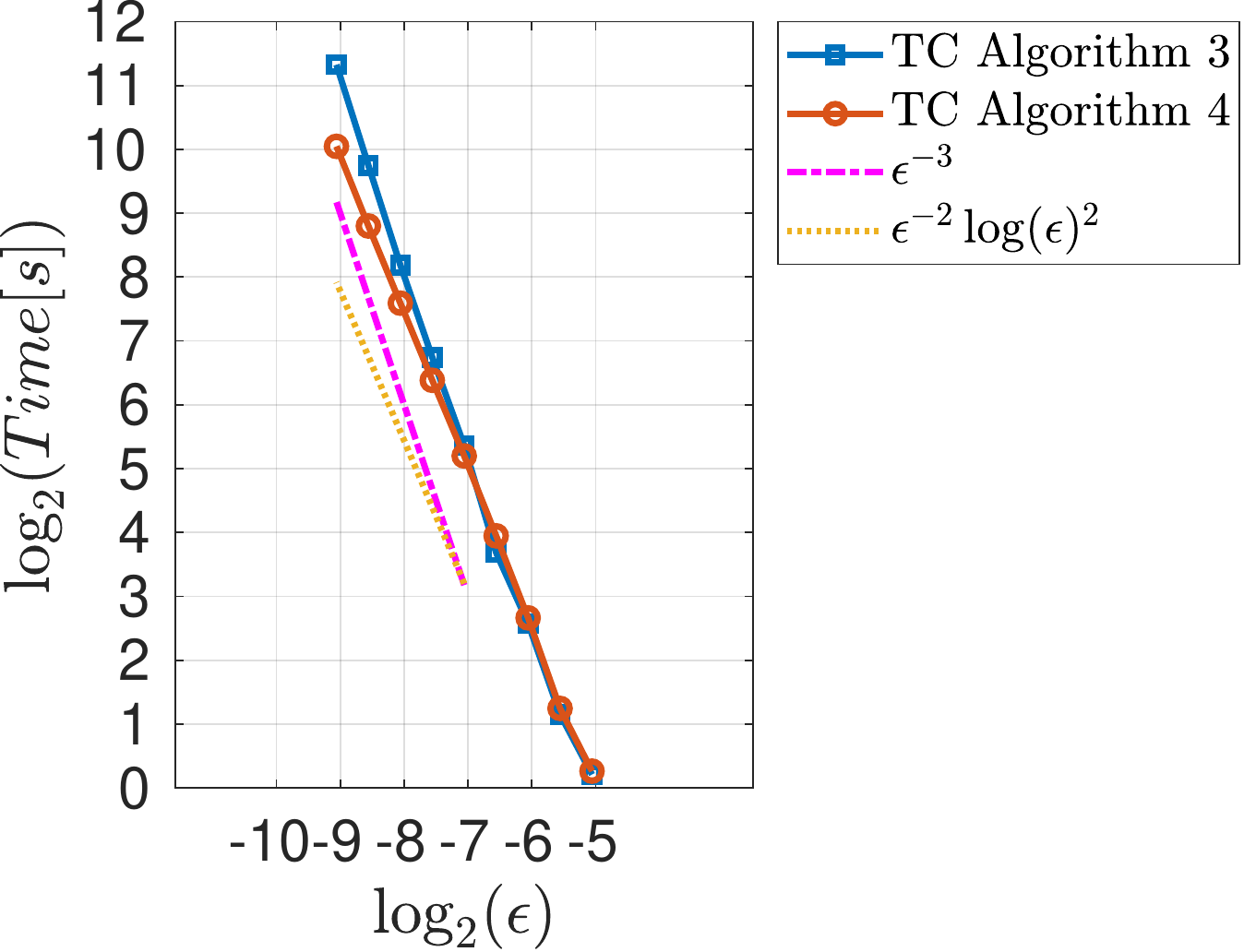}   
\caption[]{\textbf{NDT. Complexity.} Total computational time of the optimal 
  filter, measured as wall clock time in seconds, as a function of the tolerance 
  $\epsilon$.} 
\label{fig:ndt_actualwork}
\end{figure}

\begin{figure}[H]
\centering
  \includegraphics[width=0.45\textwidth]{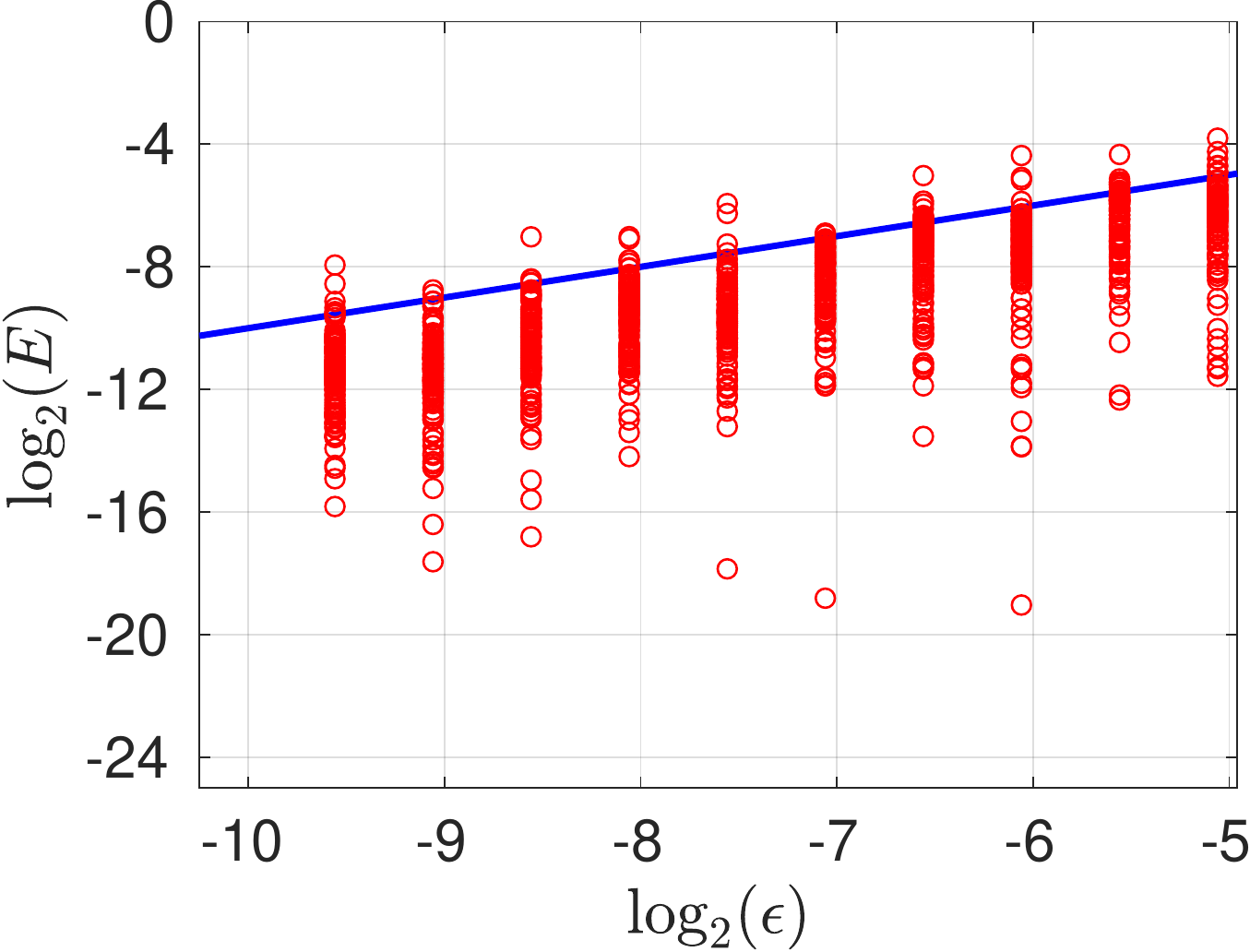} 
\hfill
  \includegraphics[width=0.45\textwidth]{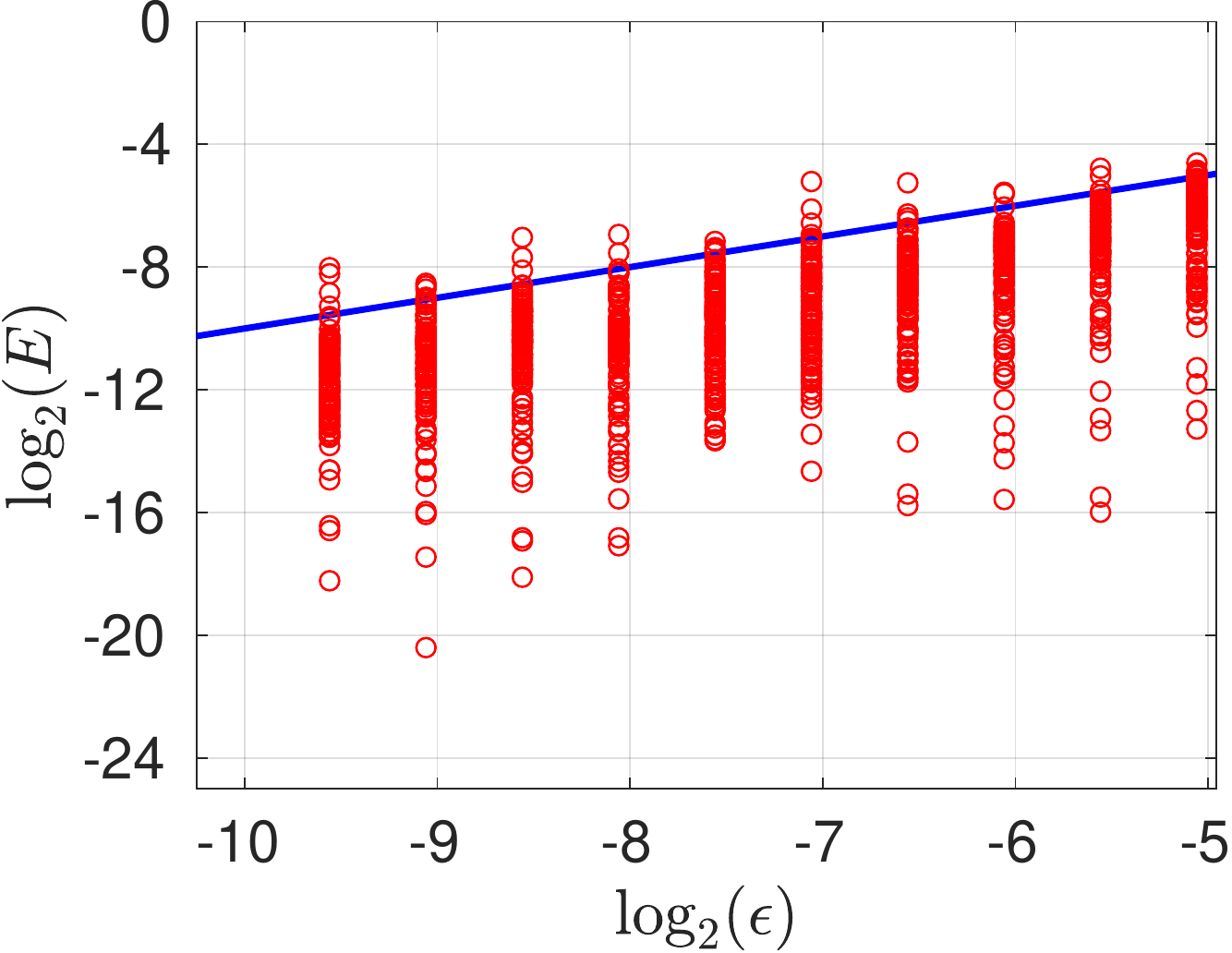} 
\caption[]{\textbf{DW. Accuracy.} Errors of the expected values of the filter 
    distributions compared to reference solution for tolerances $\epsilon_{k}$, 
    $k=\lbrace 0,\dots,9\rbrace$, in~\eqref{eq:tol}. The error $E$ is estimated 
    by $E_{k,i}$ for $i\in \lbrace 1,\dots,100\rbrace$, defined in~\eqref{eq:E_k_i}, 
    and shown as red circles. 
    On the left the results using Algorithm~\ref{alg:coupled} are displayed, and 
    on the right those using Algorithm~\ref{alg:CDFcoupling}. Based on the 
    dynamics~\eqref{eq:DW_dynamics} with the change of measure in Appendix~\ref{app:changeofmeasure}.}  
  \label{fig:dw_tolvserror}
\end{figure}

\begin{figure}[H]
\centering  
  \includegraphics[width=0.55\textwidth]{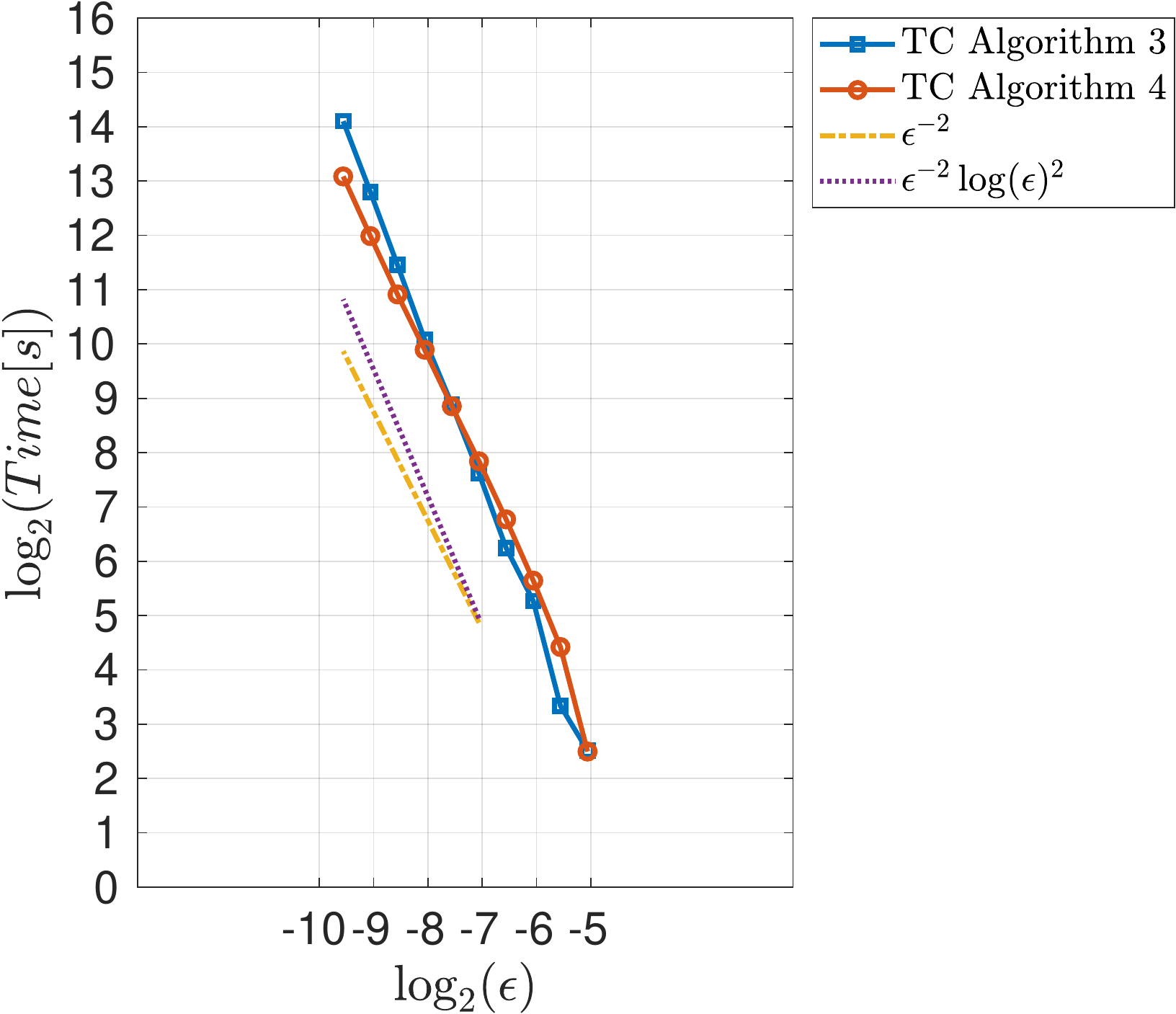}   
\caption[]{\textbf{DW. Complexity.} Total computational time of the optimal 
  filter, measured as wall clock time in seconds, as a function of the tolerance 
  $\epsilon$. Based on the dynamics~\eqref{eq:DW_dynamics} with the change of 
  measure in Appendix~\ref{app:changeofmeasure}.}
  \label{fig:dw_actualwork}
\end{figure}

\newpage

\section{Tables}

\begin{table}[H]

\centering
  \begin{subtable}[t]{\linewidth}
  {\small
    \begin{tabular}{c|c|c|c|c|c|c|c|c|c}
      \textbf{Level} & \textbf{$\epsilon_{1}$} & \textbf{$\epsilon_{2}$} & \textbf{$\epsilon_{3}$} & \textbf{$\epsilon_{4}$} & \textbf{$\epsilon_{5}$} & \textbf{$\epsilon_{6}$} & \textbf{$\epsilon_{7}$} & \textbf{$\epsilon_{8}$} & \textbf{$\epsilon_{9}$}\\ 
      \hline
      0 & 2738 & 6406 & 12641  & 25657 & 57268   & 114513 & 233370 & 467776 & 988411 \\  % <--
      1 & 595 & 1391& 2745     & 5571  &  12436  &  24866 & 50674  & 101573 & 214623  \\ % <--
      2 & 260 & 608 & 1198     & 2432  &  5427   &  10852 & 22115  & 44328  & 93664\\ % <--
      3 & - & -  & 546         & 1107  &  2471   &  4940  & 10067  & 20179  & 42637 \\ % <-- 
      4 & - & -  & -           & 530   &  1182   &  2363  & 4815   & 9651   & 20391\\ % <--            
      5 & - & -  & -           & -     &  -      &  1143  & 2328   & 4666   & 9858\\ % <--            
      6 & - & -  & -           & -     &  -      &  -     & 1132   & 2269   & 4794 \\ % <--            
      7 & - & -  & -           & -     &  -      &  -     & -      & 1130   & 2386 \\ % <--      
    \end{tabular}}
  \caption{Algorithm \ref{alg:coupled}.}  
  \end{subtable}
  
  \begin{subtable}[t]{\linewidth}
  {\small
    \begin{tabular}{c|c|c|c|c|c|c|c|c|c}
      \textbf{Level} & \textbf{$\epsilon_{1}$} & \textbf{$\epsilon_{2}$} & \textbf{$\epsilon_{3}$} & \textbf{$\epsilon_{4}$} & \textbf{$\epsilon_{5}$} & \textbf{$\epsilon_{6}$} & \textbf{$\epsilon_{7}$} & \textbf{$\epsilon_{8}$} & \textbf{$\epsilon_{9}$}\\ 
      \hline
      0 & 2342 & 4608 & 8785        & 17637 & 35262   & 70539 & 141076 & 282121   & 564166 \\  % <--
      1 & 491  & 966  & 1842        & 3698  & 7393    & 14788 & 29576  &  59145   & 118274   \\ % <--
      2 & 180  & 353  & 674         & 1352  &  2702   & 5405  & 10810  &  21617   & 43227     \\ % <--
      3 & -    & 127  & 241         & 484   &  966    & 1932  &  3864 &  7727    & 15451     \\ % <-- 
      4 & -    & -    & 87          & 174   &  347    &  694  &  1387  &  2773    & 5544     \\ % <--            
      5 & -    & -    & 31          & 62    &  124    & 248   &   496  &  992     & 1984\\ % <--            
      6 & -    & -    & -           & 23    &  45     &  89   &   177  &  354     & 707 \\ % <--            
      7 & -    & -    & -           & -     &  16     &  32   & 64     &  128     & 256 \\ % <--            
      8 & -    & -    & -           & -     &  -      &  12   & 23     &  46      & 92\\ % <--            
      9 & -    & -    & -           & -     &  -      &  -    &   9    &  17      & 33  \\ % <--            
     10 & -    & -    & -           & -     &  -      &  -    &  -     &   6      & 12\\ % <--            
     11 & -    & -    & -           & -     &  -      &  -    &   -    &   -      & 5  \\ % <--            
    \end{tabular}}
  \caption{Algorithm \ref{alg:CDFcoupling}.}  
  \end{subtable}
  \caption{\textbf{OU. Optimal hierarchies.}}  
\label{tab:ou_hierarchy}
\end{table}

\begin{table}[H]
\centering
  \begin{subtable}[t]{\linewidth}
  \small{
    \begin{tabular}{c|c|c|c|c|c|c|c|c|c}
      \textbf{Level} & \textbf{$\epsilon_{1}$} & \textbf{$\epsilon_{2}$} & \textbf{$\epsilon_{3}$} & \textbf{$\epsilon_{4}$} & \textbf{$\epsilon_{5}$} & \textbf{$\epsilon_{6}$} & \textbf{$\epsilon_{7}$} & \textbf{$\epsilon_{8}$} & \textbf{$\epsilon_{9}$}\\ 
      \hline
      0 & -     & -     &  -     & -      & -      &  -      &  -       &  -      & -       \\  % <--
      1 & -     & -     &  -     & -      & -      &  -      &  -       &  -      & -       \\ % <--
      2 & 1711  & 4201  &  -     &  -     &   -    &  -      &  -       &   -     & -       \\ % <--
      3 & -     & -     & 7051   & 17613  & 37738  & 96551   & 198022   & 425308  & 1050464 \\ % <-- 
      4 & -     & -     & -      &        & 6038   & 15447   & 31681    &  68043  & 168058  \\ % <--            
      5 & -     & -     & -      & -      &  -     &  -      & 17871    &  38382  & 94800   \\ % <--            
      6 & -     & -     & -      & -      &  -     &  -      &  -       &  22018  & 54382   \\ % <--            
    \end{tabular}}
  \caption{Algorithm \ref{alg:coupled}.}  
  \end{subtable}
  
  \begin{subtable}[t]{\linewidth}
  \small{
    \begin{tabular}{c|c|c|c|c|c|c|c|c|c}
      \textbf{Level} & \textbf{$\epsilon_{1}$} & \textbf{$\epsilon_{2}$} & \textbf{$\epsilon_{3}$} & \textbf{$\epsilon_{4}$} & \textbf{$\epsilon_{5}$} & \textbf{$\epsilon_{6}$} & \textbf{$\epsilon_{7}$} & \textbf{$\epsilon_{8}$} & \textbf{$\epsilon_{9}$}\\ 
      \hline
      0 & -     & -     &  -       & -      & -      &  -       & -       &  -       &   -      \\  % <--
      1 & -     & -     & -        & -      & -      &  -       & -       &  -       &   -      \\ % <--
      2 & 1711  & 4200  &   8797   &  18149 & 37700  &  84783   & 172764  & 344565  & 699524 \\ % <--
      3 & -     &       &   1308   &  2698  & 5604   &  12602   & 25678   & 51213   &  103970 \\ % <-- 
      4 & -     & -     & -        &  1201  & 2493   &  5607    & 11424   &  22784  & 46255   \\ % <--            
      5 & -     & -     & -        &  -     & 1178   &  2648    & 5396    &  10761  & 21845   \\ % <--            
      6 & -     & -     & -        & -      &  -     & -        & 2596    & 5178    & 10511   \\ % <--            
      7 & -     & -     & -        & -      &  -     &  -       &  -      &  2556   &  5189   \\ % <--            
      8 & -     & -     & -        & -      &  -     &   -      &   -     &  -       &  2690   \\ % <--                    
    \end{tabular}}
  \caption{Algorithm \ref{alg:CDFcoupling}.}  
  \end{subtable}
  \caption{\textbf{NDT. Optimal hierarchies.}}  
\label{tab:ndt_hierarchy}  
\end{table}

\begin{table}[H]
\centering
  \begin{subtable}[t]{\linewidth}
  \footnotesize{
    \begin{tabular}{c|c|c|c|c|c|c|c|c|c|c}
      \textbf{Level} & \textbf{$\epsilon_{1}$} & \textbf{$\epsilon_{2}$} & \textbf{$\epsilon_{3}$} & \textbf{$\epsilon_{4}$} & \textbf{$\epsilon_{5}$} & \textbf{$\epsilon_{6}$} & \textbf{$\epsilon_{7}$} & \textbf{$\epsilon_{8}$} & \textbf{$\epsilon_{9}$} & $\epsilon_{10}$ \\ 
      \hline
      0 & 1510  & 4125  &   8747 & 23185  & 45223 &  91421  &  215880  &  421851 & 854945  &  1890536  \\  % <--
      1 & -     & -     &  107   & 2835   & 5529  &  11177  &  26393   &  51574  & 104522  &  231128   \\ % <--
      2 & -     & -     &  -     &  -     &  2496 &  5046   &  11914   &   23281 & 47183   &  104335   \\ % <--
      3 & -     & -     &  -     & -      &   -   &  2590   &  6116    &  11950  &  24218  & 53553     \\ % <-- 
      4 & -     & -     & -      &  -     &  -    &   -     &  -       &   5631  & 11411   & 25232     \\ % <--            
      5 & -     & -     & -      & -      &  -    &  -      &     -    &  -      & 5493    &  12146    \\ % <--                   
    \end{tabular}}
  \caption{Algorithm \ref{alg:coupled}.}  
  \end{subtable}
  
  \begin{subtable}[t]{\linewidth}
  \footnotesize{
    \begin{tabular}{c|c|c|c|c|c|c|c|c|c|c}
      \textbf{Level} & \textbf{$\epsilon_{1}$} & \textbf{$\epsilon_{2}$} & \textbf{$\epsilon_{3}$} & \textbf{$\epsilon_{4}$} & \textbf{$\epsilon_{5}$} & \textbf{$\epsilon_{6}$} & \textbf{$\epsilon_{7}$} & \textbf{$\epsilon_{8}$} & \textbf{$\epsilon_{9}$} & $\epsilon_{10}$ \\ 
      \hline
      0 & 1457 & 2853 & 5650     &  11125 & 22013 &  43684  &  86886  &   173102 &   345275 &  689266\\  % <--
      1 & -     &  158 &  313   &  617  &  1219 &  2419  &  4812  &  9586  &   19120   &  38169 \\ % <--
      2 & -      & -      &    101  &  198   &  391 &   775   &  1541  &   3070   &   6124      &12224\\ % <--
      3 & -     &  -      & -        &  68    &   134 &    266   &  528   &  1051    &  2095     &4182\\ % <-- 
      4 & -     & -     & -        &   -     &    44  &   87    &  173   &  344   &   686  & 1369\\ % <--            
      5 & -     & -     & -        &  -     &  -      &   29     &   58    & 114     &  227    &454 \\ % <--            
      6 & -     & -     & -        & -      &  -     & -        &   19   &  38      &  76        &150 \\ % <--            
      7 & -     & -     & -        & -      &  -     &  -       &  -      &  13      &  25   & 50\\ % <--            
      8 & -     & -     & -        & -      &  -     &   -      &   -     &  -       &   9      & 17 \\ % <--                    
      9 & -     & -     & -        & -      &  -     &   -      &   -     &  -       &   -    &  6\\ % <--                    
    \end{tabular}}
  \caption{Algorithm \ref{alg:CDFcoupling}.}  
  \end{subtable}
  \caption{\textbf{DW. Optimal hierarchies.} Based on the dynamics~\eqref{eq:DW_dynamics} 
  with the change of measure. Figure~\eqref{fig:dw_work} shows that the cost is 
  not proportional to the number of particles, $N$, in a practially significant 
  range of values of $N$, in contrast to the modeling assumption. 
	To avoid complications resulting from this deviation between 
  modeled and observed work, we have for Algorithm~\ref{alg:CDFcoupling} 
  constrained the optimization in the construction of the MLPFs to satisfy 
  $L(\epsilon_{k+1})=L(\epsilon_{k})+1$, which is the theoretically expected 
  rate given the convergence rate of $V_l$. 
}
\label{tab:dw_hierarchy}  
\end{table}

\end{document}